\documentclass{amsart}
\usepackage[dvipsnames]{xcolor}
\usepackage{mathtools}
\usepackage{amssymb}
\usepackage{amsfonts}
\usepackage{palatino}
\usepackage{tikz}
\usepackage{float}
\usepackage{enumerate}
\DeclareFontFamily{U}{mathx}{}
\DeclareFontShape{U}{mathx}{m}{n}{<-> mathx10}{}
\DeclareSymbolFont{mathx}{U}{mathx}{m}{n}
\DeclareMathAccent{\widehat}{0}{mathx}{"70}
\DeclareMathAccent{\widecheck}{0}{mathx}{"71}
\renewcommand{\check}{\widecheck}

\usepackage{hyperref}
\hypersetup{
    unicode=false,          		% non-Latin characters in Acrobat bookmarks
    pdftoolbar=true,        		% show Acrobat's toolbar?
    pdfmenubar=true,        		% show Acrobat's menu?
    pdffitwindow=false,     		% window fit to page when opened
    pdfstartview={FitH},    		% fits the width of the page to the window
    linktocpage=true,			% link to page not section title in TOC
    pdfnewwindow=true,      		% links in new PDF window
    colorlinks=true,       		% false: boxed links; true: colored links
    linkcolor=blue,          	% color of internal links (change box color with linkbordercolor)
    citecolor=PineGreen,    		% color of links to bibliography
    filecolor=magenta,      		% color of file links
    urlcolor=cyan           		% color of external links
}

\theoremstyle{plain}
\newtheorem{theorem}{Theorem}[section]

\newtheorem{lemma}[theorem]{Lemma}
\newtheorem{proposition}[theorem]{Proposition}

\theoremstyle{definition}

\theoremstyle{remark}
\newtheorem{remark}[theorem]{Remark}

\DeclareMathOperator{\imag}{Im}
\DeclareMathOperator{\real}{Re}

\DeclareMathOperator{\Ai}{Ai}

\allowdisplaybreaks
\numberwithin{equation}{section}

\title[Large-Time Asymptotics for KPI]{Large-Time Asymptotics for the Kadomtsev-Petviashvili I Equation}
\author{Samir Donmazov}
\author{Jiaqi Liu}
\author{Peter Perry}
\address[Donmazov]{Department of Mathematics, University of Kentucky, Lexington, Kentucky, U.S.A.}
\email{\url{S.Donmazov@uky.edu}}

\address[Liu]{School of Mathematics, University of Chinese Academy of Sciences. No.\ 19 Yuquan Road, Beijing China}
\email{\url{jqliu@ucas.ac.cn}}

\address[Perry]{Department of Mathematics, University of Kentucky, Lexington, Kentucky, U.S.A}
\email{\url{peter.perry@uky.edu}}

\date{\today}

\begin{document}

\begin{abstract}
    We prove large-time asymptotics for solutions of the KP I equation with small initial data. Our assumptions on the initial data rule out lump solutions but enable us to give a precise description of the radiation field at large times. Our analysis uses the inverse scattering method and involves large-time asymptotics for solutions to a non-local Riemann-Hilbert problem.
\end{abstract}

\maketitle

\tableofcontents

%%%%%%%%%%%%%%%%%%%%%%%%%%%%%%%%
% New command for inserted text
% so this can be turned off for final draft
%%%%%%%%%%%%%%%%%%%%%%%%%%%%%%%%

%\newcommand{\textblue}[1]{{\color{blue} #1}}
\newcommand{\textblue}{}

%%%%%%%%%%%%%%%%%%%%%%%%%%%%%%%%

\newcommand{\calC}{\mathcal{C}}
\newcommand{\calR}{\mathcal{R}}
\newcommand{\calS}{\mathcal{S}}
\newcommand{\calT}{\mathcal{T}}
\newcommand{\R}{\mathbb{R}}
\newcommand{\C}{\mathbb{C}}
\newcommand{\wT}{\widetilde{T}}
\newcommand{\wS}{\widetilde{S}}

\newcommand{\oT}{\overline{T}}

\newcommand{\kbar}{\overline{k}}

\newcommand{\wf}{\widetilde{f}}
\newcommand{\tu}{\widetilde{u}}
\newcommand{\tmu}{\widetilde{\mu}}
\newcommand{\wlam}{\widehat{\lambda}}

\newcommand{\loc}{{\mathrm{loc}}}
\newcommand{\dotarg}{\, \cdot \,}

\newcommand{\eps}{\varepsilon}

\newcommand{\bfn}{\mathbf{n}}
\newcommand{\bfu}{\mathbf{u}}

\newcommand{\calF}{\mathcal{F}}
\newcommand{\calL}{\mathcal{L}}

\newcommand{\diff}{\partial}

\newcommand{\Norm}[2]{\left\Vert #1 \right\Vert_{#2}}
\newcommand{\norm}[2][ ]{\left\Vert #2 \right\Vert_{#1}}
\newcommand{\bigO}[2][ ]{\mathcal{O}_{#1}\left( #2 \right)}
\newcommand{\smallO}[2][ ]{o_{#1}\left( #2 \right)}

\newcommand{\comment}[1]{\marginpar{\scriptsize{\color{purple} #1}}}

%\listoftodos \relax

\section{Introduction}
The goal of this paper is to determine the large-time asymptotic behavior of solutions to the Kadomtsev-Petviashvili I equation on $\R^2$ with small initial data using inverse scattering methods. The Kadomtsev-Petviashvili (KP) equation is the dispersive equation
\begin{equation}
	\label{KP}
		\left\{
		\begin{aligned}
		 \left(u_t + u_{xxx} + 6uu_x\right)_x &= 3\lambda u_{yy} \\
		 u(0,x,y) &= u_0(x,y).
		\end{aligned}
		\right.
		%}
\end{equation}
for a function $u=u(t,x,y)$ describing long weakly dispersive waves in two space dimensions. The case $\lambda=1$ is the KP I equation, and the case $\lambda=-1$ is the KP II equation. The KP equation is integrable in the sense that a smooth function $u=u(t,x,y)$ solves the KP equation if and only if $u$ satisfies the comptability condition for the system
	\begin{align}
		\label{KP.Lax.L}
		\left(\sigma \diff_y + \diff_x^2 + u\right)\psi &= 0 \\
		\label{KP.Lax.M}
			\left(\diff_t + 4 \diff_x^3 + 6u \diff_x + 3\left(u_x - \sigma \diff_x^{-1} \diff_y u \right)\right)\psi &= 0
	\end{align}
		(where $\sigma=i$ for KP I and $\sigma=1$ for KP II)	.  The linear problem \eqref{KP.Lax.L}  defines ``scattering data'' which, by \eqref{KP.Lax.M} and the compatibility condition, evolves linearly under the flow defined by the KP equation \eqref{KP}. Inverse scattering for the KP I equation was studied by Zakharov and Manakov \cite{ZM1979}, Manakov \cite{Manakov81}, and Fokas-Ablowitz \cite{FA83-1}. A rigorous analysis of the inverse scattering method for KP I was given by Zhou \cite{Zhou90} and was further studied by Fokas-Sung \cite{FS92,FS99} and Sung \cite{Sung99}.
		
	To solve \eqref{KP} by inverse scattering, one computes time-zero scattering data from \eqref{KP.Lax.L} with $u=u_0$, evolves the scattering data linearly in time, and recovers $u(t,x,y)$ via an ``inverse scattering map'' from the time-evolved scattering data to the solution.

%% Note to PAP: Maybe add soliton resolution results for
%% * KDV by Deift, Venakides, and Zhou 1994
%% * mKdV - Chen and Liu ***
%% Modulation techniques for energy critica wave equation
%% Calogero -Moser DNLS by XXX
\textblue{One of the strengths of the completely integrable method is its potential to analyze long-term behavior of solutions. For completely integrable systems in one space dimension, the method of nonlinear steepest descent developed by Deift and Zhou \cite{DZ93}, based on the formulation of the inverse scattering map as a matrix Riemann-Hilbert problem,  facilitated a complete analysis of asymptotic behavior for the modified Korteweg-de Vries (mKdV) equation \cite{DZ93}, the Korteweg-De Vries (KdV) equation \cite{DVZ94}, and the defocusing cubic nonlinear Schr\"{o}dinger (NLS) equation \cite{DZ03}. 

The inverse scattering method has also led to proofs of the soliton resolution conjecture, which asserts that, at large times, the solution of a dispersive PDE separates into solitons and radiation. Equations for which soliton resolution holds include the KdV equation (Eckhaus-Schuur \cite{ES:1983}), the mKdV equation (Chen and Liu \cite{CL:2021}), the cubic NLS equation (Borghese, Jenkins, McLaughlin \cite{BJM18} equation, and the derivative nonlinear Schr\"{o}dinger (dNLS) equation (Jenkins, Liu, Perry, and Sulem \cite{JLPS18}). 

One might expect that  rigorous use of the inverse scattering method for the KP I and KP II equations would produce similar insights into asymptotic behavior of these models. As we will explain, our paper constitutes a first step in this direction for KP I. For results on KP II, we refer to the paper of Kiselev \cite{Kiselev01}, which we discuss in what follows, and more recent work of Derchyi Wu \cite{Wu:2026} on large-time asymptotics for solutions of the  KP II equation for small initial data.

The KP I and KP II equations have some similarities but also some important differences. Both take the form \eqref{KP} with linear term
\begin{equation}
    \label{KP.lin.pre}
        \calL = \diff_t + \diff_x^3 - 3\lambda \diff_x^{-1} \diff_{yy},
\end{equation}
and, from the asymptotics of the linear equation $\calL v = 0$, one might expect $\bigO{t^{-1}}$ decay of radiation in $L^\infty$ norm. However the KP II equation has no soliton-like solutions (see deBouard and Saut \cite{DS:1997}), while the KP I has well-studied ``lump'' solutions--solutions which translate in time and have algebraic spatial decay (see, for example, Villarroel and Ablowitz \cite{VA:1999} and references therein, and see \cite[Chapter 5.7]{KS2021} for a comprehensive survey of results up to 2021). It is known that there is a small-data threshold for lump solutions to occur (see Segur \cite{Segur82}), and our assumptions on the initial data will rule out lump solutions. What we can hope to recover is the $\bigO{t^{-1}}$ decay of solutions for such small data.

The KP I equation differs from most of its counterparts in one space dimension in ways that make the inverse scattering approach to large-time asymptotics substantially more challenging. Unlike the inverse scattering maps for KdV, mKdV, NLS, and dNLS, the inverse scattering map for KP I is defined by a \emph{nonlocal} Riemann-Hilbert problem for which there is, at present, no ``nonlinear method of steepest descent' and not even a well-developed theory comparable to the comprehensive literature on matrix Riemann-Hilbert problems. In this paper, we obtain--for the first time, to our knowledge--sharp asymptotic estimates on solutions to the nonlocal Riemann-Hilbert problem for KP I that enable us, via the reconstruction formula by inverse scattering, to obtain good asymptotics of the solutions to KP I under our hypotheses on initial data.
}
	
The KP I and KP II equations have been extensively studied by PDE methods: we refer the reader to the chapter 5 of the monograph of Klein and Saut \cite{KS2021} for a comprehensive discussion of results on the KP equation via both PDE and inverse scattering approaches. We will \textblue{largely} confine our remarks on the PDE literature here to  global well-posedness results that we will use in our analysis of KP I and to preceding work on large-time asymptotics for KP I.
	
%%%%

Ionescu, Kenig, and Tataru \cite[Theorem 1.1]{IKT2008} showed that the KP I equation is globally well-posed in the space
\begin{equation}
    \label{IKT:E1}
        \mathbf{E}_1(\R^2) = \{ \varphi: \R^2 \to \R: \norm[\mathbf{E}_1(\R^2)]{\varphi} < \infty \}
\end{equation}
where
\begin{equation}
    \label{IKT.E1.norm}
    \norm[\mathbf{E}_1]{\varphi } \coloneqq 
        \norm[L^2(\R^2)]{\varphi} + \norm[L^2(\R^2)]{\diff_x \varphi} + \norm[L^2(\R^2)]{\diff_x^{-1}\diff_y \varphi }.
\end{equation}
We will work in a subspace $\mathbf{E}_{1,w}$ of $\mathbf{E}_1$ with 
 more stringent decay and regularity hypotheses.  We define
\begin{equation}
	\label{IEw.space}
		\mathbf{E}_{1,w} = \{ u \in L^2(\R^2): \norm[\mathbf{E}_{1,w}]{u} < \infty \}
\end{equation}
where 
\begin{align}
	\label{IEw.norm}
	\norm[\mathbf{E}_{1,w}]{u}
        &=  \norm[\mathbf{H}^{4,5}(\R^2)]{u} + \norm[L^2(\R^2)]{\diff_x^{-1} u } + \norm[L^2_x L^{2,1}_y(\R^2) ]{\diff_x^{-1} \diff_y u }.
\end{align}
Here
$$ \mathbf{H}^{4,5}(\R^2) = \left\{ u \in L^2(\R^2): \norm[\mathbf{H}^{4,5}(\R^2)]{u} < \infty  \right\}$$
where
\begin{multline*}
   \norm[\mathbf{H}^{4,5}(\R^2)]{u} = \\\norm[L^{2,4}_x L^{2,5}_y(\R^2)]{ u} + \norm[L^2(\R^2)]{(1+y^2)^2(1-\diff_x^2)^2 u} + \norm[L^2(\R^2)]{(1-\diff_y)^2 u } 
\end{multline*}
Note that
\begin{equation}
	\label{IEw.toIE}
		\norm[\mathbf{E}_{1}]{u} \lesssim \norm[\mathbf{E}_{1,w}]{u}
\end{equation}
so that $\mathbf{E}_{1,w}$ is continuously embedded in $\mathbf{E}_1$. Hence, the KP I equation with initial data in $\mathbf{E}_{1,w}$ is still globally well-posed in $\mathbf{E}_1$. The stronger norms in  \eqref{IEw.norm} are needed to obtain estimates on the scattering data for $u$ and the solutions of the nonlocal Riemann-Hilbert problem--see particularly Proposition \ref{prop: mu-k} in \S \ref{sec:mupm}, Propositions \ref{prop:Tpm.L2} and \ref{prop:Tpm.kl.L2} in \S \ref{sec:Tpm}, and Lemma \ref{rem:Tpm.kl.L2} in \S \ref{App:Norm}.

A number of authors have studied asymptotics for the KP equations by PDE methods. Building on previous analysis of the generalized KP equation due to Hayashi, Naumkin and Saut \cite{HNS99},  Hayashi and Naumkin \cite{HN14} study solutions of KP I and KP II with $\diff_x^{-1}u_0 \in H^7 \cap H^{5,4}$ and $\norm[H^{5,4}]{\diff_x^{-1} u_0}$ small. They show that there is a unique solution $u$ with $\norm[L^\infty]{\diff_x u} \lesssim (1+t)^{-1}$, establish leading asymptotics for $u_x(t,x,y)$, and find a scattering asymptote for $u_x$ in $L^\infty$.  Harrop-Griffiths, Ifrim, and Tataru \cite{HID17} study solutions of KP I in a natural weighted space with Galilean-invariant norms: they obtain a similar $L^\infty$ bound on $u_x$ and also prove scattering of $u$ in $L^2$ for small data. Their approach is motivated in part by an analysis of the Hamiltonian flow associated to \eqref{KP} and its behavior in the ``propagation region'' which is essentially the region $a<0$ described in our main result, Theorem \ref{thm:main}, stated below (see their discussion in \S 1 of \cite{HID17}). Mendez, M\~{u}noz, Poblete and Pozo \cite{MMPP24} study long-time asymptotics for large data. \textblue{More recently, Alegria, Chen, Mu\'{n}oz, Poblete, and Tardy \cite{ACMPT:2025} have studied soliton structures for the KP II equation. They study classes of solutions of KP II of the form $u(x,y,t) = 2 \diff_x^2 (F(\Theta))$, where $F:[s_0,\infty) \to \R$ is a profile and $\Theta$ is a phase. They establish an equivalence between various types of soliton solutions and corresponding classes of amplitudes and phases. } 

So far as we are aware, however, there are no results in the PDE literature on the $L^\infty$ behavior of $u$ itself, \textblue{even for small initial data,} which show $t^{-1}$ decay in time.
	
On the other hand, large-time asymptotics for solutions of both the KP I and KP II equation have also been studied via inverse scattering. Kiselev \cite{Kiselev01} derived leading asymptotics for the KP II equation under regularity and small data assumptions on the time zero scattering data for KP II. Manakov, Santini and Takhtajan \cite{MST79} derived long-time asymptotics for solutions of the KP I with smooth initial data. As we will explain (and as commented by Kiselev in his review paper \cite{Kiselev04}), the results of Manakov, Santini and Takhtajan are confined to only one of several space-time r\'{e}gimes for KP I. 
	
	To understand the dynamics of KP I and KP II, we first consider the linear Cauchy problem obtained from \textblue{either} KP equation by removing the nonlinear term (recall \eqref{KP.lin.pre}):
    \textblue{
	\begin{equation}
		\label{KP.lin}
        \left\{
		\begin{aligned}
            \calL v     &= 0\\
            v(0,x,y) 	&= v_0(x,y).
		\end{aligned}\
		\right.
	\end{equation}
    }
	Introduce ``slow'' variables 
	\begin{equation}
		\label{xi.eta}
			\zeta = \frac{x}{t}, \quad \eta=\frac{y}{t}. 	
	\end{equation}
	The Cauchy problem \eqref{KP.lin} admits the explicit solution
	\begin{align}
		\label{KP.lin.sol}
		v(t,x,y) &= 
			\frac{1}{2\pi}
				\int 
					e^{it(\zeta p + \eta q + (p^3 + 3\lambda q^2/p))}
				\widehat{v_0}(p,q)
				\, dp \, dq
	\end{align}
    For the linear KP I equation, we may make the change of variables $p=l-k,q=-(l^2-k^2)$
	%where the change of variables $p=l-k,q=-(l^2-k^2)$ will be useful in comparing  
    in order to compare the solution of \eqref{KP.lin} with the solution of \eqref{KP} by inverse scattering. We find
    \begin{align}
        \label{KPI.lin.sol}
        v(t,x,y) &=
			\frac{1}{\pi} 
				\int
					e^{itS_0(k,l;\zeta,\eta)}
					\widehat{v}_0(l-k,-(l^2-k^2)) |l-k|\, dl \, dk
			\intertext{where the phase function is}
            \label{S0.def.intro}
		S_0(k,l;\zeta,\eta) &= 	\zeta(l-k)-\eta(l^2-k^2) + 4(l^3-k^3)
    \end{align}
	A similar analysis for the linear KP II equation with the change of variables $p=-(k+\kbar)$, $q=i(k^2-\kbar^2)$ shows that
    \begin{align}
        \label{KPII.lin.sol}
        v(t,x,y) &=
            \frac{1}{\pi} \int_{\C} e^{itS(k,\kbar;\zeta,\eta)}
                i(k+\kbar)\widehat{v}_0(-(k+\kbar),i(k^2-\kbar^2)) dk \wedge d\kbar 
    \end{align}
    where \textblue{the phase function is}
    $$ S(k,\kbar; \zeta,\eta) -i(k+\kbar)\zeta -i(k^2-\kbar^2)\eta + 4(k^3 + \kbar^3).$$
    These respective phase functions have:
	\begin{enumerate}[(i)]
		\item	Nondegenerate critical points if 	$-\zeta + \lambda \eta^2/12 >0$, a region where we might expect $t^{-1}$ decay of the solution,
		\item	No critical points if $-\zeta + \lambda \eta^2/12 < 0$, where we might expect rapid decay of the solution as $t \to \infty$, and
		\item	Degenerate critical points if $-\zeta + \lambda \eta^2/12 =0$, where we might expect a transition from $\bigO{t^{-1}}$ behavior to rapid decay of the solution. 
	\end{enumerate}
	Thus, in the $(\zeta,\eta)$ plane, there are three distinct regimes of asymptotic behavior, shown here for $\lambda =1$ (the linear KP I equation):
	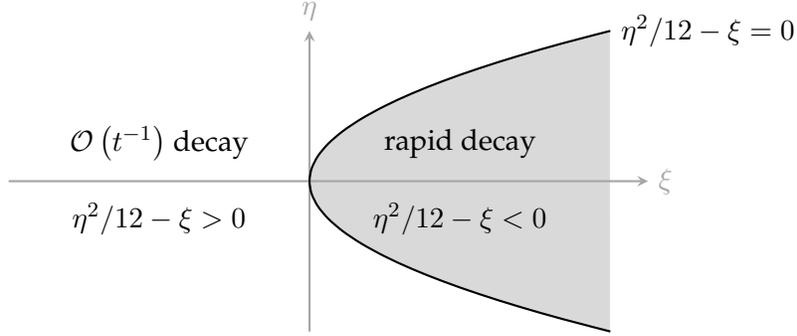
\begin{figure}[H]
		\caption{Large-Time Decay Regions for the Linear KP I Equation}
		\bigskip
		\begin{tikzpicture}
		
			\draw[white,fill=gray!30]
				plot[domain=-2:2,variable=\eta,smooth]
					({\eta*\eta},{\eta}) -- (4,-2);
					
			\draw[thick,gray!70,->,>=stealth]
				(-4,0) -- (4.5,0) 
				node[anchor=west] {$\zeta$};
			\draw[thick,gray!70,->,>=stealth]
				(0,-2) -- (0,2)
				node[anchor=south] {$\eta$};
				
			\draw[thick,black]
				plot[domain=-2:2,variable=\eta,smooth]
					({\eta*\eta},{\eta})
					node[anchor=west] {$\eta^2/12 - \zeta=0$};
			
			\node at (2,0.5) {rapid decay};
			\node at (-2,0.5) {$\bigO{t^{-1}}$ decay};
			
			\node at (2,-0.5) {$\eta^2/12 - \zeta < 0$};
			\node at (-2,-0.5) {$\eta^2/12 - \zeta > 0$};
		\end{tikzpicture}
	\end{figure}

	A similar trichotomy of asymptotic behaviors occurs in the analysis of the KP I and KP II equations by inverse scattering methods. Kiselev \cite{Kiselev01} (see also \cite[\S 3.3]{Kiselev04} studied solutions of the KP II equation with regularity and smallness conditions on the scattering data. For the KP II equation, setting
		$$ a = -\eta^2 -12\xi, \quad \text{(KP II)}$$
		Kiselev proved that if $u(t,x,y)$ is a solution to KP II satisfying the regularity and smallness conditions, then $u$ has the following asymptotic behaviors:
		\begin{enumerate}[(i)]
		\item	For $a t^{\frac13} \gg 1$, setting $a=r^2$,  
		$$			u(t,x,y) \underset{t \to \infty}{\sim}
				-2t^{-1} 
				\real 
					\left[
						e^{-11itr}
						\frac{\pi}{12ir} 
						f\left(
							\frac{r}{2} + \frac{i\eta}{12} 
						\right) 
					\right]
				+ \smallO{t^{-1}}
		$$
		\item	For $a t^{\frac23} \ll -1$, 
		$$ u(t,x,y) = \smallO{t^{-1}}.$$
		\item	For $|a| \ll 1$, 
		\begin{align*}
			u(t,x,y) 
				&= 8it^{-1} \sqrt{\pi} f(i\eta/12) F(z) + \smallO{t^{-1}}
		\intertext{where}
			z	&=	8t^\frac23
					\left(
						\frac{\eta^2}{12}+\zeta
					\right),\\
		\intertext{and}
			F(z)	&= 
				\int_0^\infty 
					\sqrt{p_1}\cos(8p_1^3-zp_1) 
				\, dp_1 +
				\int_0^\infty 
					\sqrt{p_2}\sin(8p_1^3-zp_2) 
				\, dp_2
		\end{align*}
	\end{enumerate}
	Here $f:\C \to \C$ is scattering data associated to the initial condition. 
	
	On the other hand, Manakov, Santini and Takhtajan \cite{MST79} studied asymptotics and scattering for solutions of the KP I equation under small data and smoothess assumptions. In this case, let
	\begin{equation}
		\label{KPI.a.def.intro}
			a = \frac{1}{12} \left(\zeta - \frac{\eta^2}{12}\right) \quad \text{(KP I)}.
	\end{equation}
	Formulating the inverse scattering theory for KP I using the Gelfand-Levitan approach (see Zakharov-Shabat \cite{ZS:1974}, Zakharov-Manakov \cite{ZM1979}, and Manakov \cite{Manakov81}), these authors computed long-time asymptotics of $u(t,x,y)$ in the region $a<0$. Using stationary phase methods and an ansatz for the Gelfand-Levitan integral equation, these authors derived an asymptotic form
	\begin{equation}
		\label{KPI.MST}	
			u(t,x,y) \underset{t \to \pm \infty}{\sim}
		\frac{8r}{t}\real \left(f^\pm(\zeta,\eta) e^{16itr^3} \right) 
	\end{equation}
	where $a=-r^2 < 0$  and $f^\pm$ are determined from the asymptotics of solutions to the Gelfand-Levitan equation. 
	
	Our goals in this paper are, first of all, to place the results of \cite{MST79} on a rigorous footing and, secondly, to obtain asymptotics in the other spatial regions. Rather than pursuing the Gelfand-Levitan method used in \cite{MST79}, we will use Zhou's \cite{Zhou90} formulation of inverse scattering via a nonlocal Riemann-Hilbert problem. To our knowledge, ours is the first result which obtains asymptotics for an integrable dispersive PDE through a rigorous analysis of a nonlocal Riemann-Hilbert problem.
	
	To state our results, we outline some of the main ideas of Zhou's theory (see also the papers of Fokas and Sung \cite{FS99} and Sung \cite{Sung99} where the nonlocal Riemann-Hilbert problem is used to construct solutions of KP I by the inverse scattering method). Denote by $u(x,y)$ the initial data for \eqref{KP} with $\lambda=1$. Let $\tu$ be the partial Fourier transform
	\begin{equation}
		\label{tu.def}
			\tu(l;y) = (2\pi)^{-\frac12} 
				\int e^{-ilx} u(x,y) \, dx
	\end{equation} 
	and assume that $\tu \in L^1(\R^2)$ with
	\begin{equation}
		\label{u.C1}
			c \coloneqq \frac{\norm[L^1(\R^2)]{\tu}}{\sqrt{2\pi}} < 1	
	\end{equation}
(this is a small-data condition that rules out soliton solutions and implies the existence of scattering solutions analytic in $\C \setminus \R$ needed to define the scattering data: see Segur \cite{Segur82} and Zhou \cite{Zhou90}). The direct scattering map $\calS$ takes $u$ to scattering data $T^\pm(k,l)$ where $T^\pm(k,l)=0$ for $\pm(l-k)<0$. The linearization of $\calS$ at zero potential is
\begin{equation}
	\label{KPI.S.lin}
		\calS_0(u)(k,l) = -\frac{i}{2\pi} \widehat{u}(l-k,-(l^2-k^2)).
\end{equation}

For the scattering map to have range in $L^2(\R^2)$, a necessary condition is that the linearization \eqref{KPI.S.lin} have range in $L^2(\R^2)$. The right-hand side of \eqref{KPI.S.lin} lies in $L^2(\R^2)$ provided that
\begin{equation}
	\label{u.C2.ancest}
		\norm[L^2(\R^2,|p|^{-1} \, dp \, dq)]{\widehat{u}} < \infty 
\end{equation}
as may easily be computed using the change of variables $p=l-k$, $q=-(l^2-k^2)$ to compute the $L^2$-norm of $\calS_0(u)$. By the Plancherel theorem,  condition \eqref{u.C2.ancest} is equivalent to
\begin{equation}
	\label{u.C2.pre}
		\norm[L^2_y L^{2,-1}_l]{\tu} < \infty
\end{equation}
where, here and in what follows,
\begin{equation}
	\label{L2yl.def}
	\norm[L^2_y L^{2,-1}_l]{f}
		=	\left( \int |f(l,y)|^2 |l|^{-1} \, dy \, dl \right)^\frac12.
\end{equation}
To obtain estimates on the scattering data we will also assume
\begin{equation}
	\label{u.C2}
	\norm[L^2_y L^{2,-1}_l]{\tu} < \frac{1-c}{4},
\end{equation}
and
\begin{equation}
    \label{u.C2a}
        \widetilde{c} \coloneqq \frac{\norm[L^{1,1}_l L^1_y]{\tu}}{\sqrt{2\pi} } < 1,
\end{equation}
(for \eqref{u.C2}, see Lemma \ref{lemma:tmu.exist} and the discussion following \eqref{gku.delta.X}; for \eqref{u.C2a}, see \eqref{tilde.c.cond} and the accompanying discussion).
Finally, we will assume \begin{equation}
	\label{u.C3}	
	u \in \mathbf{E}_{1,w}(\R^2)
\end{equation}
where the space $\mathbf{E}_{1,w}$ was defined in \eqref{IEw.space} and \eqref{IEw.norm}.
Zhou's reconstruction formula for $u(t,x,y)$ (see \cite[eq.\ (4.27)]{Zhou90})\footnote{In our formulas, we have included the time evolution of the scattering data and the time-dependence of the solution to the nonlocal Riemann-Hilbert problem. The scattering data evolves as
$$T^\pm(k,l,t)=T^\pm(k,l) e^{4it(l^3-k^3)}$$} involves 
the scattering data $T^\pm$ together with the solution $\mu^l=\mu^l(k,x;y,t)$ of a nonlocal Riemann Hilbert problem determined by the scattering data $T^\pm(k,l)$. 
In what follows we will sometimes make the change of variables
\begin{equation}
	\label{kl.shift.pre}
		(k,l) \to \left( \frac{\eta}{12} + k, \frac{\eta}{12} + l \right)	,
\end{equation}
and write
\begin{equation}
	\label{wT.def.pre}
		\wT^\pm(k,l) = T^\pm \left(k+\frac{\eta}{12}, l+\frac{\eta}{12} \right)	
\end{equation}
to shift critical points of the phase function to $(\pm r, \pm r)$. 

Let
\begin{equation}
\label{de:f(k,l)}
    f(k,l) = T^+(k,l) + T^-(k,l). 
\end{equation}
The reconstruction formula for $u(t,x,y)$ is
\begin{align}
	\label{u.recon.pre}
		u(t,x,y) 	&=	u_1(t,x,y) + u_2(t,x,y)
	\intertext{where}
	\label{u1.recon}
		u_1(t,x,y)	&= 	
			\frac{1}{\pi} 
				\int e^{itS_0(k,l;\zeta,\eta)} 
					i(l-k)	f(k,l) \, dl \, dk,
	\intertext{and}
	\label{u2.recon}
		u_2(t,x,y)	&=	\frac{1}{\pi}
			\int e^{itS_0(k,l;\zeta,\eta)}
				i(l-k) f(k,l)
				\left( \mu^l(l,x;y,t) - 1 \right)
				\, dl \, dk \\
			&\quad + 	\frac{1}{\pi}
			\int e^{itS_0(k,l;\zeta,\eta)}
				\frac{\diff \mu^l}{\diff x}(l,x;y,t)
				\, dl \, dk.
			\nonumber
\end{align}
In \eqref{u2.recon}, the function $\mu^l$ is the solution of a nonlocal Riemann-Hilbert problem determined by the (time-evolved) scattering data.

We will refer to $u_1$ as the ``local'' term in the reconstruction formula and $u_2$ as the ``nonlocal'' term in the reconstruction formula. The nonlocal Riemann-Hilbert problem for $\mu^l$ is the integral equation
\begin{equation}
	\label{mul.RHP.pre}
		\mu^l	= 1	+	\calC_T \mu^l	
\end{equation}
for $\mu^l(\dotarg,x;y,t) -1 \in L^2_k(\R^2)$,
where
\begin{align}
	\label{CT.def.pre}
	\calC_T 		
		&= C_+ \calT^- + C_-\calT^+,\\
	\label{CTpm.def}
	(\calT^\pm f)(k)	&=	\int e^{itS_0(k,l;\zeta,\eta)}
							T^\pm(k,l) f(l) \, dl,
\end{align}
and $C_\pm: L^2_k(\R) \to L^2_k(\R)$ are Cauchy projections. An important part of our analysis is determining long-time asymptotic behavior of $\mu^l(k,x;y,t)$ as $t \to \infty$. We study the nonlocal Riemann-Hilbert problem in detail in \S \ref{subsec:NLRHP.est}.

Recall that, for the KP I equation, we define
$$a = \frac{1}{12} \left(\zeta - \frac{\eta^2}{12}\right)$$
where $\zeta=x/t$, $\eta=y/t$. We will prove:

%\todo[size=\tiny]{Check statement as changed by PAP, particularly the $|a|<\delta$ result}
\begin{theorem}
	\label{thm:main}
	Fix \textblue{$0<\delta\le1$}. Suppose that the Cauchy data for the problem
	\begin{equation}
		\label{KPI}
		\left\{
			\begin{aligned}
			u_t + 6uu_x + u_{xxx} &= 3\diff_x^{-1} u_{yy},\\
				u(0,x,y) &= u_0(x,y)
				\end{aligned}
			\right.
	\end{equation}
	satisfy conditions \eqref{u.C1}, \eqref{u.C2}, \eqref{u.C2a}, and \eqref{u.C3}. The following asymptotics hold \textblue{pointwise} as $t \to \infty$:
    
            \begin{equation}
                \label{KPI.u.asy}
                u(t,x,y) \underset{t \to \infty}{\sim}
                \begin{cases}
                    \smallO{t^{-1}},    &   a>\delta > 0,\\
                    \\
                    \bigO{t^{-\frac23}}, & |a| \le \delta, \\
                    \\
                    \bigO{t^{-1}}, &a < -\delta < 0.
                \end{cases} 
            \end{equation}
            \textblue{where the implied constants for the cases $|a|>\delta$ diverge as $\delta \downarrow 0$, while the constant in the transition case $|a| \le \delta$ is uniform as $\delta \downarrow 0$.}
	\end{theorem}

\begin{proof}
	A direct consequence of Theorems \ref{thm:u1} and \ref{thm:u2} in \S \ref{sec:asy}.	
\end{proof}

\begin{remark}
    For the case of $a < -\delta < 0$, we can obtain an asymptotic formula for $u_1(t,x,y)$ that parallels the result of Manakov, Santini, and Takhtajan. We are not currently able to obtain explicit leading asymptotics for $u_2(t,x,y)$ in this case, so our best result overall is the $\bigO{t^{-1}}$ estimate.
\end{remark}

%\begin{remark}
%		We expect that the $a \sim 0$ bounds can be improved to $\bigO{t^{-1}}$ but we have not yet been able to do this. 
%\end{remark}
	
Here is a sketch of the contents of this paper. \textblue{Let $u$ denote the initial data for KP I}. In \S \ref{sec:mupm}, we study the map $u \mapsto \mu^\pm$  and obtain estimates on the solutions $\mu^\pm$ and their derivatives needed to study the scattering data $T^\pm$. In \S \ref{sec:NLRHP}, we review Zhou's construction of the scattering data and his formulation of the nonlocal Riemann-Hilbert problem that lies at the heart of our analysis. 
%We also obtain time-decay estimates on solutions of the nonlocal RHP that play a crucial role in our asymptotic analysis later. 
In \S \ref{sec:Tpm}, we study the direct scattering map $u \mapsto T^\pm$ and obtain estimates on $T^\pm$ and its derivatives that are needed to solve the nonlocal Riemann-Hilbert problem and derive long-time asymptotic behavior of solutions. 

Finally, in \S \ref{sec:asy}, we give the proof of Theorem \ref{thm:main},  \textblue{using the reconstruction formula \eqref{u.recon.pre}.} The local term \eqref{u1.recon} is analyzed in \S \ref{subsec:local} (see Theorem \ref{thm:u1}), the nonlocal term \eqref{u2.recon} is analyzed in  \S \ref{subsec:nonlocal} (see Theorem \ref{thm:u2}), and the estimates on $\mu^l$ needed in the analysis of the nonlocal term are proved in \S \ref{subsec:NLRHP.est} (see Proposition \ref{prop:RHP.bounds}). 

In Appendix A, we give estimates on oscillatory integrals needed in \S \ref{subsec:local}. In Appendix B, we show that the condition $u \in \mathbf{E}_{1,w}$ for the initial data implies estimates on weighted norms of $u$, $\tu$, and their derivatives which are needed throughout the paper.

\subsection{Acknowledgements}

We are grateful to Oleg Kiselev for helpful correspondence. Jiaqi Liu thanks the University of Kentucky for hospitality during part of the time this work was done.   Peter Perry thanks the Department of Mathematical Sciences at Tsinghua University for hospitality during the concluding stages of our work. This work was supported by a grant from the Simons Foundation (359431,  PAP) and a grant from the Chinese Academy of Science (E4ER0101A2, JL). \textblue{Finally, we are grateful to the referee for a very careful reading of our manuscript, and for a number of helpful suggestions which have improved the quality of this paper.}

\subsection{Notation}

We normalize the Fourier transform of functions $f:\R^n \to \C$ as follows:
\begin{align*}
	\widehat{f}(\xi) &= (2\pi)^{-n/2} \int_{\R^n} f(x) e^{-i\xi \cdot x} \, dx,\\
	\widecheck{g}(x) &= (2\pi)^{-n/2} \int_{\R^n} g(\xi) e^{i\xi \cdot x} \, d\xi.	
\end{align*}

We will use weighted spaces and Sobolev spaces in one and two variables. We denote by $L^{p,s}$ the space of measurable functions of one variable with
	$$ \norm[L^{p,s}]{f} \coloneqq \left(\int (1+x^2)^{sp/2} |f(x)|^p \, dx \right)^\frac12 $$
	finite, and by $H^{s}$ the space of functions of one variable with
	$$ \norm[H^s]{f} 
		\coloneqq 
			\left( 
				\int 
					(1+|\xi|^2)^s |\widehat{f}(\xi)|^2 
				\, d\xi 
			\right)^\frac12$$ 
	finite.
	Notations such as $L^{p,s}_x L^{q,\sigma}_y$ refer to the space of measurable functions of two variables with
	$$ \norm[L^{p,s}_x L^{q,\sigma}_y]{f} 
		\coloneqq 
			\left(
				\int (1+s^2)^{sp/2} 
					\left| 
						\int 
							(1+y^2)^{\sigma q/2}|f(x,y)|^q 
						\, dy 
					\right|^{p/q} 
				\, dx 
			\right)^{1/p} $$
	finite. We reserve the special notation $L^2_y L^{2,-1}_l$ for the space of measurable functions $g=g(y,l)$ with
	$$\norm[L^2_y L^{2,-1}_l]{g} \coloneqq \left( \int |g(y,l)|^2 |l|^{-1} \, dl \, dy \right)^\frac12 $$
	finite. 
	
\medskip

We will also use the following notation.

\medskip

\begin{itemize}
    \setlength{\itemsep}{5pt}
    %%%%%
    
    \item[$\mu^\pm(k,x;y)$] Analytic scattering solutions associated to initial data $u_0$ (see \eqref{mupm.eqn})
    \item[$\tmu^\pm(k,l;y)$] Distribution Fourier transforms of $\mu^\pm$ in the $x$ variable (see \eqref{tmupm.def})
    \item[$\tmu^\pm_\sharp(k,l;y)$] The function $\tmu^\pm - \sqrt{2\pi}\delta$    (see \eqref{tmupm.sharp.def})
    \item[$\mu^l(k,x;y,t)$] Solution to the nonlocal Riemann-Hilbert problem (see \eqref{mul.RHP.pre})
    
    %%%%%

    \item[$\zeta,\eta$]    Scaled spatial coordinates (see \eqref{xi.eta})

    \item[$a$] The quantity $\zeta/12 - \eta^2/144$ (see \eqref{KPI.a.def.intro} and  \eqref{a.def})
    \item[$c$]		The ratio $\norm[L^1]{\tu}/\sqrt{2\pi}$ (see \eqref{c.def})
    \item[$\tilde{c}$] The ratio $\norm[L^{1,1}_l L^1_y]{\widetilde{u}}/{\sqrt{2\pi}}$ (see \eqref{u.C2a})
    \item[$\tu$] The partial Fourier transform of $u$ in the $x$ variable (see \eqref{tu.def})
    
    %%%%%

     \item[$C_\pm$] Cauchy projectors $C_\pm:L^2(\R) \to L^2(\R)$  

     \item[$C_T$] Integral operator for the nonlocal Riemann-Hilbert problem for the ``left'' solution $\mu^l$ at times $t>0$ (see \eqref{mul.RHP.pre} and \eqref{CT.def.pre}; see also \eqref{mul.RHP} and \eqref{CT.def})

     \item[$C_{T_{x,y}}$] Integral operators in the nonlocal Riemann-Hilbert problem for the ``left'' solution $\mu^l$ at time $t=0$ (see \eqref{CT.left.time0})
     
    \item[$C_{R_{x,y}}$] Integral operators in the nonlocal Riemann-Hilbert problem for the ``right'' solution $\mu^r$ at time $t=0$ (see \eqref{CT.right.time0})
    
    \item[$S_0(k,l;\zeta,\eta)$] Phase function for the reconstruction formula (see \eqref{S0.def.intro} and \eqref{S0.def})
    
    \item[$S(k,l;a)$] Phase function for the reconstruction formula in shifted coordinates \eqref{kl.shift.pre} (see \eqref{S.def})

    \item[$T^\pm$] Scattering data for $u_0$ (see \eqref{Tpm})
    \item[$T_1^\pm$] Linear part of scattering data (see \eqref{Tpm.split} and \eqref{T1.def})
    \item[$T^\pm_2$] Nonlinear part of scattering data (see \eqref{Tpm.split} and \eqref{T2.def}
    \item[$\wT^\pm$] Scattering data in shifted coordinates \eqref{kl.shift.pre}  (see \eqref{wT.def.pre} and \eqref{wT.def})

    \item[$X$] The space $L^\infty_{y}(L^2_k L^2_l)$ (see \eqref{space.X})
    \item[$Y$]  The space $ L^\infty_{y,k}(L^1_l)$ (see \eqref{space.Y})

    \item[$\calT^\pm_{x,y}$] Integral operators in the nonlocal Riemann-Hilbert problem at time $t=0$ (see \eqref{calTpm.op.def})
    
    \item[$\calT^\pm$] Integral operators in the nonlocal Riemann-Hilbert problem at  $t=0$  (see \eqref{calTpm.op.def}) or the nonlocal Riemann-Hilbert problem for $t>0$ (see \eqref{CTpm.def} and \eqref{calTpm.def})
    \item[$\displaystyle\frac{\diff \calT^\pm}{\diff x}$]
    Integral operators in the equation for $\diff \mu^l/\diff x$ (see \eqref{dTdX.op})
    \item[$\widetilde{\calT}^\pm$] Integral operators in the $t=0$ nonlocal Riemann-Hilbert problem (see \eqref{caltildeT.op.def})
    \item[$\calR^\pm$] Integral operators in the $t=0$ nonlocal Riemann-Hilbert problem (see \eqref{calRpm.op.def})
    \item[$\widetilde{\calR}^\pm$] Integral operators in the $t=0$ nonlocal Riemann-Hilbert problem (see \eqref{caltildeR.op.def})
    
%    \item[$\mu^\sharp(k,x;y,t)$] The function $\mu^l(k,x;y,t)-1$

%%%

    \item[$u(t,x,y)$]   Solution to KP-I via inverse scattering (see \eqref{u.recon.pre} and \ref{u.recon})
    \item[$u_1(t,x,y)$] ``Local'' part of $u(t,x,y)$ depending only on scattering data (see \eqref{u1.recon} and \eqref{u.local})
    \item[$u_2(t,x,y)$] ``Non-local'' part of $u(t,x,y)$ depending on the scattering data and the solution to the non-local Riemann-Hilbert problem (see \eqref{u2.recon} and \eqref{u.non-local})
    \item[$\mu^l$]  Solution of the non-local Riemann-Hilbert problem \eqref{mul.RHP}
    %\item[$C_T$] Singular integral operator in the nonlocal Riemann-Hilbert problem \eqref{mul.RHP} (see \eqref{CT.def})
   
\end{itemize}
\section{Scattering Solutions}
\label{sec:mupm}

We consider the scattering solutions $\mu^\pm$ of the equation
\begin{equation}
	\label{mupm.eqn}
		\begin{aligned}
			i\mu_y + \mu_{xx} + 2ik\mu_x + u(x,y) \mu = 0,\\
			\lim_{x \to \pm \infty} \mu(k,x;y) = 1
		\end{aligned}
\end{equation}
with analytic extensions to $\pm \imag k > 0$.  We set 
\begin{equation}
	\label{tmupm.def}
		 \tmu^\pm(k,l;y) = \frac{1}{\sqrt{2\pi}}\int e^{-ilx} \mu^\pm(k,x;y) \, dx,
\end{equation}
(distribution Fourier transform) and
\begin{equation}
	\label{tmupm.sharp.def}
		\tmu^\pm_\sharp(k,l;y) = \tmu^\pm(k,l;y) - \sqrt{2\pi} \delta(l).
\end{equation}
As shown in Zhou's analysis \cite[\S 2]{Zhou90}, $\tmu^\pm$ obeys the integral equation
\begin{equation}
	\label{tmupm.eqn}
	\tmu^\pm = \sqrt{2\pi}\delta(l) + g_{k,u}^\pm(\tmu^\pm)
\end{equation}
where
\begin{equation}
	\label{gku.def}
		g_{k,u}^\pm (f)(l;y) = \frac{i}{\sqrt{2\pi}} \int_{\pm l \cdot \infty}^y 
			e^{-il(l+2k)(y-\eta)} (\tu*f)(l;\eta) \, d\eta.
\end{equation}
Here and in what follows,
$$ \int_{\pm l \cdot \infty}^y g(l;\eta) \, d\eta \coloneqq \left( H(l) \int_{\pm \infty}^y + H(-l) \int_{\mp \infty}^y\right) g(l;\eta) \, d\eta $$
where $H(l)=1$ for $l>0$ and $H(l)=0$ for $l<0$. 

All of our estimates in this section will be based on the integral equation
\begin{equation}
	\label{tmupm.sharp.eqn}
		\tmu^\pm_\sharp(k,l;y) = g_{k,u}^\pm (\sqrt{2\pi}\delta) + g_{k,u}^\pm (\tmu^\pm_\sharp)
\end{equation}
which is an immediate consequence of \eqref{tmupm.eqn}.
We will sometimes make a change of variables to obtain
\begin{equation}
	\label{int.l-eta.shift}
		\int_{\pm l \cdot \infty}^y g(l;\eta) \, d\eta =
	\int_{\pm l \cdot \infty}^0 g(l;\eta+y) \, d\eta.
\end{equation}

Let
$$ I(v) = \int_a^b e^{il(l+2k)y} v(l,y) \, dy$$
where $-\infty \leq a < b \leq \infty$. 
In what follows we will often use the bound \cite[p.\ 554, after (C2.14)]{Zhou90}
\begin{equation}
	\label{Iv.bd}
		\norm[L^2(\R^2, dl \, dk)]{I(v)} \leq \pi^\frac12 \norm[L^2(|l|^{-1} \, dl \, dy)]{v}.
\end{equation}
We will use the notation
\begin{equation}
	\label{L2.l.inv}
	\norm[L^2_y L^{2,-1}_l]{v} \coloneqq \norm[L^2(|l|^{-1} \, dl \, dy)]{v}
\end{equation}
for the norm on the right-hand side of \eqref{Iv.bd}.

Let
\begin{align}
	\label{space.X}
		X &= L^\infty_y(L^2(\R^2,\, dl \, dk)),\\
	\label{space.Y}
		Y &= L^\infty_{y,k} (L^2(\R, \, dl)).
\end{align}
It is not difficult to see that
\begin{align}
	\label{gku.bd.X}
	\norm[X \to X]{g_{k,u}^\pm} & \leq \frac{\norm[L^1(\R^2)]{\tu}}{\sqrt{2\pi}},\\	
	\label{gku.bd.Y}
	\norm[Y \to Y]{g_{k,u}^\pm} & \leq \frac{\norm[L^1(\R^2)]{\tu}}{\sqrt{2\pi}},
\end{align}
so that $(I-g_{k,u}^\pm)^{-1}$ is a bounded operator from $X$ to $X$ and also from $Y$ to $Y$ if \eqref{u.C1} holds.
%% got rid of this because (2.12) is now the same
%For later use, we also note that
%\begin{align}
%	\label{gku.bd.L2l}
%	\norm[L^\infty_y L^2_l \to L^\infty_y L^2_l]{g_{k,u}^\pm}
%		&\leq \frac{\norm[L^1(\R^2)]{\tu}}{\sqrt{2\pi}}	
%\end{align}

Recall that
\begin{equation}
	\label{c.def}
		c= \frac{\norm[L^1(\R^2)]{\tu}}{\sqrt{2\pi}}.
\end{equation}
It is easy to estimate
\begin{align}
	\label{gku.delta.X}
	\norm[X]{g_{k,u}^\pm(\sqrt{2\pi}\delta)} &\leq \pi^\frac12 \norm[L^2_y L^{2,-1}_l]{\tu},\\
	\label{gku.delta.Y}
	\norm[Y]{g_{k,u}^\pm(\sqrt{2\pi}\delta)} &\leq \frac{\norm[L^1_y(L^2_l)]{\tu}}{\sqrt{2\pi}},
\end{align}
where we used \eqref{Iv.bd} to obtain \eqref{gku.delta.X}. 
Motivated by \eqref{gku.delta.X}, we will also assume that \eqref{u.C2.ancest} holds.

The following lemma is an immediate consequence of \eqref{tmupm.sharp.eqn}, the resolvent bounds \textblue{\eqref{gku.bd.X}--\eqref{gku.bd.Y}},  \textblue{the bounds} \eqref{gku.delta.X}--\eqref{gku.delta.Y}, and the second resolvent formula.
\begin{lemma}
	\label{lemma:tmu.exist}
	Suppose that \eqref{u.C1} and \eqref{u.C2.ancest} hold. Then
	equation \eqref{tmupm.sharp.eqn} has a unique solution $\tmu^\pm_\sharp \in X \cap Y$. Moreover:
	\begin{enumerate}[(i)]
		\item	The estimates
			\begin{align}
				\label{tmu.est0.X}
					\norm[X]{\tmu^\pm_\sharp} &\leq (1-c)^{-1} \pi^\frac12 \norm[L^2_y L^{2-1}_l]{\tu},
				\intertext{and}
                \label{tmu.est0.Y}
					\norm[Y]{\tmu^\pm_\sharp} &\leq (1-c)^{-1} \frac{\norm[L^1_y(L^2_l)]{\tu}}{\sqrt{2\pi}}
			\end{align}	
			hold.
		\item	Suppose that $u_1,u_2 \in L^1(\R^2) \cap L^2_y L^{2,-1}_l$ with $\norm[L^1(\R^2)]{\tu_1}, \norm[L^1(\R^2)]{\tu_2} < \sqrt{2\pi}$. Then
			\begin{align}
				\label{tmu.cont.X}
					\norm[X]{\tmu^\pm(u_1) - \tmu^{\pm}(u_2)}
						&\lesssim (1-c_1)^{-1} (1-c_2)^{-1} \norm[L^1(\R^2)]{\tu_1-\tu_2},
				\intertext{and}
				\label{tmu.cont.Y}
					\norm[Y]{\tmu^\pm(u_1) - \tmu^\pm(u_2)}
						&\lesssim (1-c_1)^{-1} (1-c_2)^{-1} \norm[L^1_yL^2_l(\R^2)]{\tu_1-\tu_2},
			\end{align}
			where $c_i = \norm[L^1(\R^2)]{\tu_i}/\sqrt{2\pi}$. 
	\end{enumerate}

\end{lemma}

Now we consider smoothness of the solutions $\mu^\pm$ in the $x$ and $y$ variables. For use below we note that
\begin{align}
	\label{gku.l}
		l^m g_{k,u}(f)(l;y)
			&= \sum_{0 \leq 	r \leq m} 
					\binom{m}{r} \, g_{k,u_{m-r,0}}^\pm( l^r f),\\
	\label{gku.y}
		\diff_y^{m'} g_{k,u}(f)(l;y)
			&=	\sum_{0 \leq r' \leq m'} 
					\binom{m'}{r}\, g_{k,u_{0,m'-r'}}^\pm (\diff_y^{r'} f),
\end{align}
where
$$ g_{k,u_{r,r'}}(f) \coloneqq \frac{i}{\sqrt{2\pi}} \int_{\pm l \cdot \infty}^y e^{-il(l+2k))(y-\eta)}((l^r \diff_y^{r'} \tu)*f)(l;\eta) \, d\eta.$$

We will need the bounds
\begin{align}
	\label{gku.rr'.bd}
		\norm[X]{g_{k,u_{r,r'}}^\pm(\sqrt{2\pi} \delta)}	 
			&\lesssim \norm[L^2_y L^{2,-1}_l]{l^r \diff_y^{r'} \tu},\\
	\intertext{and}
	\label{gku.rr'.bd.Y}
		\norm[Y]{g_{k,u_{r,r'}}^\pm(\sqrt{2\pi} \delta)}	 
			&\lesssim \norm[L^1(\R^2)]{l^r \diff_y^{r'} \tu},
\end{align}
which follow from \eqref{gku.delta.X} and \eqref{gku.delta.Y}.

The next lemma establishes regularity of $\mu^\pm$ (see Zhou \cite[Lemma 2.24]{Zhou90}; our proof uses his ideas but establishes some additional estimates).

\begin{lemma}
	\label{lemma:tmu.smooth}
	Suppose that \eqref{u.C1} holds. Fix nonnegative integers $j$ and $j'$. Suppose that for all $m,m'$ with
	$0 \leq m \leq j$, $0 \leq m' \leq j'$,
	\begin{equation}
		\label{u.smooth.X.C1}
			l^m \diff_y^{m'} \tu \in L^1(\R^2) \cap L^2_y L^{2,-1}_l.
	\end{equation}
	Then for all $m,m'$ with $0 \leq m \leq j$, $0 \leq m' \leq j'$, 
	\begin{equation}
		\label{mupm.xy.smooth}
			\frac{\diff^{m+m'}}{\diff x^m \diff y^{m'}}(\mu^\pm-1) \in L^\infty_y(L^2(\R^2, dx \, dk))
	\end{equation}
	with norms depending on $\norm[L^1]{\tu}$ and on the norms $\norm[L^1]{l^r \diff_y^{r'} \tu}$, $\norm[L^2_y L^{2,-1}_l]{l^r \diff_y^{r'} \tu}$ for $0 \leq r \leq m$, $0 \leq r' \leq m'$.  
	
	\end{lemma}
	
	\begin{remark}
	\label{rem:mupm.xy.smooth}
	To prove \eqref{mupm.xy.smooth}, we will actually show that
	\begin{equation}
		\label{mupm.ly.smooth.X}
			l^m \diff_y^{m'} \tmu^\pm_\sharp \in X
	\end{equation}
	and recover the result from the Plancherel theorem.
	\end{remark}

\begin{proof}

	We will show that $l^m \diff_y^{m'}\tmu^\pm \in X$ from which the result follows.
	Let
	$$ \tmu^\pm_{\sharp,r,r'} = l^r \diff_y^{r'} \tmu^\pm_\sharp. $$
	A short computation using \eqref{tmupm.sharp.eqn} shows that 
	\begin{align}
		\label{tmupm.ly.smooth}
			\tmu^\pm_{\sharp,m,m'} &= h_{m,m'} + g_{k,u}^\pm ( \tmu^\pm_{\sharp,m,m'})	,
		\intertext{where}
		\label{tmupm.ly.h}
			h_{m,m'} &= l^m \diff_y^{m'} g_{k,u}(\sqrt{2\pi}\delta) + \sum_{\substack{0 \leq r \leq m\\0 \leq r' \leq m' \\ (r,r') \neq (m,m') }}
				c_{r,r'} g_{k,u_{m-r,m'-r'}}^\pm( l^r \diff_y^{r'} \tmu^\pm_\sharp),
	\end{align}
	and
	$$ c_{r,r'} = \binom{m}{r} \binom{m'}{r'}. $$
	To prove the result, it suffices to show that $h_{m,m'} \in X$ for all $(m,m')$ with
	$0 \leq m \leq j$ and $0 \leq m' \leq j'$. Note that estimation of $h_{m,m'}$ involves bounds on $\tmu^\pm_{\sharp,r,r'}$ for $0 \leq r \leq m$, $0\leq r' \leq m'$, but $(r,r') \neq (m,m')$. A natural approach is by double induction on $(m,m')$.
	
	First, as $h_{0,0} = g_{k,u}^\pm(\sqrt{2\pi}\delta)$, the inclusion $h_{0,0} \in X$ follows from \eqref{gku.delta.X}. The first induction steps are to bound
	\begin{align}
		\label{h.10}
			h_{1,0} 
				&= 	lg_{k,u}^\pm(\sqrt{2\pi}\delta) + 
					g_{k,u_{1,0}}^\pm(\tmu^\pm_\sharp),\\
		\label{h.01}
			h_{0,1} 
				&= \diff_y(g_{k,u}^\pm(\sqrt{2\pi}\delta)) + 
					g_{k,u_{0,1}}^\pm( \tmu^\pm_\sharp),\\
		\label{h.11}
			h_{1,1}
				&= l\diff_y g_{k,u}^\pm(\sqrt{2\pi}\delta) +
					g_{k,u_{1,0}}^\pm(\diff_y \tmu^\pm_\sharp) +
					g_{k,u_{0,1}}^\pm(l \tmu^\pm_\sharp) .
	\end{align}
	It is easy to see that
	\begin{align}
		\label{h.10.bd}
		\norm[X]{h_{1,0}}
			&\lesssim \norm[L^2_y L^{2,-1}_l]{l \tu}
				+ \norm[L^1(\R^2)]{l\tu} \norm[X]{\tmu^\pm_\sharp},\\
		\label{h.01.bd}
		\norm[X]{h_{0,1}}
			&\lesssim \norm[L^2_y L^{2,-1}_l]{\diff_y \tu}
				+ \norm[L^1(\R^2)]{\diff_y \tu } \norm[X]{\tmu^\pm_\sharp},
	\end{align}
	which implies that
	\begin{equation}
		\label{tmupm.10.01.bd}
			\norm[X]{l\tmu^\pm_\sharp} < \infty, \quad \norm[X]{\diff_y \tmu^\pm_\sharp} < \infty.
	\end{equation}
	From \eqref{tmupm.10.01.bd}, \eqref{gku.rr'.bd}, and \eqref{h.11}, it now follows that
	\begin{equation}
		\label{h.11.bd}
			\norm[X]{h_{1,1}} < \infty,
	\end{equation}
	and that $\norm[X]{l \diff_y \tmu^\pm_\sharp} < \infty$.
	
	Now suppose that $\norm[X]{l^r \diff_y^{r'} \tmu^\pm_\sharp}< \infty$ for $0 \leq r \leq m$ and $0 \leq r' \leq m'$. We proceed in a similar way, first bounding $\norm[X]{l^{m+1} \diff_y^{r'} \tmu^\pm_\sharp}$ for $0 \leq r' \leq m$ and $\norm[X]{l^r \diff_y^{m'+1} \tmu^\pm_\sharp}$ for $0 \leq r \leq m$. Finally we bound $\norm[X]{l^{m+1} \diff_y^{m+1} \tmu^\pm_\sharp}$ using these results.
	
	\medskip

\end{proof}
\begin{remark}
\label{rem:lemma tmu smooth}
    \textblue{We will only make use of the following three cases of \eqref{u.smooth.X.C1} in the rest of the paper: (1) $j=1$, $j'=0$; (2) $j=j'=1$; (3) $j=3$, $j'=0$. In particular these three cases are sufficient to deduce \eqref{Tpm.est03.bd2.2.norms} and Proposition \ref{prop:Tpm.L2} which is needed in the proof of Part (i) and (ii) of Proposition \ref{prop:Tpm.kl.L2}.}
\end{remark}

%{\color{red}
Finally, we estimate derivatives of $\tmu^\pm_\sharp$ with respect to $l$ and $k$. Recall $\tmu^\pm$ is the solution of the equation \eqref{tmupm.sharp.eqn} and that
$$ g_{k,u}^\pm (f)(l;y) = \frac{i}{\sqrt{2\pi}} \int_{\pm l \cdot \infty}^y 
			e^{-il(l+2k)(y-\eta)} (\tu*f)(l;\eta) \, d\eta.
 $$

In order to obtain estimates on derivatives of $\tmu^\pm_\sharp$, we first need to obtain better weighted estimates on $\tmu^\pm_\sharp$ itself. We wish to show that, given the small norm condition
\begin{equation}
    \label{tilde.c.cond}
     \widetilde{c}= \frac{\norm[L^{1,1}_l L^1_y]{\tu}}{(2\pi)^{\frac12}} < 1,
\end{equation}
\eqref{tmupm.sharp.eqn} is uniquely solvable with $\tmu^\pm_\sharp \in \widetilde{X}$ where
$$ \widetilde{X}=L^\infty_y(L^2_k L^{2,1}_l).$$ 

To solve equation \eqref{tmupm.sharp.eqn} in $\widetilde{X}$, we need to show that $g_{k,u}^\pm :  \widetilde{X} \to  \widetilde{X}$ with operator norm $<1$ and to show that the inhomogeneous term $g_{k,u}^\pm(\sqrt{2\pi} \delta) \in \widetilde{X}$. We can then apply the solution formula
$$ \tmu^\pm_\sharp = (I- g_{k,u}^\pm)^{-1} (g_{k,u}^\pm ( \sqrt{2\pi} \delta)). $$

First we need a lemma on convolutions in weighted spaces. We focus on the $l$ variable only since this is where we need the weight.

\begin{lemma}
	\label{lemma:convolve}
	Suppose that $f \in L^{2,1}_l(\R)$ and $\tu \in L^{1,1}_l(\R)$. Then $\tu*f \in L^{2,1}_l(\R)$ with
	\begin{equation}
		\label{Young.wt}
			\norm[L^{2,1}_l]{\tu*f}
				\leq \norm[L^{1,1}_l]{\tu}
						\norm[L^{2,1}_l]{f}.
	\end{equation}
\end{lemma}

\begin{proof}
	From the usual Young's inequality, we have
	\begin{equation}
		\label{l.Young.old}
			\norm[L^2]{\tu*f} \leq \norm[L^1_l]{\tu} \norm[L^2]{f}.
	\end{equation}	
	Next, we write
	$$ l (\tu*f)(l) = \int (l-l') \tu(l-l') f(l') \ dl' + \int \tu(l-l') l' f(l') \, dl'.$$
	Applying \eqref{l.Young.old} (but with different choices of $\tu$ and $f$), we get
	\begin{equation}
		\label{l.Young.new}
			\norm[L^2]{l (\tu*f)}
				\leq \norm[L^1]{l 
				\tu} \norm[L^2]{f} +
			\norm[L^1]{\tu} \norm[L^2]{l f}.
	\end{equation}
	We conclude that
	\eqref{Young.wt} holds.
\end{proof}

Next, we solve \eqref{tmupm.sharp.eqn} in the space $\widetilde{X}$.
\begin{lemma}
\label{lemma:Xtilde}
	Suppose that $\widetilde{c} < 1$
	and $\tu \in L^2_y L^{2,-1}_l \cap L^2_y L^2_l$. Then, the equation \eqref{tmupm.sharp.eqn} has a unique solution $\tmu^\pm$ with $\tmu^\pm - \sqrt{2\pi}\delta \in  \widetilde{X}$.	
\end{lemma}

\begin{proof}

First, we estimate
\begin{align*}
	\norm[\widetilde{X}]{g_{k,u}^\pm f}
		&\leq \frac{1}{\sqrt{2\pi}}\int_{-\infty}^\infty	
				\norm[L^2_k L^{2,1}_l]{(\tu*f)(k,\cdot,\eta)} \, d\eta\\
		&\leq \frac{1}{\sqrt{2\pi}}\int_{-\infty}^\infty	
				\norm[L^2_k]{\norm[L^{2,1}_l]{(\tu*f)(k,\cdot,\eta}} \, d\eta\\
		&\leq \frac{1}{\sqrt{2\pi}}\int_{-\infty}^\infty 
				\norm[L^2_k L^{2,1}_l]{f}
				\norm[L^{1,1}_l]{\tu(\cdot;\eta)}
				\, d\eta\\
		&\leq	\frac{1}{\sqrt{2\pi}}\norm[L^1_y L^{1,1}_l]{\tu}
				\norm[\widetilde{X}]{f},
\end{align*}
from which it follows that
\begin{equation}
	\label{gku.op}
	\norm[\widetilde{X} \to \widetilde{X}]{g_{k,u}^\pm} \leq \frac{1}{\sqrt{2\pi}}\norm[L^{1,1}_l L^1_y]{\tu}	
\end{equation}
so that $\norm[\widetilde{X} \to \widetilde{X}]{g_{k,u}^\pm} < 1$ provided
\begin{equation}
	\label{tu.cond}
	\norm[L^{1,1}_l L^1_y]{\tu} < \sqrt{2\pi}.
\end{equation}
We then have 
\begin{equation}
	\label{gku.res.bd}
		\norm[\widetilde{X} \to \widetilde{X}]{(I-g_{k,u}^\pm)^{-1}} \leq \frac{1}{1-\widetilde{c}}
\end{equation}
provided $0 <  \widetilde{c} < 1$. 
Next, we compute
\begin{align*}
	g_{k,u}^\pm (\sqrt{2\pi} \delta)
		&=	\int_{\pm l \cdot \infty}^y
				e^{-il(l+2k)(y-\eta)} \tu(l;\eta) 
			\, d\eta,
	\intertext{and}
	l g_{k,u}^\pm (\sqrt{2\pi} \delta)
		&=	\int_{\pm l \cdot \infty}^y
				e^{-il(l+2k)(y-\eta)} l\tu(l;\eta) 
			\, d\eta.
\end{align*}
We conclude that
\begin{equation}
	\label{gku.vec}
	\norm[ \widetilde{X}]{g_{k,u}^\pm (\sqrt{2\pi} \delta)}
		\leq \norm[L^2_y L^{2,-1}_l]{\tu} +
			\norm[L^2_y L^2_l]{\tu}.	
\end{equation}
From \eqref{gku.res.bd} and \eqref{gku.vec} we conclude the proof.
\end{proof}

\begin{lemma}
\label{lemma:tmu y}
	Suppose that $\widetilde{c} < 1$
	and $\dfrac{\diff }{\diff y}\tu \in L^2_y L^{2,-1}_l \cap L^{2,1}_y L^{2,1}_l$. Then, the equation \eqref{tmupm.sharp.eqn} has a unique solution $\tmu^\pm$ with 
 $$\frac{\diff \tmu^\pm_\sharp}{\diff y}=\frac{\diff }{\diff y}(\tmu^\pm - \sqrt{2\pi}\delta) \in L^\infty_y(L^2_k L^{2}_l) .$$
 \end{lemma}
 \begin{proof}
Direct calculation gives
    \begin{equation}
	\label{tmu.sharp.int.y}
	\frac{\diff }{\diff y}\tmu^\pm_\sharp =\frac{\diff }{\diff y} g_{k,u}^\pm (\sqrt{2\pi} \delta) +\left[\frac{\diff }{\diff y}g_{k,u}^\pm \right] (\tmu^\pm_\sharp) + g_{k,u}^\pm \left(\frac{\diff }{\diff y}\tmu^\pm_\sharp \right).
\end{equation}
     We first make a translation:
     \begin{equation}
      g_{k,u}^\pm(f)(l;y) =-
		\frac{1}{\sqrt{2\pi}} \int_{\pm l \cdot \infty}^0 e^{-il(l+2k)(\eta)} (\tu*f)(l;y+\eta) \, d\eta,   
     \end{equation}
     so that, in particular, 
     \begin{align*}
         g_{k,u}^\pm (\sqrt{2\pi} \delta)
		&=	\int_{\pm l \cdot \infty}^0
				e^{-il(l+2k)\eta} \tu(l;y+\eta) 
			\, d\eta,\\
   \frac{\diff }{\diff y}  g_{k,u}^\pm (\sqrt{2\pi} \delta)&=\int_{\pm l \cdot \infty}^0
				e^{-il(l+2k)\eta}\frac{\diff }{\diff y} \tu(l;y+\eta) 
			\, d\eta.
     \end{align*}  
     We then deduce that 
     \begin{align*}
         \norm[{X}]{\frac{\diff }{\diff y}  g_{k,u}^\pm (\sqrt{2\pi} \delta)}\leq \norm[L^2_y L^{2,-1}_l]{\frac{\diff }{\diff y} \tu }.
     \end{align*}
     For the other inhomogeneous term in \eqref{tmu.sharp.int.y}, 
\begin{align*}
    \left[\frac{\diff }{\diff y}g_{k,u}^\pm \right] (\tmu^\pm_\sharp)(k,l;y) =-
		\frac{1}{\sqrt{2\pi}} \int_{\pm l \cdot \infty}^0	
			e^{-il(l+2k)\eta} \left(\frac{\diff }{\diff y}\tu*\tmu^\pm_\sharp\right)(l;\eta) \, d\eta.
\end{align*}
We then estimate

\begin{align*}
	\norm[{X}]{ \left[\frac{\diff }{\diff y}g_{k,u}^\pm \right] (\tmu^\pm_\sharp)}
		&\leq \int_{-\infty}^\infty	
				\norm[L^2_k L^{2}_l]{\left(\frac{\diff }{\diff y}\tu*\tmu^\pm_\sharp\right)(k,\cdot,\eta)} \, d\eta\\
		&\leq \int_{-\infty}^\infty	
				\norm[L^2_k]{\norm[L^{2}_l]{\left(\frac{\diff }{\diff y}\tu*\tmu^\pm_\sharp\right)(k,\cdot,\eta)}} \, d\eta\\
		&\leq \int_{-\infty}^\infty 
				\norm[L^2_k L^{2}_l]{\tmu^\pm_\sharp}
				\norm[L^{2,1}_l]{\frac{\diff }{\diff y}\tu(\cdot;\eta)}
				\, d\eta\\
		&\leq	\norm[L^{2,1}_y L^{2,1}_l]{\frac{\diff }{\diff y}\tu}
				\norm[{X}]{\tmu^\pm_\sharp}.
\end{align*}

The rest of the proof is an application of the resolvent bound
\eqref{gku.res.bd}.

\end{proof}

For the next result, we need to study the operator and resolvent of $g_{k,u}^\pm$ viewed as maps from $L^2_l$ to itself. First we have, uniformly in $y,k$, the estimate
\begin{equation}
	\label{gku.ll.bd}
		\norm[L^2_l \to L^2_l]{g_{k,u}^\pm}
			\lesssim \frac{1}{\sqrt{2\pi}}\norm[L^1]{\tu}	
\end{equation}
as easily follows from \eqref{gku.def}, Minkowski's inequality for integrals, and Young's inequality.

\begin{proposition}
   \label{prop: mu-k}
Suppose that $\widetilde{c}<1$ and 
 $u\in \mathbf{E}_{1,w}(\R^2)$. Then
\begin{equation}
\label{est: mu-k}
  \Norm{ \partial_k{\tmu_\sharp}^\pm }{L^\infty_k L^2_l}\lesssim 1+|y|,
\end{equation}
and
\begin{equation}
\label{est: mu-k-l}
   \Norm{l\partial_k{\tmu_\sharp}^\pm}{L^\infty_k L^2_l} \lesssim_k 1+|y|.
\end{equation}
Moreover,
\begin{equation}
\label{est: mu-k L2}
  \Norm{ \partial_k{\tmu_\sharp}^\pm }{L^2_k L^2_l}\lesssim 1+|y|,
\end{equation}
and 
\begin{equation}
    \label{est: mu-k  L2lk}
        \Norm{ l \partial_k{\tmu_\sharp}^\pm }{L^2_k L^2_l}\lesssim 1+|y|.
\end{equation}
\end{proposition}
\begin{proof}
 We study the following integral equation that is an easy consequence of \eqref{tmupm.sharp.eqn}:
\begin{equation}
    \label{eq:mu-k}
    \frac{\partial \tmu_\sharp^\pm }{\partial k}=\frac{\partial}{\partial k} \left[{g^\pm_{k,u}} (\delta) \right]+\left(\frac{\partial}{\partial k} g^\pm_{k,u}\right) (\tmu_\sharp^\pm) + g^\pm_{k,u}\left( \frac{\partial}{\partial k} \tmu_\sharp^\pm\right).
\end{equation}

\medskip

(i) Proof of  \eqref{est: mu-k}:

\medskip

To obtain \eqref{est: mu-k}, we will solve \eqref{eq:mu-k}  in the space $L^\infty_k L^2_l$.  First, we bound the first inhomogeneous term on the right-hand side of \eqref{eq:mu-k}.
As
\begin{align}
	\nonumber
	\frac{\diff}{\diff k}\left( g_{k,u}^\pm (\sqrt{2\pi} \delta) \right)
		&=	-i \int_{\pm l \cdot \infty}^y e^{-il(l+2k)(y-\eta)}(-2il)(y-\eta)\tu(l;\eta) \, d\eta,
\end{align}
we may bound
\begin{align}
	\label{tmu.k.inhom.bd1}
	\norm[L^2_l]{\frac{\diff}{\diff k}\left( g_{k,u}^\pm (\sqrt{2\pi} \delta) \right)}
		&\lesssim \int_{-\infty}^\infty \norm[L^2_l]{|l|\,|y-\eta|\,|\tu(l;\eta)|} \, d\eta\\
		\nonumber
		&\lesssim |y| \norm[L^1_y L^{2}_l]{l \tu } + \norm[L^1_y L^2_l]{\tu} \\
		\nonumber
		&\lesssim (1+|y|) \norm[L^{2,1}_y L^{2,1}_l]{\tu}.
\end{align}
Next, we deal with the second inhomogeneous term:
\begin{multline}
\label{mu-l}
    \left(
    	\frac{\partial}{\partial k} 
    	g^\pm_{k,u} 
    \right)(\tmu_\sharp^\pm)\\ 
    =i(2 \pi)^{-1 / 2} \int_{ \pm \infty \cdot l}^y -2il(y-y')  e^{-i l(l+2 k)(y-y')} \widetilde{u} *\tmu_\sharp^\pm(l ; y')d y'.
\end{multline}
From this we deduce that 
\begin{align}
    \Bigl\vert  (\partial_k{g^\pm_{k,u}})& (\tmu_\sharp^\pm) \Bigr\vert 
    	\leq i(2 \pi)^{-1 / 2} \\
     \nonumber
     & \left( \int_{ \pm \infty \cdot l}^y \left\vert 2il y  (\widetilde{u} *\tmu_\sharp^\pm)(l ; y') \right\vert + \left\vert 2il y' \widetilde{u} *\tmu_\sharp^\pm(l ; y') \right\vert dy'   \right)
     \intertext{so that}
    \label{tmu.k.inhom.bd2}
    \Bigl\Vert (\partial_k{g^\pm_{k,u}})& (\tmu_\sharp^\pm) \Bigr\Vert_{L^2_l}
    \\
   &\lesssim |y|
   		\Norm{\tmu_\sharp^\pm}{L^{\infty}_{y}L^{2,1}_{l}}
   		\Norm{\widetilde{u}}{L^{2,1}_{y} L^{2,2}_{l}} + 
   		\Norm{\tmu_\sharp^\pm}{L^{\infty}_{y} L^{2,1}_{l}}
   		\Norm{\widetilde{u}}{L^{2,2}_{y}L^{2,1}_{l}}.
   \nonumber
\end{align}
Finally, observe that
\begin{equation}
	\label{est: gku.op.L2l}
		\norm[L^\infty_k L^2_l \to L^\infty_k L^2_l]{g^\pm_{k,u}} 
	\leq \frac{1}{\sqrt{2\pi}} \norm[L^1]{\tu}.
\end{equation}
It now follows from \eqref{tmu.k.inhom.bd1} \eqref{tmu.k.inhom.bd2}, and \eqref{est: gku.op.L2l} that, under the condition $c<1$, equation \eqref{eq:mu-k} admits the solution
$$ \frac{\diff \tmu^\pm_\sharp}{\diff k}
	= (I-g_{k,u}^\pm)^{-1} 
		\left[ 
			\frac{\diff}{\diff k} 
				\left( 
					g_{k,u}^\pm ((2\pi)^\frac12 \delta) 
				\right) + 
			\left(
				\frac{\diff g_{k,u}^\pm}{\diff k}
			\right)(\tmu^\pm_\sharp)
		\right]
$$
which implies the bound \eqref{est: mu-k}.

\medskip

(ii) Proof of \eqref{est: mu-k L2}:

\medskip
By Plancherel's identity,
\begin{align}
	\label{tmu.k.inhom.bd1 L2}
    \Biggl\Vert \frac{\diff}{\diff k} & \left( g_{k,u}^\pm (\sqrt{2\pi} \delta) \right)  \Biggr\Vert_{L^2_k L^2_l} \\
\nonumber
 &=\left(\iint \left\vert \int_{\pm l \cdot \infty}^y e^{-il(l+2k)(y-\eta)}(-2il)(y-\eta)\tu(l;\eta) \, d\eta \right\vert^2 dkdl\right)^{1/2}\\
 \nonumber
		&\lesssim \left(\int_{\pm l \cdot \infty}^y \norm[L^2_\eta]{|l|^{ \frac32}\,(|y-\eta)\,|\tu(l;\eta)|}^2 dl \right)^{1/2}\\[5pt]
		\nonumber
		&\lesssim |y| \norm[L^2_y L^{2}_l]{|l|^{ \frac32} \tu } + \norm[L^{2,1}_y L^2_l]{|l|^{\frac32} \tu} \\[5pt]
		\nonumber
		&\lesssim (1+|y|) \norm[L^{2,1}_y L^{2,{2}}_l]{\tu}.
\end{align}
Moreover, using \eqref{mu-l} again 
\begin{align}
    \label{tmu.k.inhom.bd2 L2}
    \Bigl\Vert & (\partial_k{g^\pm_{k,u}})(\tmu_\sharp^\pm) \Bigr\Vert_{L^2_k L^2_l}
    \\
    \nonumber
   &\lesssim |y|\Norm{\tmu_\sharp^\pm}{L^{\infty}_{y} (\R, L^2_k L^{2,1}_{l})}\Norm{\widetilde{u}}{L^{2,1}_{y}  L^{2,1}_{l}} + \Norm{\tmu_\sharp^\pm}{L^{\infty}_{y} (\R, L^2_k L^{2,1}_{l})}\Norm{\widetilde{u}}{L^{2,2}_{y}L^{2,1}_{l}}.
\end{align}
So using \eqref{tmu.k.inhom.bd1 L2} and \eqref{tmu.k.inhom.bd2 L2} and the operator bound \eqref{gku.bd.X}, we conclude that \eqref{est: mu-k L2} holds.

\medskip

(iii) Proof of \eqref{est: mu-k-l}:

\medskip

On multiplying by $l$, equation \eqref{eq:mu-k}
becomes
\begin{align}
	\label{eq:mu-k-l}
	l \frac{\diff \tmu^\pm_\sharp}{\diff k}
		&=	l \frac{\diff}{\diff k}\left[ g_{k,u}^\pm((2\pi)^\frac12 \delta)\right] + 
			l \left(\frac{\diff}{\diff k} g_{k,u}^\pm \right)(\tmu^\pm_\sharp) + 
			l g_{k,u}^\pm \left( \frac{\diff}{\diff k} \tmu^\pm_\sharp \right)\\
	\nonumber	
		&=	h + g_{k,u}^\pm \left( l \frac{\diff \tmu^\pm_\sharp}{\diff k}\right). 
	\intertext{Here}
	\label{eq:mu-k-l.h}
		h  &= h_1 + h_2 + h_3\\
    \intertext{where}
    \nonumber
    h_1 &= l \frac{\diff}{\diff k}\left[ g_{k,u}^\pm((2\pi)^\frac12 \delta)\right],\\
    \nonumber
    h_2 &= l \left(\frac{\diff}{\diff k} g_{k,u}^\pm \right)(\tmu^\pm_\sharp),\\
    \nonumber
    h_3 &= - \frac{i}{\sqrt{2\pi}}
			\int_{\pm l \cdot \infty}^y 
				e^{-il(l+2k)(y-\eta)} 
					(l-l') \tu (l-l';y) 
						\frac{\diff \tmu^\pm_\sharp}{\diff k} 
			\, dl' \, d\eta.
\end{align}
To prove \eqref{est: mu-k-l}, it suffices to show that $h \in L^2_l$ with a bound of order $(1+|y|)$. 

We consider the three terms in the inhomogeneous term \eqref{eq:mu-k-l.h}.

First, we bound $h_1$. Since
\begin{align}
	\label{eq:mu-k-l.h1}
	h_1	&=	-2 \int_{\pm l \cdot \infty}^y e^{-il(l+2k)(y-\eta)} l^2(y-\eta)\tu(l;\eta) \, d\eta,
\intertext{we may estimate}
	\label{mukl.k.h1.est}
	\norm[L^2_l]{h_1}
		&\lesssim |y| \, \norm[L^{2,2}_l L^{1}_y]{\tu} + \norm[L^{2,{2}}_l L^{1,1}_y]{\tu}\\
		\nonumber
		&\lesssim |y| \, \norm[L^{2,2}_l L^{2,2}_{ y}]{\tu}.
\end{align}

Next, we bound $h_2$. Integrating by parts, we have 
\begin{align}
    \label{mukl.k.h2.form}
    l &\left(\frac{\partial g^\pm_{k,u}}{\diff k} \right) (\tmu_\sharp^\pm) \\
    \nonumber
    &=i(2 \pi)^{-1 / 2} \int_{ \pm \infty \cdot l}^y -2il^2(y-y')  e^{-i l(l+2 k)(y-y')} \widetilde{u} *\tmu_\sharp^\pm(l ; y')d y' \\
    \nonumber
    &=-i(2 \pi)^{-1 / 2}\int_{ \pm \infty \cdot l}^y 2(y-y')\frac{\partial}{\partial y'} e^{-i l(l+2 k)(y-y')} \widetilde{u} *\tmu_\sharp^\pm(l ; y')d y'\\
    \nonumber
    &\quad +i(2 \pi)^{-1 / 2}\int_{ \pm \infty \cdot l}^y 2lk(y-y')e^{-i l(l+2 k)(y-y')} \widetilde{u} *\tmu_\sharp^\pm(l ; y')d y'\\
    \nonumber
    &=i(2 \pi)^{-1 / 2}\int_{ \pm \infty \cdot l}^y 2 e^{-i l(l+2 k)(y-y')} \widetilde{u} *\tmu_\sharp^\pm(l ; y')d y'\\
    \nonumber
    &\quad +i(2 \pi)^{-1 / 2}\int_{ \pm \infty \cdot l}^y 2(y-y') e^{-i l(l+2 k)(y-y')}\frac{\diff}{\diff y'}\left[\widetilde{u} *\tmu_\sharp^\pm(l ; y')\right]d y'\\
    \nonumber
   &\quad +i(2 \pi)^{-1 / 2}\int_{ \pm \infty \cdot l}^y 2lk(y-y')e^{-i l(l+2 k)(y-y')} \widetilde{u} *\tmu_\sharp^\pm(l ; y')d y',
\end{align}
where
\begin{align}
    \label{muk.k.h2.form.conv}
    \frac{\diff }{\diff y'}\left(\widetilde{u} *\tmu_\sharp^\pm\right)(l ; y')
    	&= \int_{\mathbb{R}}\frac{\diff \widetilde{u}}{\diff y'}(l', y')\tmu_\sharp(l-l',y')dl'\\
    	\nonumber
    	&\quad +\int_{\mathbb{R}}\widetilde{u}(l', y')\frac{\diff \tmu_\sharp^\pm}{\diff y'}(l-l',y')dl'.
\end{align}
Thus we deduce 
\begin{align}
	\label{mu.kl.h2.est}
	\bigl\Vert h_2 & \bigr\Vert_{L^2_l}\\
	\nonumber
    &\lesssim 	\Norm{\tmu_\sharp^\pm}{L^{\infty}_{y}L^{2}_{l}}
    			\Norm{\widetilde{u}}{L^{2,1}_{y} L^{2,1}_{l}} \\
    \nonumber
    &\quad +|y|	\Norm{\tmu_\sharp^\pm}{L^{\infty}_{y}L^{2}_{l}}
    			\Norm{\frac{\diff \widetilde{u}}{\diff y}}{L^{2,1}_{y} L^{2,1}_{l}}+
    			\Norm{\tmu_\sharp^\pm}{L^{\infty}_{y}L^{2,1}_{l}}
    			\Norm{\frac{\diff \widetilde{u}}{\diff y}}{L^{2,2}_{y} L^{2,1}_{l}}\\
    \nonumber
    &\quad +|y|	\Norm{\frac{\diff \tmu_\sharp^\pm }{\diff y}}{L^{\infty}_{y}L^{2}_{l}}
    			\Norm{\widetilde{u}}{L^{2,1}_{y} L^{2,1}_{l}}+
    			\Norm{\frac{\diff \tmu_\sharp^\pm}{\diff y}}{L^{\infty}_{y}L^{2}_{l}}
    			\Norm{\widetilde{u}}{L^{2,2}_{y} L^{2,1}_{l}}\\
    \nonumber
    &\quad +|k|\left( |y|\Norm{\tmu_\sharp^\pm}{L^{\infty}_{y}L^{2,1}_{l}}\Norm{\widetilde{u}}{L^{2,1}_{y} L^{2,1}_{l}} + \Norm{\widetilde{u}}{L^{2,2}_{y} L^{2,2}_{l}}\Norm{\tmu_\sharp^\pm}{L^{\infty}_{y}L^{2}_{l}}\right).
\end{align}

Finally, we bound $h_3$. We have
\begin{align}
	\label{mukl.k.h3.est}
	\norm[L^2_l]{h_3}
		&\lesssim 	\int_{-\infty}^\infty \norm[L^2_l]{\left(l \tu * \frac{\diff \tmu^\pm_\sharp}{\diff k} \right)(k,\cdot,\eta) } \, d\eta \\
		\nonumber
		&\lesssim \norm[L^1_l L^1_y]{l\tu} \norm[L^2_l]{\frac{\diff \tmu^\pm_\sharp}{\diff k} }\\
		\nonumber
		&\lesssim \norm[L^{2,2}_l L^{2,1}_y]{\tu} (1+|y|),
\end{align}
where we used \eqref{est: mu-k} in the last line.

Combining \eqref{mukl.k.h1.est}--\eqref{mukl.k.h3.est} and the $L^2_l \to L^2_l$ boundedness of the resolvent, we recover \eqref{est: mu-k-l}.

\medskip
(iv) Proof of \eqref{est: mu-k  L2lk}:

\medskip

Using \eqref{tmupm.sharp.eqn} we find that
\begin{align}
	\label{tmupm.kl.eqn}
		l\frac{\diff \tmu^\pm_\sharp}{\diff k}
			&= h_{1,1} +
				g_{k,u}^\pm \left( l\frac{\diff \tmu^\pm_\sharp}{\diff k}\right) ,
		\intertext{where}
	\label{tmupm.k.h11}
		h_{1,1} 
			&= l \frac{\diff g_{k,u}^\pm}{\diff k}(\sqrt{2\pi}\delta) 
				+ l \frac{\diff g_{k,u}^\pm}{\diff k}(\tmu^\pm_\sharp).
\end{align}
Since $(I-g_{k,u}^\pm)^{-1}$ is a bounded operator from $L^2_lL^2_k$ to itself, it suffices to show that $h_{1,1}$ belongs to $L^2_k L^2_l$ spaces up to factors of $(1+|y|)$. 

To estimate $h_{1,1}$, we first note that 
\begin{align}
	\label{tmupm.k.h11.t1}
		l \frac{\diff g_{k,u}^\pm}{\diff k}&(\sqrt{2\pi}\delta)\\
			&=	\int_{\pm l \cdot \infty}^0 e^{il(l+2k)\eta}2l^2\eta\tu(l;y+\eta)\, d\eta,
			\nonumber
	\intertext{and}
    \label{tmupm.k.h11.gku.conv}
		l \frac{\diff g_{k,u}^\pm}{\diff k}&(f)(l;y)\\
			&= -\frac{1}{\sqrt{2\pi}}
				\int_{\pm l \cdot \infty}^0 e^{il(l+2k)\eta} (2\eta) 
			\nonumber\\[5pt]
			&\qquad \qquad 	
						\left[
							(l^2\tu)*f + 2 (l\tu)*(lf) + \tu*(l^2f)
						\right](l;y+\eta) 
				\, d\eta,
			\nonumber 
\end{align}
from which we obtain the bounds
\begin{align}
	\label{gku.lk.delta.l2kl}
		\norm[L^2_k L^2_l]{l \frac{\diff g_{k,u}^\pm}{\diff k}(\sqrt{2\pi}\delta)}
			&\lesssim (1+|y|)\norm[L^2_\eta L^2_l]{(1+l^2)\tu}\\
            &\quad + \norm[L^2_\eta L^2_l]{(1+l^2)l\eta \tu(\eta,l) }
            \nonumber\\
	\label{gku.lk.l2kl}
		\norm[L^2_k L^2_l]{l \frac{\diff g_{k,u}^\pm}{\diff k}(\tmu^\pm_\sharp)}
			&\lesssim 
                |y|\norm[L^1_{l,\eta}]{l^2 \tu} 
                    \norm[L^\infty_\eta(L^2_kL^2_l)]{\tmu^\pm_\sharp}\\
                &\quad + 
                \norm[L^1_{l,\eta}]{\eta l^2 \tu}
                \norm[L^\infty_\eta(L^2_k L^2_l)]{\tmu^\pm_\sharp}
			\nonumber\\
            &\quad + |y| \norm[L^1_{l,\eta}]{l\tu} 
                        \norm[L^\infty_\eta(L^2_kL^2_l)]{l\tmu^\pm_\sharp}
            \nonumber\\
            &\quad + \norm[L^1_{l,\eta}]{\eta l \tu} 
                        \norm[L^\infty_\eta(L^2_lL^2_k)]{l \tmu^\pm_\sharp}
            \nonumber\\
            &\quad + |y|\norm[L^1_{l,\eta}]{\tu} 
                        \norm[L^\infty_\eta(L^2_k L^2_l)]{l^2 \tmu^\pm_\sharp} 
            \nonumber\\
            &\quad + \norm[L^1_{l,\eta}]{\eta \tu} 
                        \norm[L^\infty_\eta(L^2_k L^2_l)]{l^2 \tmu^\pm_\sharp}
            \nonumber
\end{align}
To obtain \eqref{gku.lk.delta.l2kl}, we used \eqref{tmupm.k.h11.t1} and the estimate \eqref{Iv.bd} together with the bound $\norm[L^2_k L^{2,-1}_l]{l^2 f} \lesssim 
\norm[L^2_k L^2_l]{(1+l^2) f}$. To obtain \eqref{gku.lk.l2kl}, we bounded each of the three right-hand terms in \eqref{tmupm.k.h11.gku.conv} (with $f=\tmu^\pm_\sharp$) using the estimate
$$ \int_{\R} \norm[L^2_k L^2_l]{v*f(\dotarg,\dotarg;y)} \, dy \leq \norm[L^1_y L^1_l]{v} \norm[L^\infty_y(L^2_k L^2_l)]{f},$$
where the convolution is in the $l$ variable only and $v=v(l;y)$. There are two terms for each right-hand term of \eqref{tmupm.k.h11.gku.conv} as the convolution is a function of $y+\eta$. The use of the variables $y$ and $\eta$ distinguishes between $y$ dependence and estimates involving $L^1$ and $L^\infty$ norms.

These estimates show that $\norm[L^2_k L^2_l]{h_{1,1}} \lesssim (1+|y|)$ provided that \eqref{u.C1}, \eqref{u.C2}, \eqref{u.C2a}, \eqref{u.C3} hold 
(in particular these conditions imply that the hypotheses of Lemma \ref{lemma:tmu.smooth}(i) hold with $j=2,j'=0$). 

\end{proof}

\begin{remark}
    From Proposition \ref{prop: mu-k} we can deduce
\begin{align}
\label{tmupm.k.est3}
	\norm[ L^\infty_k L^1_l ]{\frac{\diff \tmu^\pm_\sharp}{\diff k}}
	&\lesssim (1+|y|). 
\end{align}
 To obtain \eqref{tmupm.k.est3}, we use \eqref{est: mu-k} and \eqref{est: mu-k-l} together with the estimate
$$ \norm[L^1_l]{f} \lesssim \norm[L^2]{(1+l^2)^\frac12 f}.$$
\end{remark}

\section{The Non-Local Riemann Hilbert Problem}
\label{sec:NLRHP}

Following \cite{Zhou90}, we now formally construct the scattering data (\S \ref{subsec:scattering.data}) and the nonlocal RHP (\S \ref{subsec:NLRHP.def}). %Finally, we prove key estimates on the large-time behavior of solutions to the nonlocal Riemann-Hilbert problem in the absence or presence of critical points of the phase function \eqref{S0.def.intro} (\S \ref{subsec:NLRHP.est}).

\subsection{Scattering Data}
\label{subsec:scattering.data}

Below we denote by $\psi$ (respectively, $\phi$) the solution corresponding to $\mu^l, \mu^r, \mu^\pm$ (respectively $\nu^l, \nu^r, \nu^\pm$ see \cite[(2.5)]{Zhou90} for the derivation ), 
%\todo[size=\tiny]{$\nu^l$, $\nu^r$, and $\nu^\pm$ are never defined}
and $\Psi$ (respectively $\Phi$) the integral operator with kernel $\psi(\cdot, \cdot; y)$ (respectively $\phi(\cdot, \cdot; y)$), and by $W(y)$ the unitary multiplier by $e^{-ik^2y}$ and by $\Psi_0(y)$ the integral operator with kernel $(2\pi)^{-1/2}(\mu(k,x;y)-1)e^{ikx} $. Finally, let $\mathbf{F}$ denote the Fourier transform: $\mathbf{L}_2(dk)\to \mathbf{L}_2(dx)$. We can write for instance 
\begin{align}
\Psi(y)[f]&=\dfrac{1}{\sqrt{2\pi }}e^{-ik^2y}\int \left( f(x)e^{ikx} + f(x)(\mu(k,x;y)-1)e^{ikx}\right)dx\\
\nonumber
 &=\dfrac{1}{\sqrt{2\pi }}\int e^{-i(-kx+k^2y)}f(x)\mu(k,x;y)dx
\end{align}
and similarly
\begin{align}
\Phi(y)[f]&=\dfrac{1}{\sqrt{2\pi }}\int \left( f(k)e^{-ikx} + f(k)(\nu(k,x;y)-1)e^{-ikx}\right)e^{ik^2y}dk\\
\nonumber
 &=\dfrac{1}{\sqrt{2\pi }}\int e^{i(-kx+k^2y)}f(k)\nu(k,x;y)dk.
\end{align}
Letting $\check{f}$ be the inverse\textit{ Fourier} transform with respect to the $x$ variable,  we proceed to define
%\todo[size=\tiny]{Should $v$ be $\nu$ in equations for $\widetilde{T}$ and $\widetilde{R}$?}
%\todo[size=\tiny]{PAP added labels for (3.3)-(3.6) for use later - no other changes made}
\begin{align}
\label{scatt.Tpm.def} %% PAP
& T^{ \pm}(k, k+l)=-i(2 \pi)^{-1} H( \pm l) \int  \widetilde{u} * \widetilde{\mu}^{ \pm}(k, l ; y') e^{i l(l+2 k) y'}d y', \\
\label{scatt.tildeTpm.def} %% PAP
& \tilde{T}^{ \pm}(k+l, k)=i(2 \pi)^{-1} H(\mp l) \int  \check{u} * \check{\nu}^{ \pm}(k, l ; y') e^{-i l(l+2 k) y'}d y', \\
\label{scatt.Rpm.def} %% PAP
& R^{ \pm}(k, k+l)=i(2 \pi)^{-1} H(\mp l) \int  \widetilde{u} * \widetilde{\mu}^{ \pm}(k, l ; y') e^{i l(l+2 k) y'}d y', \\
\label{scatt.tildeRpm.def} %% PAP
& \tilde{R}^{ \pm}(k+l, k)=-i(2 \pi)^{-1} H( \pm l) \int  \check{u} * \check{\nu}^{ \pm}(k, l ; y') e^{-i l(l+2 k) y'}d y'.
\end{align}

We then define integral operators $\calT^\pm$ and $\calR^\pm$ by 
\begin{align}
    \label{calTpm.op.def}
    \left( \calT^\pm_{x,y} f\right)(k)& = \int T^\pm(k,l) e^{i(l-k)x-i(k^2-l^2)y} f(l) \, dl, \\
    \label{caltildeT.op.def}
    \left( \widetilde{\calT}^\pm_{x,y} f\right)(k)& = \int \tilde{T}^\pm(k,l) e^{i(l-k)x-i(k^2-l^2)y} f(l) \, dl,
    \\
    \label{calRpm.op.def}
    \left( \calR^\pm_{x,y} f \right)(k) &= \int R^\pm(k,l) e^{i(l-k)x-i(k^2-l^2)y} f(l) \, dl \\
    \label{caltildeR.op.def}
    \left( \widetilde{\calR}^\pm_{x,y} f\right)(k)  &= \int \widetilde{R}^\pm(k,l) e^{i(l-k)x-i(k^2-l^2)y} f(l) \, dl
\end{align}

By letting $y\to \pm \infty$, we first note that 
\begin{equation}
\Psi^l\Phi^l=\Psi^r\Phi^r=I
\end{equation}
and further derive
\begin{align}
\label{phipsi}
\Psi^+&\Phi^l\left[f \right]\\
	&=W(y)\left( \mathbf{F}^{-1}+\Psi^+_0(y) \right)\left( \mathbf{F}+\Phi^l_0(y) \right)W^{-1}(y)\left[f \right]
	\nonumber\\
              &=W(y)\left(I+\Psi^+_0(y) \mathbf{F} \right)W^{-1}(y)\left[f \right]
              \nonumber\\
              &=f+W(y)\Psi^+_0(y) \mathbf{F} W^{-1}(y)\left[f \right]
              \nonumber\\
              &=f+   \dfrac{1}{\sqrt{2\pi }}\int e^{-i(-kx+k^2y)}(\mu^+(k,x;y)-1) \left(\int e^{-ilx} e^{il^2y} f(l)\, dl \right) \, dx
              \nonumber\\
              &= f+ \frac{1}{\sqrt{2\pi}}\int e^{i(l^2-k^2)y}\left[  \frac{1}{\sqrt{2\pi}} \int e^{-i(l-k)x}  \left(\mu^+(k,x;y)-1  \right)dx\right]f(l) \, dl
 \nonumber \\            
              &=f+ \frac{1}{\sqrt{2\pi}}\int e^{i(l^2-k^2)y} 
              \nonumber\\
              &\qquad \left[ \frac{i}{\sqrt{2\pi}}\int_{+\infty\cdot(l-k)}^y e^{-i(l^2-k^2)(y-y')}\widetilde{u}*\widetilde{\mu}^+(k,l-k; y')dy' \right] f(l)\, dl
              \nonumber\\
              &=f-\frac{i}{\sqrt{2\pi}}\int H(l-k)\left[ \int e^{i(l^2-k^2)y'}\widetilde{u}*\widetilde{\mu}^+(k,l-k; y')dy'  \right] f(l)\, dl
              \nonumber\\
              &=f-\frac{i}{\sqrt{2\pi}}\int H(+l)\left[\int e^{il(l+2k)y'} \widetilde{u}*\widetilde{\mu}^+(k,l; y')dy'\right] f(l+k)\, dl
             \nonumber 
\end{align}
Thus we obtain
\begin{equation}
\lim_{y\to-\infty } \Psi^+\Phi^{l}[f]=f+\int T^+(k,k+l)f(k+l)dl.
\end{equation}
Similarly,
\begin{align}
\label{psiphi}
\Psi^l&\Phi^+\left[f \right]\\
	\nonumber 
	&=W(y)\left( \mathbf{F}^{-1}+\Psi^l_0(y) \right)\left( \mathbf{F}+\Phi^+_0(y) \right)W^{-1}(y)\left[f \right]\\
\nonumber
              &=W(y)\left(I+ \mathbf{F}^{-1}\Phi^+_0(y) \right)W^{-1}(y)\left[f \right]\\
              \nonumber
              &=f+W(y)\mathbf{F}^{-1}\Phi^+_0(y)  W^{-1}(y)\left[f \right]\\
              \nonumber
              &=f+   \dfrac{1}{\sqrt{2\pi }}\int e^{i(-l^2y+lx)} \left(\int e^{-ikx+ik^2y}(\nu^\pm(k,x;y)-1) f(k)dk \right) \, dx\\
              \nonumber
              &= f+ \frac{1}{\sqrt{2\pi}}\int e^{i(-l^2+k^2)y}\left[  \frac{1}{\sqrt{2\pi}} \int e^{i(l-k)x}  \left(\nu^+(k,x;y)-1  \right)dx\right]f(k) \, dk\\
 \nonumber             
              &=f+\frac{1}{\sqrt{2\pi}}\int e^{i(-l^2+k^2)y}\\
             \nonumber
              & \qquad  \left[ \frac{i}{\sqrt{2\pi}}\int_{-\infty\cdot(l-k)}^y e^{i(l^2-k^2)(y-y')}\check{u}*\check{\nu}^+(k, l-k; y')dy' \right] f(k)\, dk\\
              \nonumber
              &=f-\frac{i}{\sqrt{2\pi}}\int H(k-l)\left[ \int e^{i(l^2-k^2)y'}\check{u}*\check{\nu}^+(k, l-k; y') \, dy'  \right] f(k)\, dk\\
              \nonumber
              &=f-\frac{i}{\sqrt{2\pi}}\int H(-l)\left[\int e^{il(l+2k)y'} \check{u}*\check{\nu}^+(k,l; y')dy'\right] f(k)\, dk.
\end{align}
We then consider
\begin{equation}
\Psi^+\Phi^+[f]=(I+\mathcal{T}^+)(I+\widetilde{\mathcal{T}}^+)[f].
\end{equation}
 We observe that the integral on the RHS of \eqref{psiphi} is given by
$$\int_l^\infty\left[ \int e^{i(l^2-k^2)y'}\check{u}*\check{\nu}^+(k, l-k; y')dy'  \right] f(k)dk$$
which implies that $\widetilde{\mathcal{T}}$ is upper triangular. Similarly, the integral on the RHS of \eqref{phipsi} is given by  
$$\int_k^\infty\left[ \int e^{i(l^2-k^2)y'}\widetilde{u}*\widetilde{\mu}^+(k,l-k; y')dy'  \right] f(l)dl$$
which implies that ${\mathcal{T}}$ is upper triangular. We thus conclude that $\Psi^+\Phi^+=I$. Now suppose that
$\Phi^+f=0$, then we have
\begin{equation}
\Psi^+ \Phi^+ f=0=f
\end{equation}
Since $\Phi^+$ is Fredholm with index zero, we conclude that $\Phi^+$ is invertible and thus $\Psi^+$ is invertible. 
Setting the operator
\begin{equation}
\label{op: f}
I+\mathcal{F}:=(I+\mathcal{T})(I-\widetilde{\mathcal{T}})=(I+\mathcal{R})(I-\widetilde{\mathcal{R}})
\end{equation}
and this operator relates:
\begin{equation}
\label{RHP-1}
\Psi^+=(I+\mathcal{F})\Psi^-.
\end{equation}
\begin{proposition}\cite[Proposition 4.4]{Zhou90}
If we assume that $I+\mathcal{S}$ is unitary, then the only solution to the Riemann-Hilbert problem \cite[(4.1)]{Zhou90} is zero.
\end{proposition}
\begin{proof}
We first derive the following identity 
\begin{equation}
\left( I-\mathcal{T}^-_{x,y} \right)^{-1}=I+\mathcal{T}^{+*}_{x,y}.
\end{equation}
To see this, we observe that
\begin{align*}
&(I+\mathcal{R}^+)(I+\mathcal{S}^+)=I+\mathcal{T}^+\\
&\Leftrightarrow (I+\mathcal{S}^{+*})(I+\mathcal{R}^{+*})=I+\mathcal{T}^{+*}\\
&\Leftrightarrow (I+\mathcal{S})^{-1}(I+\mathcal{R}^{+*})=I+\mathcal{T}^{+*}\\
&\Leftrightarrow (I+\mathcal{S})^{-1}(I-\widetilde{\mathcal{R}}^{-})=I+\mathcal{T}^{+*}\\
&\Leftrightarrow (I-\widetilde{\mathcal{R}}^{-})=(I-\mathcal{R}^-)^{-1}(I-\mathcal{T}^-)(I+\mathcal{T}^{+*})\\
&\Leftrightarrow I=(I-\mathcal{T}^-)(I+\mathcal{T}^{+*})
\end{align*}
where the fourth equality comes from \cite[Proposition 3.18]{Zhou90} and the fifth equality comes from \cite[(3.17)]{Zhou90}. Notice that 
\begin{align}
\int \overline{\mu^-(l; x,y)}\mu^+(l; x,y)dl=0
\end{align}
we therefore deduce
\begin{align}
    \int \overline{\mu^-(l; x,y)}\mu^+(l; x,y)dl &= \int  \overline{\mu^-(l; x,y)} (I+\mathcal{F}_{x,y})\mu^-(l; x,y)dl\\
    \nonumber
        &=\int\overline{(I+\calT^{+*}_{x,y})}\overline{\mu^-(l; x,y)}(I-\widetilde{\calT}^-_{x,y})\mu^-(l; x,y) dl\\
        \nonumber
        &= \Norm{(I+\calT^{+*}_{x,y})\mu^-(\cdot; x,y)}{L^2}^2\\
        \nonumber
        &=0
\end{align}
where we make use of the fact that
$$(I-\calT^-_{x,y})^{-1}=I+ \calT^{+*}_{x,y} =I-\widetilde{\calT}^-_{x,y}.$$
\end{proof}

\subsection{The Nonlocal Riemann-Hilbert Problem}
\label{subsec:NLRHP.def}

We now proceed to derive the nonlocal Riemann-Hilbert problem for the inverse problem. We first define two integral operators:
\begin{align}
\label{CT.left.time0}
C_{T_{x,y}}&:= C_+ \calT^-_{x,y}+C_- \calT^+_{x,y},\\
\label{CT.right.time0}
C_{R_{x,y}}&:= C_+ \calR^-_{x,y}+C_- \calR^+_{x,y}.
\end{align}

\begin{proposition}\cite[Proposition 4.10]{Zhou90}
\label{prop:NLRHP.def}
A vector $\mu^l$ respectively $u^r$ is a fundamental solution of the Riemann-Hilbert problem 
\begin{align*}
\mu^\pm &= (I\pm \calT^\pm_{x,y})\mu^l,\\
\mu^\pm &= (I\pm \calR^\pm_{x,y})\mu^r
\end{align*}
respectively if and only if it satisfies the equations
\begin{align*}
\mu^l&=1+C_{T_{x,y}}\mu^l,\\
\mu^r&=1+C_{R_{x,y}}\mu^r
\end{align*}
respectively and the homogeneous solutions of these equations are exactly the vanishing solutions of the Riemann-Hilbert problem.
\end{proposition}
\begin{proof}
Notice that
\begin{align*}
\mu^+-\mu^-=\calT^+_{x,y}\mu^l+\calT^-_{x,y}\mu^l
\end{align*}
hence taking into account the normalization condition, we obtain from Plemelj's formula
\begin{equation}
\label{mu-int-1}
\mu^\pm=1+ C_\pm \left[ \calT^+_{x,y}\mu^l+\calT^-_{x,y}\mu^l \right]
\end{equation}
\begin{equation}
\label{mu-int-2}
\mu^\pm=1+ C_\pm \left[ \calR^+_{x,y}\mu^r+\calR^-_{x,y}\mu^r\right],
\end{equation}
and in particular
\begin{align*}
\mu^+&=1+C_+ \left[ \calT^+_{x,y}\mu^l+\calT^-_{x,y}\mu^l \right]\\
   &\Rightarrow (I+\calT^+_{x,y})\mu^l = 1+\calT^+_{x,y}\mu^l+ C_+\calT^-_{x,y}\mu^l+ C_-\calT^+_{x,y}\mu^l\\
   &\Rightarrow \mu^l=1+C_{T_{x,y}}\mu^l.
\end{align*}
We have similarly
$$\mu^r=1+C_{R_{x,y}}\mu^r.$$
\end{proof}

The following proposition deals with the solvability of the nonlocal RHP:
\begin{proposition}\cite[Proposition 4.13]{Zhou90}
The norm of the operator $\left( I-C_{T_{x,y}} \right)^{-1}$ (respectively $\left( I-C_{R_{x,y}} \right)^{-1}$), for a fixed $y_0$ is bounded by a number $M_{y_0}$ independent of $x$ and $y$ for large $|x|+|y|$ with $y\leq y_0$ (respectively $y\geq y_0$).
\end{proposition}
\begin{proof}
We first consider the following identity:
\begin{equation}
\label{op:E}
\left( I-C_{\widetilde{T}_{x,y}} \right)\left( I-C_{T_{x,y}} \right)=I+E_{T_{x,y}}
\end{equation}
where 
$$ E_{T_{x,y}}= C_+{\widetilde{\calT}_{x,y}^-} C_-\left(\calT_{x,y}^+ + \calT_{x,y}^-  \right) +C_-{\widetilde{\calT}_{x,y}^+} C_+\left(\calT_{x,y}^+ + \calT_{x,y}^-  \right). $$
Note that
\begin{align*}
C_{\widetilde{T}_{x,y}}C_{T_{x,y}}=C_+{\widetilde{\calT}_{x,y}^-}\left( C_-\calT_{x,y}^+ + C_+\calT_{x,y}^-  \right)+C_-{\widetilde{\calT}_{x,y}^+}\left( C_-\calT_{x,y}^+ + C_+\calT_{x,y}^-  \right).
\end{align*}
We then make use of the fact that 
\begin{align*}
C_+\calT^-_{x,y}&=\calT^-_{x,y}+C_-\calT^-_{x,y},\\
C_-\calT^+_{x,y}&=-\calT^+_{x,y}+C_+\calT^-_{x,y}
\end{align*}
so that
\begin{align*}
\left( I-C_{\widetilde{T}_{x,y}} \right)&
\left( I-C_{T_{x,y}} \right)\\
	&= I-	C_+\widetilde{\calT}^-_{x,y}-
			C_-\widetilde{\calT}^+_{x,y}-
			C_+\calT^-_{x,y}-
			C_-\calT^+_{x,y}+
			C_{\widetilde{T}_{x,y}}C_{T_{x,y}}\\
	&= I-
			C_+\widetilde{\calT}^-_{x,y}-
			C_-\widetilde{\calT}^+_{x,y}-
			C_+\calT^-_{x,y}-C_-\calT^+_{x,y}\\
	&\quad + C_+{\widetilde{\calT}_{x,y}^-} 
			C_-\left(\calT_{x,y}^+ + \calT_{x,y}^-  \right) +
			C_-{\widetilde{\calT}_{x,y}^+} 
				C_+\left(
						\calT_{x,y}^+ + \calT_{x,y}^-  
					\right)\\
	&\quad +	C_+\widetilde{\calT}^-_{x,y}\calT^-_{x,y}-
			C_-\widetilde{\calT}^+_{x,y}\calT^+_{x,y}.
\end{align*}
Using the fact that $I\pm \widetilde{\calT}^\pm=(I\pm \calT^\pm)^{-1}$ we get the identity \eqref{op:E}. And using the rational approximation argument in the proof of \cite[Propostion 4.13]{Zhou90}, we deduce that 
$$\left( I+E_{T_{x,y}} \right)^{-1}\left( I-C_{\widetilde{T}_{x,y}} \right)\left( I-C_{T_{x,y}} \right)=I.$$
We then repeat the same argument on the following 
\begin{equation}
\label{op:E.bis}
\left( I-C_{T_{x,y}} \right)\left(I-C_{\widetilde{T}_{x,y}}\right)=I+\widetilde{E}_{T_{x,y}}
\end{equation}
to obtain 
$$\left( I-C_{T_{x,y}} \right)\left(I-C_{\widetilde{T}_{x,y}}\right)\left( I+\widetilde{E}_{T_{x,y}}\right)^{-1}=I$$
thus we construct the resolvent operator $\left( I-C_{T_{x,y}} \right)^{-1}$ and the bound follows from the compactness of $C_{\widetilde{T}_{x,y}}$.
\end{proof}

In \cite[(4.27)]{Zhou90}, the reconstruction of the potential $u$ is given by
\begin{equation}
u(x, y)=\frac{1}{\pi} \frac{\partial}{\partial x} \iint\left(T^{+}(k, l)+T^{-}(k, l)\right) e^{i(l-k) x-i\left(l^2-k^2\right) y} \mu^l(l, x ; y) d l d k.
\end{equation}
To derive this, we note that from \cite[Proposition 4.1]{Zhou90} we have the following integral equation representation:
\begin{align}
\label{mu-int}
\mu^{l}(l; x,y)=1+\int \dfrac{\calT^+_{x,y}\mu^l}{k'-l+i0}dk'+\int \dfrac{\calT^-_{x,y}\mu^l}{k'-l-i0}dk',
\end{align}
then we let $l\to\infty$ and substitute \eqref{mu-int} into \eqref{mupm.eqn} to  obtain \cite[(4.27)]{Zhou90}. We then let $L\mu^l=0$ and use the integral representation in \eqref{mu-int-1}-\eqref{mu-int-2} to obtain
\begin{align*}
u(x, y)&=\frac{1}{\pi} \frac{\partial}{\partial x} \iint \left(T^{+}(k, l)+T^{-}(k, l)\right) e^{i(l-k) x-i\left(l^2-k^2\right) y} \mu^l(l, x ; y)d l d k\\
&=\frac{1}{\pi} \frac{\partial}{\partial x} \iint\left(R^{+}(k, l)+R^{-}(k, l)\right) e^{i(l-k) x-i\left(l^2-k^2\right) y} \mu^r(l, x ; y) d l d k.
\end{align*}

\section{Estimates on Scattering Data}
\label{sec:Tpm}

We consider scattering data
\begin{align}
	\label{Tpm}
		T^\pm(k,l) 
		&= -\frac{i}{2\pi} 
			\int e^{-i((l-k)x - (l^2-k^2)y)}
				u(x,y)\mu^\pm(k,x;y) \, dx \, dy\\
		&= -\frac{i}{2\pi}
			\int e^{i(l^2-k^2)y}
				(\tu*\tmu^\pm)(k,l-k;y)
			\, dy
		\nonumber
		\end{align}
where
$$ \tu(l;y) = \frac{1}{\sqrt{2\pi}}\int e^{-ilx} u(x,y) \, dx$$
and
$$ (\tu*\tmu^\pm)(k,l;y) = \int \tu(l-l';y) \tmu^\pm(k,l';y) dl'\, dy.$$ 
In this section, we omit the Heaviside function $H(\pm(l-k))$ from the definition of $T^\pm$ for the present purpose of obtaining both weighted and derivative estimates. The integral \eqref{Tpm} still makes sense for $\pm (l-k) < 0$. 

\textblue{This section is a careful analysis of the nonlinear maps $u \mapsto T^\pm$, which maps initial data for KP I into scattering data. In section \ref{sec:asy}, we will use these decay and regularity properties of the scattering data to study the reconstruction formula for the solution $u(t,x,y)$--which involves both the scattering data and the solution to the nonlocal Riemann-Hilbert problem.  We first establish continuity of the map $\tu \mapsto T^\pm$ (Proposition \ref{prop:Tpm.exist}), then establish weighted estimates (Proposition \ref{prop:Tpm.L2}). Finally, we prove further decay and regularity estimates on $T^\pm$ which we will need to analyze the reconstruction formula (Proposition \ref{prop:Tpm.kl.L2}). Throughout, we assume that $u \in \mathbf{E}_{1,w}$, which gives us control over the norms of $u$ needed to carry out these estimates. }

\textblue{We begin by fixing notation.}  As in section \ref{sec:mupm}, we write
$$ \tmu^\pm(k,l;y) = \sqrt{2\pi} \delta(l) + \tmu^\pm_\sharp(k,l;y)$$
and we will routinely write
\begin{equation}
	\label{Tpm.split}
		T^\pm(k,l) = T_1(k,l) + T_2^\pm(k,l)
\end{equation}
where
\begin{align}
	\label{T1.def}
		T_1(k,l) &= -\frac{i}{2\pi}\int e^{i(l^2-k^2)y} \tu(l-k;y) \, dy,\\
	\intertext{and}
		\label{T2.def}
		T_2^\pm(k,l) &= -\frac{i}{2\pi} 
			\int e^{i(l^2-k^2)y} (\tu*\tmu^\pm_\sharp)(k,l-k;y) \,dy	.
\end{align}
Note that $T_1$ corresponds to the linearization of the scattering transform at $u=0$ (where $\mu^\pm = 1$), so the nonlinearity of the scattering transform is encoded in $T^\pm_2$. 

First we prove a basic boundedness and continuity result for the map $u \mapsto T^\pm$ analogous to Lemma \ref{lemma:tmu.exist} for the scattering solutions. Denote by $B$ the set of $\tu$ with $\norm[L^1]{\tu} < \sqrt{2\pi}$ and $\tu \in L^2_yL^{2,-1}_l$. Recall that $c \coloneqq \norm[L^1]{\tu}/\sqrt{2\pi}$.

\begin{proposition}
	\label{prop:Tpm.exist}
	The maps
	\begin{equation}
		\label{Tpm.map}
		\begin{aligned}
			\tu &\mapsto T^\pm,\\
			B	 &\mapsto L^2(\R^2)	
		\end{aligned}
	\end{equation}
	are bounded continuous maps with 
	\begin{equation}
		\label{T.bdd}
		\norm[L^2]{T^\pm} \leq  \frac{1}{1-c}\norm[L^2_yL^{2,-1}_l]{\tu}
	\end{equation}
	and
	\begin{equation}
		\label{T.cont}
			\norm[L^2]{T^\pm(u_1) - T^\pm(u_2)}
				\leq C(c_1,c_2) \left( \norm[L^2_y L^{2,-1}_l]{\tu_1 - \tu_2}+\norm[L^1]{\tu_1 - \tu_2} \right)
	\end{equation}
	where $C(c_1,c_2)$ diverges as $c_1 \uparrow 1$ or $c_2 \uparrow 1$. 
\end{proposition}

\begin{proof}
	Writing $T^\pm = T_1 + T_2^\pm$ as in \eqref{Tpm.split} we have, using \eqref{Iv.bd} for the estimate on $T_1$ and Young's inequality together with \eqref{tmu.est0.X} for the estimate on $T^\pm_2$, that
	\begin{align}
	\label{T1pm.L2.est}
		\norm[L^2(\R^2)]{T_1}
			&\leq \frac{1}{2\pi^\frac12} \norm[L^2_y L^{2,-1}_l]{\tu},\\
	\label{T2pm.L2.est}
		\norm[L^2(\R^2)]{T_2^\pm}
			&\leq \frac{1}{2\pi} \norm[L^1]{\tu} \norm[X]{\tmu^\pm_\sharp}\\
			&\leq \frac{1}{2\pi^\frac12} \frac{c}{1-c} \norm[L^2_y L^{2,-1}_l]{\tu}
			\nonumber
	\end{align}
	which implies \eqref{T.bdd}. 
	
	Next, from the estimates
	\begin{align*}
		\norm[L^2]{T_1(u_1) - T_1(u_2)}
			&\lesssim \norm[L^2_y L^{2,-1}_l]{\tu_1 - \tu_2},\\
		\norm[L^2]{T_2^\pm(u_1) - T_2^\pm(u_2)}
			&\lesssim \norm[L^1]{\tu_1 - \tu_2} \norm[X]{\tmu^\pm_\sharp(u_1)} +
						\norm[L^1]{\tu_2} \norm[X]{\tmu^\pm_\sharp(u_1) - \tmu^\pm_\sharp(u_2)}\\
			&\lesssim \norm[L^1]{\tu_1 - \tu_2} (1-c_1)^{-1}\norm[L^2_y L^{2,-1}_l]{u_1}\\
			&\quad +
				(1-c_1)^{-1} (1-c_2)^{-1} \norm[L^1]{\tu_1 - \tu_2},
	\end{align*}
	we recover \eqref{T.cont}.	
\end{proof}

Next, we prove an analogue of Zhou's result \cite[Proposition 3.7]{Zhou90}. Our method of proof is similar to his, and our only contribution is to keep careful track of what norms of $\tu$ are required for results of a given order of regularity.

\begin{proposition}
	\label{prop:Tpm.L2}
	Suppose that $\norm[L^1(\R^2)]{\tu} < \sqrt{2\pi}$.
	Fix nonnegative integers $j$ and $j'$ and suppose that for $0 \leq m \leq j$ and $0 \leq m' \leq j'$,
	\begin{align}
		\label{L2.Tpm.u.C1}
			l^m \diff_y^{m'} \tu & \in L^2_y L^{2,-1}_l \cap L^1(\R^2).
	\end{align}	
	Then 
	\begin{equation}
		\label{L2.Tpm.wt.est}
			(1+(l-k)^2)^j (1+(l^2-k^2)^2)^{j'} T^\pm(k,l) \in L^2(\R^2).
	\end{equation}
\end{proposition}

\begin{proof}
	We use the decomposition \eqref{Tpm.split} and change variables from $(k,l)$ to $(k,l+k)$. First, for $u \in \calS(\R^2)$, 
	\begin{align*}
		l^{m} &(l(l+2k))^{m'}T_1(k,k+l)
			\\
			 &=-\frac{i}{2\pi} l^m (l(l+2k))^{m'}\int e^{il(l+2k)y} \tu(l,y) \, dy\\
			 &=-\frac{i}{2\pi} (-i)^{m'} \int e^{il(l+2k)y} l^m \diff_y^{m'}\tu(l;y) \, dy
	\end{align*}
	from which it follows that 
	$$ \norm[L^2(\R^2)]{l^m (l(l+2k))^{m'} T_1(k,k+l)} \lesssim \norm[L^2_y L^{2,-1}_l]{l^m \diff_y^{m'}\tu }$$
	giving
	$$ \norm[L^2(\R^2)]{(l-k)^m (l^2-k^2)^{m'} T_1} \lesssim \norm[L^2_y L^{2,-1}_l]{l^m \diff_y^{m'}\tu }.$$
	
	Next, again assuming $u \in \calS(\R^2)$, we compute
	\begin{align*}
		l^m & (l(l+2k))^{m'}T_2^\pm(k,k+l)\\
			&=-\frac{i}{2\pi} \int (-i \diff_y)^m e^{il(l+2k)y} l^m (\tu*\tmu^\pm_\sharp)(k,l;y) \, dy\\
			&=-\frac{i}{2\pi} (i)^m 
				\int e^{il(l+2k)y} l^m \diff_y^{m'} (\tu*\tmu^\pm_\sharp)(k,l;y) \, dy.
	\end{align*} 
	Using the expansion
	\begin{align*}
		l^m \diff_y^{m'}(\tu*\tmu^\pm_\sharp)(k,l;y)
			&= \sum_{\substack{0 \leq r \leq m\\0 \leq r' \leq m'}}c_{r,r'}(\tu_{m-r,m'-r'} *  \tmu^\pm_{\sharp,r,r'})(k,l;y) 
	\end{align*}
	where
	$$
		c_{r,r'} = \binom{m}{r} \binom{m'}{r'}, \qquad \tu_{r,r'} = l^r \diff_y^{r'} \tu, \qquad 
		\tmu^\pm_{\sharp,r,r'} = l^r \diff_y^{r'} \tmu^\pm_\sharp
	$$
	we see that
	\begin{multline*}
	\norm[L^2(\R^2)]{l^m (l(l+2k))^{m'} T_2^\pm(k,k+l)}\\
	\lesssim 
		\sum_{\substack{0 \leq r \leq m \\ 0 \leq r' \leq m'}} 
			c_{r,r'} 
				\norm[L^1(\R^2)]{\tu_{m-r,m'-r'}}
				\norm[X]{\tmu^\pm_{\sharp,r,r'} }
	\end{multline*}
	and use the fact that $\norm[X]{\tmu^\pm_{\sharp,r,r'}}$ is bounded by Lemma \ref{lemma:tmu.smooth}.
\end{proof}

We will also need the following derivative estimates and pointwise estimates on the scattering data $T^\pm$.  For the statement and proof of Proposition \ref{prop:Tpm.kl.L2}, we drop the factor of $H(\pm(l-k))$ in the definition of $T^\pm$, noting that the integrals in \eqref{scatt.Tpm.def} make sense for either sign of $l$ and define differentiable functions of $k$ and $l$ with the appropriate regularity.

\begin{proposition}
	\label{prop:Tpm.kl.L2}
		Suppose that \eqref{u.C1} and \eqref{u.C2.ancest} hold and that $u \in \mathbf{E}_{1,w}$.
        \begin{enumerate}[(i)]
            \item The estimate
            \begin{equation}
                \label{Tpm.est0}
				(l-k) T^\pm(k,l) \in L^1(\R^2),
            \end{equation}
            holds.
            \item Suppose that $a > \delta$ for a fixed $\delta>0$. Then
            \begin{align}
                \label{Tpm.est01}
                    \frac{T^\pm}{a+l^2} 
                        &\in L^1(\R^2),\\
                \label{Tpm.est02}
                    \frac{l(l-k)T^\pm}{(a+l^2)^2}
                        &\in L^1(\R^2),\\
                \intertext{and}
                \label{Tpm.est03}
                    \frac{(l-k)\diff_lT^\pm(k,l)}{a+l^2}
                        &\in L^1(\R^2).
            \end{align}
            \item For any compact subset $E$ of $\R^2$,
            \begin{equation}
                \label{Tpm.est04}
                    \norm[L^2(E)]{\frac{\diff^3 T^\pm}{\diff k \diff l^2}} \lesssim 1.
            \end{equation}
            \item For any $\psi \in C_0^\infty(\R^2)$, the quantities
            \begin{equation}
                \label{Tpm.est05}
                \begin{aligned}
                    \text{(a)   } & \psi T^\pm, & \qquad 
                    \text{(b)   } & \frac{\diff}{\diff k}(\psi T^\pm), \\
                    \text{(c)   } & \frac{\diff^3}{\diff k \diff^2 l}(\psi T^\pm), & \qquad 
                    \text{(d)   } & \frac{\diff^2}{\diff l^2}(\psi T^\pm),\\
                    \text{(e)   } & (l-k)\frac{\diff}{\diff k}(\psi T^\pm), & \qquad 
                    \text{(f)   } & \frac{\diff^2}{\diff l^2} 
                                    \left( (l-k) \frac{\diff}{\diff k}(\psi T^\pm) \right)
                \end{aligned}
            \end{equation}
            are all bounded in $L^2(\R^2)$.
        \end{enumerate}
\end{proposition}

%\todo[size=\tiny]{Correct remark for new  version of proposition}

\begin{proof}[Proof of Proposition \ref{prop:Tpm.kl.L2}]
	In most of the proof we will use the decomposition \eqref{Tpm.split} to estimate derivatives of $T_1$ and $T_2$ separately. We often use the change of variables $(k,l) \to (k,l+k)$ to simplify the estimates. 
	
	\bigskip
	
	(i): Estimate \eqref{Tpm.est0}  is a consequence of the identity
			\begin{multline*}
				(l-k) T^\pm(k,l) = \\
					\frac{l-k}{(1+(l-k)^2 (1+(l^2-k^2)^2)} (1+(l-k)^2)(1+(l^2-k^2)^2) T^\pm(k,l),
			\end{multline*}
		\eqref{L2.Tpm.wt.est}, and the Cauchy-Schwarz inequality.

    \medskip

    (ii): We prove each of \eqref{Tpm.est01}--\eqref{Tpm.est03} in turn. First,
    \begin{align*}
        \int &\frac{|T^+(k,l)|}{a+l^2} \, dl \, dk\\
            &\lesssim
                \int \frac{(1+(l-k)^2 \, |T^\pm(k,l)|}{(1+l^2)(1+(l-k)^2 }\ dl \, dk\\
            &\lesssim 
                \norm[L^2(\R^2)]{(1+(l-k)^2)T^\pm(k,l)} \, dl \, dk
    \end{align*}
    which is finite by Proposition \ref{prop:Tpm.L2} with $j=1$, $j'=0$.

    Second, noting that
    $$ \frac{|l|}{l^2+1} \lesssim 1,$$
    we may bound
    \begin{align*}
        \int \frac{|l(l-k)T^\pm(k,l)|}{(r^2+l^2)^2} \, dl \, dk
            &\lesssim
                \int |l-k| \, |T^\pm(k,l)| \, dl \, dk  \\
            &\lesssim \norm[L^2(\R^2)]{(1+(l-k)^2)(1+(l^2-k^2)^2)T^\pm}
    \end{align*}
    where we used the Schwarz inequality. The right-hand side is finite by Proposition \ref{prop:Tpm.L2} with $j=j'=1$.

    Third, we estimate
    \begin{align*}
        \int &\frac{|(l-k) \diff_lT^\pm(k,l)|}{a+l^2} \, dl \, dk \\
            & \int 
                \frac{|l-k|(1+(l-k)^2)}{(1+l^2)(1+(l-k)^2)}|\diff_l T^\pm(k,l)| \, dl \, dk\\
            &= I_1 + I_2
    \end{align*}
    where
    \begin{align*}
        I_1 &=  \int \frac{|l-k|}{(1+l^2)(1+(l-k)^2 )}  |\diff_l T^\pm(k,l)|\, dl \, dk,\\
        I_2 &=  \int \frac{|l-k| }{(1+l^2)(1+(l-k)^2) }|(l-k)^2 \diff_lT^\pm(k,l)| \, dl \, dk.
    \end{align*}
    By the Schwarz inequality, to bound $I_1$ and $I_2$, it suffices to show that
    \begin{align}
        \label{Tpm.est03.bound1}
        (l-k) \diff_lT^\pm  &\in L^2(\R^2),
        \intertext{and}
        \label{Tpm.est03.bound2}
        (l-k)^2 \diff_l T^\pm &\in L^2(\R^2).
    \end{align}
    We will use the decomposition \eqref{Tpm.split}.

    First, we prove \eqref{Tpm.est03.bound1}. We will bound in turn the terms involving $T_1$ and $T_2^\pm$. 
    
    To bound the term involving $T_1$, we compute
    \begin{align}
        \label{Tpm.est03.bd1.split}
        (l-k) \frac{\diff T_1}{\diff l}
            &=  (2\pi)^{-1} \int e^{i(l^2-k^2)y} 2(l-k)ly\tu(l-k;y) \, dy\\
            \nonumber
            &\quad -i(2\pi)^{-1} \int e^{i(l^2-k^2)y} (l-k) \frac{\diff \tu}{\diff l}(l-k;y) \, dy.
    \end{align}
    We will make the change of variables $l'=l-k$ and $k'=k$ and use the identity
    $$ \frac{1}{i} \frac{\diff}{\diff y}e^{il'(l'+2k)y}=l'(l'+2k') e^{il'(l'+2k')y }.$$
    After an integration by parts on the first right-hand term in \eqref{Tpm.est03.bd1.split}, we have (after dropping the primes)
    \begin{align}
        \label{Tpm.est03.bd1.1}
        l\frac{\diff T_1}{\diff l}(k,l+k)
            &=  i(2\pi)^{-1}
                    \int e^{il(l+2k)y} \tu(l;y) \, dy\\
            \nonumber
            &\quad  +i(2\pi)^{-1}\int e^{il(l+2k)y} \frac{\diff \tu}{\diff y}(l,y) \, dy\\
            \nonumber
            &\quad  +   (2\pi)^{-1} \int e^{il(l+2k)y} l^2 y \tu(l;y) \, dy\\
            \nonumber
            &\quad  +   \frac{i}{2\pi} \int  e^{il(l+2k)y} l \frac{\diff \tu}{\diff l}(l;y) \, dy
    \end{align}
    By \eqref{Tpm.est03.bd1.1}, Minkowski's integral inequality, and \eqref{Iv.bd}, to bound the $L^2_lL^2_k$ norm of $l \dfrac{\diff T_1}{\diff l}$, it suffices to have
    \begin{equation}
        \label{Tpm.est03.bd1.1.norms}
         \tu, \, \frac{\diff \tu}{\diff y}, \, l^2 y  \tu, \, l \frac{\diff \tu}{\diff l} \in L^2(|l|^{-1} \, dl \, dy). 
    \end{equation}
    \textblue{These conditions follow from inclusion relations (i)-(iii) in Lemma ~\ref{rem:Tpm.kl.L2}.}

    To finish the proof of \eqref{Tpm.est03.bound1}, we bound the term in \eqref{Tpm.est03.bound1} involving $T^\pm_2$. We compute
    \begin{align}
        \label{Tpm.est03.bd1.split2.pre}
        (l-k)\frac{\diff T_2^\pm}{\diff l}(k,l)
            &=  (2\pi)^{-1} \int e^{i(l^2-k^2)y}2(l-k) ly(\tu*\tmu^\pm_\sharp)(k,l-k,;y) \, dy\\
            \nonumber
            &\quad  -
                i(2\pi)^{-1} \int e^{i(l^2-k^2)y}(l-k)
                    \left( \frac{\diff \tu}{\diff l}*\tmu^\pm_\sharp \right) (k,l-k;y) \, dy.
    \end{align}
    As before we make the change of variables $l'=l-k, \, k'=k$, integrate by parts and apply the product rule for convolutions to conclude that
    \begin{align}
        \label{Tpm.est03.bd1.split2}
        l \frac{\diff T_2^\pm}{\diff l}(k,l+k)
            &=  \int e^{il(l+2k)y} (S_1 + S_2+S_3)(k,l,y) \, dy
        \intertext{where}
        \nonumber
        S_1 &= \frac{i}{2\pi} 
                \left[
                    (\tu*\tmu^\pm_\sharp) + 
                    y\left( \frac{\diff \tu}{\diff y}*\tmu^\pm_\sharp\right) + 
                    y \left((\tu)*\frac{\diff \tmu^\pm_\sharp}{\diff y}\right) 
                \right],\\
        \nonumber
        S_2 &=  \frac{1}{2\pi}
                \left[
                    y((l^2\tu)*\mu^\pm_\sharp) +
                    2y((l \tu)*(l\mu^\pm_\sharp)) +
                    y (\tu*(l^2\tmu^\pm_\sharp))
                \right],
        \intertext{and}
        \nonumber
        S_3 &=  -\frac{i}{2\pi}
                \left[
                    \left(
                 \left(l\frac{\diff \tmu}{\diff l}\right)*\tmu^\pm_\sharp 
                    \right) +
                    \left(\frac{\diff \tu}{\diff l}*(l\tmu^\pm_\sharp) \right)
                \right].
    \end{align}
    Applying Minkowski's inequality on \eqref{Tpm.est03.bd1.split2} and Young's inequality for convolutions, we can bound the $L^2_k L^2_l$ norm of \eqref{Tpm.est03.bd1.split2} provided 
    \begin{equation}
    \label{Tpm.est03.bd1.needs}
    \begin{aligned}
            \tu, \, y \diff_y \tu, \, l^2 y \tu, \, ly\tu, \, y\tu, \, l\diff_l \tu, \, \diff_l \tu &\in L^1(\R^2),\\
            \tmu^\pm_\sharp, \, l\tmu^\pm_\sharp, \, l^2 \tmu^\pm_\sharp, \, \diff_y \tmu^\pm_\sharp &\in X.
    \end{aligned}
    \end{equation}
    \textblue{These conditions 
follow from inclusion relations in Lemma ~\ref{rem:Tpm.kl.L2}.}
    Hence, assumptions \eqref{Tpm.est03.bd1.1.norms} and \eqref{Tpm.est03.bd1.needs} allow us to show \eqref{Tpm.est03.bound1}.

    To finish the proof of \eqref{Tpm.est03}, we need to prove \eqref{Tpm.est03.bound2}. Again, we estimate terms involving $T_1$ and $T_2^\pm$ separately.

    First, we compute
    \begin{align}
        \label{Tpm.est03.bd2.1.pre}
            (l-k)^2 \frac{\diff T_1}{\diff l}(k,l)
                &= (2\pi)^{-1} \int e^{i(l^2-k^2)y}2(l-k)^2 ly \tu(l-k;y) \, dy\\
                \nonumber
                &\quad -\frac{i}{2\pi}\int e^{i(l^2-k^2)y}(l-k)^2 \frac{\diff \tu}{\diff l}(l-k;y) \, dy.
    \end{align}
    Making the same change of variables and integrating by parts as before, we conclude that
    \begin{align*}
        l^2 \frac{\diff T_1}{\diff l}(k,l+k)
            &= \int e^{il(l+2k)y}(S_1+S_2) \, dy
    \intertext{where}
    S_1 &=  (2\pi)^{-1} l^2(l+2k) y \tu(l,y), \\
    S_2 &=  \frac{i}{2\pi}  
                \left(
                    l^2 \tu(l,y) + l^2 y \frac{\diff \tu}{\diff y}- l^2 \frac{\diff \tu}{\diff l}(l,y)
                \right).
    \end{align*}
    We can bound the left-hand side of \eqref{Tpm.est03.bd2.1.pre} provided
    \begin{equation}
        \label{Tpm.est03.bd2.1.norms}
            l^3 y \tu, \, l^2 \tu, \, y\diff_y \tu, \, l^2 \diff_l \tu 
            \in L^2(|l|^{-1} \, dl \, dy).
    \end{equation}
\textblue{These weighted $L^2$ bounds again follow from inclusion relations (i)-(iii) in Lemma ~\ref{rem:Tpm.kl.L2}.}

    Next, we compute
    \begin{align}
        \label{Tpm.est03.bd2.2.pre}
            (l-k)^2 \frac{\diff T^\pm_2}{\diff l}(k,l)
                &=  (2\pi)^{-1} 2l (l-k)^2 y (\tu*\tmu^\pm_\sharp)(k,l-k;y) \, dy \\
            \nonumber
                &\quad  - \frac{i}{2\pi}
                    \int e^{i(l^2-k^2)y}(l-k)^2
                        \left(\frac{\diff \tu}{\diff l}*\tmu^\pm_\sharp \right)(k,l-k;y) \, dy.
    \end{align}
    As before, we make a change of variables, integrate by parts, and apply the product rule for convolutions to obtain
    \begin{align*}
        l^2 \frac{\diff T^\pm_2}{\diff l}(k,l+k)
            &=  \int e^{il(l+2k)y}(S_1 + S_2 + S_3) \, dy
        \intertext{where}
        \nonumber
        S_1 &= \frac{1}{2\pi} y   \left[
                        (l^3 \tu)*\tmu^\pm_\sharp+3(l^2\tu)*(l\tmu^\pm_\sharp) \right.\\
        \nonumber 
            &\qquad \qquad +    \left. 
                        3(l\tmu)*(l^2 \tmu^\pm_\sharp) + \tu*(l^3 \tmu^\pm_\sharp) 
                    \right],  \\[5pt]
        \nonumber
        S_2 &=  \frac{i}{2\pi} y   \left[
                        (l\tmu)*\tmu^\pm_\sharp +
                        \tmu*(l\tmu^\pm_\sharp) +
                        \left(l \frac{\diff \tu}{\diff y} 
                        \right)*\tmu^\pm_\sharp 
                        \right. \\
        \nonumber
                    &\qquad \qquad +
                        \left. 
                        \left( \frac{\diff \tu}{\diff y}*(l\tmu^\pm_\sharp)\right) +
                        (l \tu)*\frac{\diff \tmu^\pm_\sharp}{\diff y} +
                        \tu*\left( l \frac{\diff \tmu^\pm_\sharp}{\diff y}
                        \right)
                    \right],
        \\[5pt]
        \nonumber
        S_3 &=  -\frac{i}{2\pi}
                    \left[
                        \left(l^2 \frac{\diff \tu}{\diff l }\right)*\tmu^\pm +
                        2 \left(l \frac{\diff \tu}{\diff l} \right)*(l\tmu^\pm_\sharp)
                        +
                        \frac{\diff \tu}{\diff l}*(l^2 \tmu^\pm_\sharp)
                    \right].
    \end{align*}
    Applying Minkowski's inequality, Young's inequality for convolutions, and H\"{o}lder's inequality on each of the terms in \eqref{Tpm.est03.bd2.2.pre}, we may bound the $L^2$ norm provided
    \begin{equation}
        \label{Tpm.est03.bd2.2.norms}
        \begin{aligned}
            l^3 y \tu, \, l^2 y \tu, \, ly \tu, \, y \tu, \, l \tu, \, \tu, \, ly \diff_y\tu, \, y\diff_y \tu, \, l^2 \diff_l \tu, l \diff_l \tu, \, \diff_l \tu &\in L^1(\R^2),\\
            \tmu^\pm_\sharp, \, l \tmu^\pm_\sharp, \, l^2 \tmu^\pm_\sharp, \, l^3 \tmu^\pm_\sharp, \, \diff_y \tmu^\pm_\sharp, \, l \diff_y \tmu^\pm_\sharp &\in X.
        \end{aligned}
    \end{equation}
   
    \medskip

    (iii): Again, we use the decomposition \eqref{Tpm.split} and estimate the derivatives of $T^\pm_1$ and $T^\pm_2$ separately. Since $E$ is a compact set, we can bound polynomials in $l$ and $k$ by constants depending on $E$.

    First,
    \begin{align}
        \label{Tpm.est04.T1}
            \frac{\diff^3 T_1}{\diff k \diff l^2}
                &=  \int e^{i(l^2-k^2)y}(S_1+S_2) \, dy\\
        \intertext{where}
        \nonumber
            S_1 &=  \frac{1}{2\pi}
                        \left[ 
                            4ky^2(2l^2y-i)\tu(l-k;y) -
                            2y(2lk) \frac{\diff^2 \tu}{\diff l^2}(l-k;y)
                        \right],\\[5pt]
        \nonumber
            S_2 &=  \frac{i}{2\pi} 
                        \left[
                            -4ly^2 (l+2k)\frac{\diff \tu}{\diff l}(l-k;y)
                            + \frac{\diff^3 \tu}{\diff l^3}(l-k;y)
                        \right].
    \end{align}
    Ignoring polynomials in $(l,k)$ we conclude that 
    \begin{align}
        \label{Tpm.est04.T1.bd}
            \norm[L^2(E)]{\frac{\diff^3 T_1}{\diff k \diff l^2}}
                &\lesssim    
                    \norm[L^{2,4}_yL^2_l]{\tu} +
                    \norm[L^{2,3}_y L^2_l]{ \frac{\diff \tu}{\diff l}} + 
                    \norm[L^{2,2}_y L^{2}_l]{ \frac{\diff^2 \tu}{\diff l^2}} +
                    \norm[L^{2,1}_y L^2_l]{\frac{\diff^3 \tu}{\diff l^3}}.    
    \end{align}

    Second, we have
   \begin{align}
       \label{Tp.est04.T2}
        \frac{\diff^3 T^\pm_2}{\diff l^2 \diff k}
            &=  \int e^{i(l^2-k^2)y}(S_1+S_2) \, dy\\
        \intertext{where}
        \nonumber
        S_1 &=   \frac{1}{2\pi }
                    \Biggl[
                        4ky^2 (2l^2y -i)(\tu*\mu^\pm_\sharp)(l-k;y)
                        \Bigr.\\
                \nonumber
                &\qquad 
                        -
                        \Bigl. 2y(2l+k)
                            \left(
                                \frac{\diff^2 \tu}{\diff l^2}*\mu^\pm_\sharp
                            \right)(l-k;y)
                            \Bigr. \\
                \nonumber
                &\qquad +   
                    \Bigl.
                        4ky^2(2l^2y-i)
                        \left(
                            \tu * \frac{\diff \tmu^\pm_\sharp}{\diff k}
                        \right)(l-k;y)
                  \Biggr],\\[5pt]
        \nonumber
        S_2 &=  \frac{i}{2\pi} 
                    \Biggl[
                        \left(
                            \frac{\diff^3 \tu}{\diff l^3}*\mu^\pm_\sharp
                        \right)(l-k;y) 
                        -4ly^2(l+2k)
                        \left(
                            \frac{\diff \tu}{\diff l}*\tmu^\pm_\sharp
                        \right)(l-k;y)
                    \Biggr.\\[5pt]
            \nonumber
            &\quad -
                    \Bigg.
                         8ly^2k 
                            \left(
                                \frac{\diff \tu}{\diff l}*
                                \frac{\diff \tmu^\pm_\sharp}{\diff k}
                            \right)(l-k;y)
                        - 2yk 
                            \left(
                                \frac{\diff^2 \tu}{\diff l^2}*
                                \frac{\diff \tmu^\pm_\sharp}{\diff k}
                            \right)(l-k;y)
                    \Biggr].
   \end{align}
   Ignoring polynomials in $(l,k)$ we conclude that
   \begin{align}
    \label{TP.est04.T1.bd}
       \norm[L^2(E)]{\frac{\diff^3 T_2^\pm}{\diff l^2 \diff k}}
        &\lesssim \left[ 
            \norm[L^{2,4}_y L^2_l]{\tu} + 
            \norm[L^{2,3}_y L^2_l]{\frac{\diff^2 \tu}{\diff l^2} } \right. \\
        &\quad \left. +
            \norm[L^{2,1}_y L^2_l]{\frac{\diff^3 \tu}{\diff l^3}} + 
            \norm[L^{2,3}_y L^2_l]{\frac{\diff \tu}{\diff l}}\right] 
            \norm[L^2_k L^2_l]{\tmu^\pm_\sharp}
        \nonumber
        \\
        &\quad + \left[ \norm[L^{2,5}_y L^2_l]{\tu} + \norm[L^{2,4}_y L^2_l]{\frac{\diff \tu}{\diff l} }+ \norm[L^{2,3}_y L^2_l]{\frac{\diff ^2 \tu}{\diff l^2} }\right] \times 
        \nonumber\\
        &\quad \sup_y \left((1 +|y|)^{-1}\norm[L^{2,1}_l L^2_k]{\frac{\diff \tmu^\pm_\sharp}{\diff k}}  \right)
        \nonumber
   \end{align}
   where we have controlled convolutions using 
   $$ \norm[L^2_l]{\tu * \tmu^\pm_\sharp} \lesssim \norm[L^2_l]{\tu} \norm[L^1_l]{\tmu^\pm_\sharp} \lesssim \norm[L^2_l]{\tu} \norm[L^{2,1}_l]{ \tmu^\pm_\sharp},$$ 
   and similarly for convolutions with derivatives.
   It follows that
   \begin{align}
       \label{TP.est04.T2.bd}
       \norm[L^2(E)]{\frac{\diff^3 T^\pm}{\diff k \diff l^2}}
            &\lesssim 
                \left[ 
            \norm[L^{2,4}_y L^2_l]{\tu} + 
            \norm[L^{2,3}_y L^2_l]{\frac{\diff^2 \tu}{\diff l^2} } \right. \\
        &\quad \left. +
            \norm[L^{2,1}_y L^2_l]{\frac{\diff^3 \tu}{\diff l^3}} + 
            \norm[L^{2,3}_y L^2_l]{\frac{\diff \tu}{\diff l}}\right] 
            \left(1+ \norm[L^2_k L^2_l]{\tmu^\pm_\sharp} \right)
            \nonumber \\
            &\quad + \left[ \norm[L^{2,5}_y L^2_l]{\tu} + \norm[L^{2,4}_y L^2_l]{\frac{\diff \tu}{\diff l} }+ \norm[L^{2,3}_y L^2_l]{\frac{\diff ^2 \tu}{\diff l^2} }\right] \times 
        \nonumber\\
        &\qquad \sup_y \left((1 +|y|)^{-1}\norm[L^{2,1}_l L^2_k]{\frac{\diff \tmu^\pm_\sharp}{\diff k}}  \right)
        \nonumber
   \end{align}
    where we used the facts that 
    \begin{align}
        \label{TP.est04.T2.sub1}
        \tmu^\pm_\sharp &\in X,
        \intertext{(see \eqref{tmu.est0.X}) and}
        \label{TP.est04.T2.sub2}
        \norm[L^2_k L^{2.1}_l]{\frac{\diff \tmu^\pm_\sharp}{\diff k}(\dotarg,\dotarg;y)}  &\lesssim (1+|y|).
    \end{align}
    (see \eqref{est: mu-k L2} and \eqref{est: mu-k  L2lk})
    \textblue{All estimates involving $\tu$ and  $\tmu^\pm_\sharp$ appearing on the right-hand side of \eqref{TP.est04.T1.bd}-
\eqref{TP.est04.T2.bd} are controlled by while the two factors involving are controlled by inclusion relations (i)-(iii) of 
 Lemma \ref{rem:Tpm.kl.L2}. See also Remark \ref{rem:Tpm.kl.L2'}.}

	\medskip

    (iv): Owing to the compact support of the cutoff function $\psi$, we may neglect polynomial factors of $l$ and $k$. We begin with several reductions.

    First, Proposition \ref{prop:Tpm.exist} implies that (a) is an $L^2$ function. Next, we note that \eqref{est: mu-k L2} implies that (b) is an $L^2$ function, which in turn implies that (e) is an $L^2$ function. Next, we note that \eqref{Tpm.est04} implies that (c) is an $L^2$ function since we may neglect the factor $(l-k)$. Finally, we may bound (d) as follows. Note that
    $$ \frac{\diff^2}{\diff l^2}(\psi T^\pm) = 
        \frac{\diff^2 \psi}{\diff l^2} T^\pm +
        2\frac{\diff \psi}{\diff l} \frac{\diff T^\pm}{\diff l} +
        \psi \frac{\diff^2 T^\pm}{\diff l^2}$$
    and the first term is bounded since (a) is bounded. Thus it suffices to bound local $L^2$ norms of $\dfrac{\diff T^\pm}{\diff l}$ and $\dfrac{\diff^2 T^\pm}{\diff l^2}$. Such bounds can be obtained by the same strategy as the proof of \eqref{Tpm.est04}.

    It remains to prove that (f) is an $L^2$ function. Since we may neglect polynomials in $l$ and $k$, and (c) is an $L^2$ function, it only remains to show that $\displaystyle\frac{\diff^2}{\diff l \diff k}(\psi T^\pm)$ is an $L^2$ function. We compute
    $$ \frac{\diff^2}{\diff l \diff k}(\psi T^\pm) =
        \frac{\diff^2 \psi}{\diff l \diff k} T^\pm + \frac{\diff \psi}{\diff k} \frac{\diff T^\pm}{\diff l} + 
        \frac{\diff \psi}{\diff l} \frac{\diff T^\pm}{\diff k} + 
        \psi \frac{\diff^2 T^\pm}{\diff k \diff l}.$$
    The first right-hand term is $L^2$ since (a) is $L^2$, and the third right-hand term is $L^2$ since (b) is $L^2$. It remains to bound local $L^2$ norms of $\diff T^\pm/\diff l$ and $\diff^2T^\pm/\diff l \diff k$.
    Again, we may use the same strategy of proof as used to prove \eqref{Tpm.est04}.

\end{proof}
\section{The Reconstruction Formula}

\label{sec:asy}

\textblue{In this section we prove the asymptotic estimates on the solution $u(t,x,y)$ asserted in Theorem \ref{thm:main}. Recall the decomposition of $u(t,x,y)$ into local and nonlocal parts:
\begin{align}
	\label{u.recon}
		u(t,x,y) 
			&= 	u_1(t,x,y) + u_2(t,x,y)
	\intertext{where}
	\label{u.local}
		u_1(t,x,y)
			&=	\frac{1}{\pi} \int e^{itS_0(k,l;\zeta,\eta)}
					i(l-k) f(k,l) \, dl \, dk	,
	\intertext{and}
	\label{u.non-local}
		u_2(t,x,y)
			&= \frac{1}{\pi} \frac{\diff}{\diff x}
					\int e^{itS_0(k,l;\zeta,\eta)}
						f(k,l)
						\left( \mu^l(l,x;y,t) -1 \right)
					\, dl \, dk
\end{align}
with $f(k,l)$ given by \eqref{de:f(k,l)}. In \eqref{u.local} and \eqref{u.non-local}, $S_0$ denotes the phase function
\begin{equation}
	\label{S0.def}
		S_0(k,l;\zeta,\eta)
			=	(l-k)\zeta - (l^2-k^2)\eta + 4(l^3-k^3).
\end{equation}
In \eqref{u.non-local}, $\mu^l = \mu^l(k,x;y,t)$ denotes the unique solution of the nonlocal Riemann-Hilbert problem
\begin{equation}
	\label{mul.RHP}
		\mu^l = 1 + \calC_T(\mu^l)
\end{equation}
with $\mu^l(\cdot,x;y,t) - 1 \in L^2_k(\R)$. 
Here
\begin{equation}
	\label{CT.def}
		\calC_T = C_+ \calT^- + C_-\calT^+
\end{equation}
where $C_\pm:L^2(\R) \to L^2(\R)$ are Cauchy projectors and
$\calT^\pm$ are the integral operators
\begin{equation}
	\label{calTpm.def}
	\calT^\pm(f)(k) = \int e^{itS_0} T^\pm(k,l)f(l) \, dl.
\end{equation}
The nonlocal Riemann-Hilbert problem considered here differs from the one considered in section \ref{subsec:NLRHP.def} (see Proposition \ref{prop:NLRHP.def}) by the addition of time evolution for the scattering data: the time-zero scattering data $T^\pm(k,l)$ is replaced by the time-evolved scattering data $T^\pm(k,l) e^{4it(l^3-k^3)}$, giving rise to the third term in \eqref{S0.def}.

Given the decomposition \eqref{u.recon}, our analysis naturally divides into two steps: the analysis of the local term $u_1$ in \S\ref{subsec:local}, and the analysis of the nonlocal term in \S \ref{subsec:nonlocal}. In each step, the analysis has three parts, dictated by the form of the phase function $S_0$: the cases of no critical points, nondegenerate critical points, and the transition region where the critical points degenerate to a single critical point and disappear.

The local term (treated in \S \ref{subsec:local}) involves an oscillatory integral whose amplitude is the scattering data 
\begin{equation}
	\label{t.def}
		f(k,l) = T^+(k,l) + T^-(k,l),
\end{equation}
where $T^\pm$ is given by \eqref{scatt.Tpm.def},
and whose phase function is $S_0$ (see \eqref{S0.def}).
For the analysis, we consider in turn the cases where the phase function $S_0$ has no stationary points (see \S \ref{subsec:KPI.loc.nsp}),
where the phase function has nondegenerate stationary phase points (see \S \ref{subsec:KPI.loc.nondegensp}), and in the transition region where the phase function passes a point of degenerate stationary phase (see \S \ref{subsec:KPI.loc.degensp}). The functions $T^\pm(k,l)$ are discontinuous along the line $l=k$, but the factor $(l-k)$  regularizes this singularity.

The nonlocal term (treated in \S \ref{subsec:nonlocal}) has the same phase function as the local term, but, by contrast, its amplitude now involves both the scattering data and the solution to the nonlocal Riemann-Hilbert problem. Moreover, the $x$-derivative in \eqref{u.non-local} yields two terms
\begin{align}
    \label{u.non-local.int}
    f(k,l)  &\diff_x 
            \left(
                e^{itS_0(k,l;\xi,\eta)}
                (\mu^l(l,x;y,t)-1)
            \right)\\
        \nonumber
        &=  e^{itS_0(k,l;\xi,\eta)} 
            f(k,l)
            \left[ 
                i(l-k) (\mu^l(l,x;y,t)-1) 
            \right]\\
        \nonumber
        &\quad + e^{itS_0(k,l;\xi,\eta)}
            f(k,l)
            \left[ 
                \frac{\diff \mu^l}{\diff x}(l,x;y,t) 
            \right].
\end{align}
The first right-hand term in \eqref{u.non-local.int} has a regularizing factor for the jump discontinuities of $f$, but the second term does not.

To analyze the nonlocal term, we consider the same three cases as before but our approach differs because, for the nonlocal Riemann-Hilbert problem, we can obtain $L^2$ estimates on the decay of $\mu^l(l,x;y,t)-1$ and $\diff_x \mu(l;x;y,t)$ in the same three regions as before (these estimates are, to our knowledge, entirely new), but we cannot implement stationary phase arguments in both variables. These estimates on the nonlocal Riemann-Hilbert problem are proved in \S \ref{subsec:NLRHP.est}. Our approach involves $L^2$ estimates in the $l$ variable for the nonlocal Riemann-Hilbert problem and stationary phase arguments in the $k$ variable for the scattering data. 

Section \ref{subsec:nonlocal} is organized as follows: first, we assume the estimates on $\mu^l$ and $\diff \mu^l/\diff x$ that we will need (see \eqref{norm.mul-1}--\eqref{norm.dmul.dx}). We defer the proof of \eqref{norm.mul-1}--\eqref{norm.dmul.dx} to section \ref{subsec:NLRHP.est}
%and \eqref{norm.I1.k}--\eqref{norm.I2.k} 
and give the proof of the asymptotic estimates for $u_2$ (see Theorem \ref{thm:u2}) assuming that \eqref{norm.mul-1}--\eqref{norm.dmul.dx} hold.
%assuming that \eqref{norm.mul-1}--\eqref{norm.dmul.dx} %and \eqref{norm.I1.k}--\eqref{norm.I2.k}
hold. As before, the proof of Theorem \ref{thm:u2} is broken up into three parts, corresponding to the respective cases of no stationary phase points for $S_0$, nondegenerate stationary phase points for $S_0$, and the transition region where the stationary phase points coalesce into a single stationary phase point and disappear. 

In section 5.3, we  analyze the nonlocal Riemann-Hilbert problem and prove \eqref{norm.mul-1}--\eqref{norm.dmul.dx} (see Lemma \ref{lemma:res.CT.est} and Proposition \ref{prop:RHP.bounds}). 
%These results then imply \eqref{norm.I1.k} (see Remark \ref{rem:I2.bd}) and
%\eqref{norm.I2.k}
%(see \ref{rem:I1-I2.nocp.est} )
}
%%%%%%%%%%%%%%%%%%%%%%%%%%%%%%%%%%%%%%%%%%%%%%%
% Material deleted and moved back to blue area
%%%%%%%%%%%%%%%%%%%%%%%%%%%%%%%%%%%%%%%%%%%%%%%
%% text deleted here also %%

Throughout this section, we assume that \textblue{$u \in \mathbf{E}_{1,w}$ and that smallness assumptions \eqref{u.C1} and \eqref{u.C2} hold.}
In our estimates, we will sometimes use the decomposition \eqref{Tpm.split} for $T^\pm$.

We will sometimes use the change of variables
\begin{equation}
	\label{kl.shift}
	(k,l)	\to (k+\eta/12,l+\eta/12)
\end{equation}
which changes the phase function to 
\begin{align}
	\label{S.def}
	S(k,l;a) &= 12a (l-k) + 4(l^3-k^3)
	\intertext{where}
	\label{a.def}
	a	&= \frac{1}{12}\left(\zeta- \frac{\eta^2}{12} \right)=\frac{1}{12}\left(\frac{x}{t}- \frac{y^2}{12t^2} \right).
\end{align}
The amplitude becomes
\begin{equation}
	\label{wTpm.def}
		\wT^\pm(k,l) = T^\pm(k+\eta/12,l+\eta/12). 
\end{equation}

\textblue{We will repeatedly use relations (i)-(iii) in Lemma \ref{rem:Tpm.kl.L2} to control weighted
derivatives of $\tu$. When there are estimates involving $T^\pm$ (and hence $\mu_\sharp^\pm$), we refer the reader to Section \ref{sec:mupm}.
%The mixed-norm estimates below use Minkowski's inequality to pass the $y$-integral
% $L^2_{k,l}$, Young's inequality for convolution in $l$, Plancherel in the
%oscillatory variable, and weighted Cauchy--Schwarz to convert weighted $L^2$
%bounds into $L^1$ bounds.
}

\subsection{The Local Term}
\label{subsec:local}

We consider the large-time asymptotics of
\begin{equation}
	\label{KPI.loc.u1+}
		u_1^\pm(t,x,y)
			= \frac{1}{\pi} 
				\int_{\R^2}
					e^{itS(k,l;a)}i(l-k) \wT^\pm(k,l) \, dl \, dk.
\end{equation}

The integral for $u_1^\pm$ is absolutely convergent under our hypotheses on $u$ owing to the estimate \eqref{Tpm.est0}
in Proposition \ref{prop:Tpm.kl.L2}. We will sometimes write $S(l)$ and $S(k)$ for the functions $S(l)=12al+4l^3$ and $S(k)=12ak+4k^3$.

We will consider the cases
\textblue{$a>0$, $a<0$, and $a \sim 0$} corresponding to no stationary phase, nondegenerate stationary phase, and degenerate stationary phase. \textblue{After the shift \eqref{kl.shift}, we have
$$\partial_l S(k,l;a)=12(a+l^2), \quad \partial_k S(k,l;a)=-12(a+k^2).$$ 

For $a>0$, there are no real stationary
points and the integrations by parts use the lower bound
$$\partial_lS=12(r^2+l^2)\geq 12r^2.$$ 

For $a=-r^2<0$, the stationary points occur
at $l=\pm r$ and $k=\pm r$, and the nondegenerate stationary phase estimates use
$\partial_l^2S=24l$ and $\partial_k^2S=-24k$. These quantities are bounded away
from zero only away from $r=0$, so the constants may deteriorate as $r\to0$; this
is why the transition region $|a|<\delta$ is treated separately.}  

In what follows, we will carry out the analysis for $u_1^+$ only since the analysis for $u_1^-$ is very similar.
Estimates that we will need on scattering data are given in \eqref{KPI.u1+.needed.1} and \eqref{KPI.u1+.needed.2}.

We will prove:

%\todo[size=\tiny]{Check degenerate cp result versus result for $u_2$}
\begin{theorem}
    \label{thm:u1}
    Suppose that $u \in \mathbf{E}_{1,w}$ and that conditions \eqref{u.C1}, \eqref{u.C2}, and \eqref{u.C2a} hold. Fix $\delta>0$. Then
	\begin{equation}
		\label{u1.bd}
			|u_1(t,x,y)| \lesssim 
				\begin{cases}
					o(t^{-2}), & a > \delta > 0,\\
					t^{-\frac23}, & |a| < \delta,\\
					t^{-1}, 	& a < -\delta < 0.	
				\end{cases}
	\end{equation}
    where the implied constants for the cases $|a|>\delta$ diverge as $\delta \downarrow 0$.
\end{theorem}

\begin{proof}
    An immediate consequence of Propositions \ref{prop:KPI.loc.u1.+}, \ref{prop:KPI:u1.-}, and 
    \ref{prop:KPI.loc.u1.0}.
\end{proof}

\subsubsection{No Stationary Phase}
\label{subsec:KPI.loc.nsp}

\begin{proposition}
\label{prop:KPI.loc.u1.+}
	Suppose that $a = r^2>0$, $u_0 \in \mathbf{E}_{1,w}$, Suppose that $u \in \mathbf{E}_{1,w}$ and that conditions \eqref{u.C1}, \eqref{u.C2}, and \eqref{u.C2a} hold.. Then
	\begin{equation}
		\label{KPI.loc.u1.+}
		u_1^+(t,x,y) = o_r(t^{-1}).
	\end{equation}
\end{proposition}

\begin{proof}
    \textblue{In what follows, let
    $$ S(l;a) = 12al + 4l^3, \quad S(k;a) = 12ak + 4k^3. $$}
	From \eqref{KPI.loc.u1.+}, after an integration by parts in the $l$-variable we have
	\begin{equation}
		u_1^+(t,x,y)
			=	-\frac{i}{\pi t} \int_\R e^{-itS(k;a)}\int_k^\infty e^{itS(k;l)} A(k,l;a) \, dl \, dk
	\end{equation}	
	where
	$$ A(k,l;a) = \diff_l \left( \frac{(l-k)\wT^+(k,l)}{S_1'(l;a)}\right),$$
	and 
	$$ S'(l;a) = 12(r^2+l^2). $$
   To show that $u_1^+(t,x,y) =o_r(t^{-1})$, consider the map 
	$$ L^1(\R^2) \ni B \mapsto I(t) = \int_{\R^2} e^{itS(k,l;a)} B(k,l)\, dl \, dk. $$
	This map is continuous uniformly in $t$ and, for $B \in C_0^\infty(\R^2)$, an integration by parts argument shows that $I(t) \to 0$ as $t \to \infty$. It follows by a density argument that the same is true for $B \in L^1(\R^2)$. Hence, to obtain the result, it suffices to show that $A(k,l;a) \in L^1(\R^2)$.
	This will follow if %the functions
	\begin{equation}
		\label{KPI.u1+.needed.1}
			\frac{\wT^+(k,l)}{l^2+ r^2}, \quad \frac{l-k}{l^2+r^2}\frac{\diff \wT^+(k,l)}{\diff l}, \quad
		\frac{l(l-k) \wT^+(k,l)}{(l^2+r^2)^2} \quad \in L^1(\R^2).
	\end{equation}

These estimates are proved in Proposition \ref{prop:Tpm.kl.L2}, equations \eqref{Tpm.est01}--\eqref{Tpm.est03}.

\textblue{The dependence on $r$ comes from the lower bound $S'(l;a) \geq 12r^2$, where $a=r^2$. Consequently, the constants in this no stationary phase estimate are uniform only when $a$ is bounded away from zero and may diverge as $a \downarrow 0$.}
\end{proof}

\begin{remark}
	\label{rem:KPI.u1+.nocp.kl}
	 \textblue{ When we carry out integration by parts, notice that the lower endpoint contribution at
$l=k$ vanishes because the amplitude contains the factor $l-k$. %And 
Part (iv) of Proposition~\ref{prop:Tpm.kl.L2} gives the required local regularity
near the diagonal $l=k$. For the endpoint at $+\infty$, note that by \eqref{Tpm.est0} and \eqref{KPI.u1+.needed.1}, the quotient
$(l-k)\wT^+(k,l)/S'(l;a)$ belongs to $L^1$, and its $l$-derivative belongs to
$L^1$ by \eqref{KPI.u1+.needed.1}. Hence the boundary contribution at infinity
is zero.} A similar proof can be given carrying out integration by parts with respect to $k$ rather than $l$. 
\end{remark}

\begin{remark}
    \label{rem:KPI.u1+.awaycp}
	If $a=-r^2$ instead, the phase function has critical points at $(k,l)=(\pm r, \pm r)$. Suppose that $\psi \in C^\infty(\R^2)$ with $\psi=0$ in balls of radius $r/2$ about each critical point and $\psi=1$ outside balls of radius $3r/4$ about each critical point. Since
	$$ \frac{\diff S}{\diff l}(k,l) = 12(l^2-r^2), \quad \frac{\diff S}{\diff k}(k,l) = -12(k^2-r^2),$$	
	it follows that, in the support of $\psi$, we have either
	$$ \left| \frac{\diff S}{\diff l}(k,l) \right| \geq \frac{3r^2}{4},$$
	or 
	$$ \left| \frac{\diff S}{\diff k}(k,l) \right| \geq \frac{3r^2}{4} $$
throughout the support of $\psi$. By Remark \ref{rem:KPI.u1+.nocp.kl},  we can modify the proof of Proposition \ref{prop:KPI:u1.-} to obtain an $o_r(t^{-1})$ estimate for the integral
	\begin{equation}
		\label{KPI.u1+.cp.away}
			I(t) = \frac{1}{\pi} \int_{\R^2} e^{itS(k,l;a)}\psi(k,l)i(l-k)\wT^+(k,l) \, dl \, dk.
	\end{equation}
\end{remark}

\subsubsection{Nondegenerate Stationary Phase}
\label{subsec:KPI.loc.nondegensp}

We suppose that $a=-r^2 < 0$ where $r>\delta>0$ for some $\delta$ that remains fixed throughout the analysis. Constants in the estimates depend on $\delta$ and diverge as $\delta \to 0$.  We will prove:

\begin{proposition}
	\label{prop:KPI:u1.-}
	Suppose that $a =-r^2$,  that $u_0 \in \mathbf{E}_{1,w}$ and that conditions \eqref{u.C1}, \eqref{u.C2}, and \eqref{u.C2a} hold. Then
	\begin{equation}
		\label{KPI.loc.u1.-}
		|u_1^+(t,x,y)| \lesssim_{\, r} t^{-1}.
	\end{equation}
\end{proposition}

\begin{proof}
	By Remark \ref{rem:KPI.u1+.awaycp}, it suffices to consider the integrals $I(t,x,y)$ of the form 
	$$ I(t,x,y)  = \frac{1}{\pi} 
				\int_{\R^2}
					e^{itS(k,l;a)}i(l-k) \psi(k,l) \wT^+(k,l) \, dl \, dk$$ 
	where $\psi \in C_0^\infty(\R^2)$ has support in a ball of radius $3r/4$ about one of the critical points $(-r,r)$, $(-r,-r)$, or $(r,r)$ (there is no contribution from the critical point $(r,-r)$ due to the Heaviside function). For the critical point $(\alpha r, \beta r)$ we may choose
		\begin{equation}
			\label{KPI.u1+.cp.cutoff}
				\psi(k,l) = \varphi_1(k) \varphi_2(l) 
		\end{equation}
		where
		$$ \varphi_1(k) = \varphi\left(\frac{k-\alpha r}{r} \right), \quad \varphi_2(l) =  \varphi\left(\frac{l-\beta r}{r} \right) ,
		$$
		and $\varphi(s)$ has support in $(-3/4,3/4)$. 
	We write
	$$ I(t,x,y)= \frac{1}{\pi} \int f(k) g(k) \, dk $$
	where
	\begin{align*}
		f(k) 	&= e^{-itS(k)},\\
		g(k)	&=	\int_k^\infty 
						e^{itS(l)}i(l-k)\psi(k,l) \wT^+(k,l) \, dl.
	\end{align*}	
	It follows that
	\begin{equation}
		\label{KPI.u1+.rep1}
			I(t,x,y) = \frac{1}{\pi} \int \widehat{f}(-\xi) \widehat{g}(\xi) \, d\xi 
	\end{equation}
	where
	\begin{align}
		\label{KPI.u1+.fft.def}
		\widehat{f}(\xi) &= \frac{\sqrt{2}}{(12t)^\frac13}\Ai\left((12t)^\frac23\left(a+ \frac{\xi}{12t}\right)\right),\\
		\label{KPI.u1+.gft.def}
		\widehat{g}(\xi) &= \frac{1}{\sqrt{2\pi}}
			\int e^{-i\xi k} \int_k^\infty e^{itS(l)}i(l-k) \psi(k,l) \wT^+(k,l) \, dl \, dk,
	\end{align}
	so that
	\begin{equation}
		\widehat{f}(-\xi) = \frac{\sqrt{2}}{(12t)^{\frac13}}
			\Ai\left((12t)^{\frac23}\left(a-\frac{\xi}{12t} \right)\right).
	\end{equation}
	The Airy function obeys the bounds
	\begin{equation}
		\label{Airy.left}
			\Ai(z) \lesssim (1+|z|)^{-\frac14}.
	\end{equation}
	Hence, if
	\begin{equation}
		\label{KPI.Ai.c}
			\left| a-\frac{\xi}{12t} \right| \geq c> 0,
	\end{equation}
	then the estimate
	\begin{equation}
		\label{KPI.Ai.est.t.pre}
			|f(-\xi)| \lesssim c^{-\frac14}   t^{-\frac12}
	\end{equation}
	holds.
	
	Condition \eqref{KPI.Ai.c} is fulfilled if $\xi > 12t(-r^2+c)$. If $c=r^2/2$, estimate \eqref{KPI.Ai.est.t.pre} implies that
	\begin{equation}
		\label{KPI.Ai.est.t}	
			f(-\xi) \lesssim r^{-1/2} t^{-\frac12}, \quad
			\xi \in (-6tr^2,\infty). 
	\end{equation}
	The complementary region $\xi<12t(-r^2+c)$ is the interval $(-\infty,-6tr^2)$. 	In this interval, we have
	\begin{equation}
		\label{KPI.Ai.est.t.bad}
			|f(-\xi)| \lesssim t^{-\frac13}.
	\end{equation}
	Write 
	$$
	\int \widehat{f}(-\xi) \widehat{g}(\xi) \, d\xi = I_1(t) + I_2(t)
	$$
	where
	\begin{align}
		I_1(t) &= \int_{\xi < -6tr^2} \widehat{f}(-\xi) \widehat{g}(\xi) \, d\xi,\\
		I_2(t)	&=	\int_{\xi >-6tr^2} \widehat{f}(-\xi) \widehat{g}(\xi) \, d\xi.
	\end{align}
	
	For $I_1$, we may may use \eqref{KPI.Ai.est.t.bad} to estimate
	\begin{align}
		\label{KPI.u1+.I1.est}
		|I_1(t)| & \lesssim  t^{-1/3} 
				\int_{\xi < -6tr^2} (1+\xi^2)^{-\frac12} (1+\xi^2)^\frac12 \widehat{g}(\xi) \, d\xi\\
				\nonumber
				&\lesssim  t^{-\frac13} (tr^2)^{-\frac12} 
					\left( \int_{\xi < -6tr^2} (1+\xi^2)|\widehat{g}(\xi)|^2 \, d\xi \right)^\frac12.
	\end{align}
	
	For $I_2$, using \eqref{KPI.Ai.est.t}, we conclude that
	\begin{align}
		\label{KPI.u1+.I2.est}
			|I_2(t)| 
				&\lesssim (rt)^{-\frac12} \int |\widehat{g}(\xi)| \, d\xi	\\
				\nonumber
				&\lesssim (rt)^{-\frac12} \left( \int_{\xi>-6tr^2} (1+\xi^2) |\widehat{g}(\xi)|^2 \, d\xi \right)^\frac12.
	\end{align}
	
	Owing to \eqref{KPI.u1+.I1.est} and \eqref{KPI.u1+.I2.est}, to complete the bounds it will suffice to show that
	$$ \norm[L^2(\R)]{(1+\xi^2)^\frac12 \widehat{g} } \simeq
		\norm[L^2]{g} + \norm[L^2]{\frac{\diff g}{\diff k}} \lesssim_{\, r} t^{-\frac12}. $$
	To this end, we write
	\begin{align}
	\label{KPI.u1+.g.id}
		g(k)	&=	\int h_1(l) h_2(l) \, dl
		\intertext{where}
		\label{KPI.u1+.h1.def}
		h_1(l)	&=	e^{itS(l)}H(l-k), \\
		\label{KI.u1+.h2.def}
		h_2(l)	&=	i(l-k)\psi(k,l)\wT^+(k,l). 
	\end{align}
	We have
	\begin{equation}
		\label{KPI.u1+.fmult.h}
			\int h_1(l) h_2(l) \, dl = \int \widehat{h_1}(-\xi) \widehat{h_2}(\xi) \, d\xi
	\end{equation}
	where $\widehat{h_1}$ and $\widehat{h_2}$ denote the Fourier transforms of $h_1$ and $h_2$ with respect to $l$. Observe that
	\begin{align}
		\label{KPI.u1+.h1.ft}
		\widehat{h_1}(\xi)	
			&=	(2\pi)^{-\frac12}\int_k^\infty e^{it(12al+4l^3)} e^{-i\xi l} \, dl,
		\intertext{and}
		\label{KPI.u1+.h2.ft}
		\widehat{h_2}(\xi)
			&=	(2\pi)^{-\frac12}\int e^{-i\xi l} \psi(k,l) i(l-k) \wT^+(k,l) \, dl.
	\end{align}
	First, we have the estimate (see Lemma \ref{lemma:partial.airy.ndg})
	\begin{equation}
		\label{KPI.g1.est}
			|\widehat{h_1}(\xi)| \lesssim_{\, r} t^{-\frac12} (1+|\xi|^2)^\frac12,
	\end{equation}
	so that
	\begin{align*}
		|g(k)|	&\leq 		t^{-\frac12} \int (1+|\xi|^2)^\frac12  |\widehat{h_2}(\xi)| \, d\xi\\
				&\lesssim 	t^{-\frac12} \left(\int  (1+|\xi|^2)^2 |\widehat{h_2}(\xi)|^2 \, d\xi \right)^\frac12 \\
				&\lesssim 	t^{-\frac12} \left(\norm[L^2]{h_2}+ \norm[L^2]{\frac{\diff^2 h_2}{\diff l^2}}\right) 
	\end{align*}
	giving finally
	\begin{equation}
	\label{KPI.u1+.cp.g.est}	
		\norm[L^2]{g} \lesssim t^{-\frac12} 
			\left( 
				\norm[L^2_k L^2_l]{h_2} + 
				\norm[L^2_k L^2_l]{\frac{\diff^2 h_2}{\diff l^2}} 
			\right).
	\end{equation}

	We can find an analogous representation for $\diff g/\diff k$, taking some care since the Heaviside function $H(l-k)$ is not smooth in $k$. Writing
	$$
		g(k) =  \int_k^\infty e^{itS(l)} i(l-k)\psi(k,l) \wT^+(k,l) \, dl ,
	$$		
	we recover
	\begin{align}
		\label{KPI.u1+.cp.g.k}
		\frac{\diff g}{\diff k}(k)
			&=	-i \int_k^\infty e^{itS(l)} \psi(k,l) \wT^+(k,l) \, dl \\
			\nonumber
			&\quad + 
				i \int_k^\infty
					e^{itS(l)}(l-k) 
						\frac{\diff}{\diff k}
							\left( \psi(k,l) \wT^+(k,l) \right) 
				\, dl \\
			\nonumber
			&=	-i \int e^{itS(l)}H(l-k)\psi(k,l) \wT^+(k,l) \, dl \\
			\nonumber
			&\quad + 
				i \int 
					e^{itS(l)}H(l-k) (l-k) 
						\frac{\diff}{\diff k}
							\left( \psi(k,l) \wT^+(k,l) \right) 
				\, dl 	.
	\end{align}
	As before we view each of the two integrals in the last two lines of \eqref{KPI.u1+.cp.g.k} as the product of $h_1(l) = e^{itS(l)}H(l-k)$ with a function $h_{2,1}$ or $h_{2,2}$ that depends on $k$ and $l$, where
	\begin{align*}
		h_{2,1}(l)	&=	\psi(k,l) \wT^+(k,l)	, \\
		h_{2,2}(l)	&=	(l-k) \frac{\diff}{\diff k}
							\left( \psi(k,l) \wT^+(k,l) \right).
	\end{align*} We use \eqref{KPI.u1+.fmult.h} and \eqref{KPI.g1.est} to conclude that
	%\todo[size=\tiny]{Complete this estimate}
	\begin{align}
		\label{KPI.u1+.cp.gk.est}
			\norm[L^2]{\frac{\diff g}{\diff k}} 
				&\lesssim t^{-\frac12}
						\left(\,
							\norm[L^2_kL^2_l]{h_{2,1}} + 
							\norm[L^2_kL^2_l]{\frac{\diff^2 h_{2,1}}{\diff l^2}} 
						\right.\\
				\nonumber
				&\qquad \qquad +  	\left. 
							\norm[L^2_kL^2_l]{h_{2,2}} + 
							\norm[L^2_kL^2_l]{\frac{\diff^2 h_{2,2}}{\diff l^2}}
						\right) .
	\end{align}
\end{proof}

\begin{remark}
	\label{rem:u1+.cp.est}
	The bounds \eqref{KPI.u1+.cp.g.est} and  \eqref{KPI.u1+.cp.gk.est} are effective if we have $L^2_k L^2_l$  bounds on
	\begin{align}
	\label{KPI.u1+.needed.2}
		\psi \wT^+, 
		&& \frac{\diff}{\diff k}(\psi \wT^+),
		\\ 
		 	\frac{\diff^3}{\diff k \diff l^2}(\psi \wT^+),
		&&  \frac{\diff^2 }{\diff l^2}(\psi \wT^+), \nonumber
		\\
		 (l-k) \frac{\diff}{\diff k}(\psi \wT^+),
		 &&\frac{\diff^2}{\diff l^2}
		 	\left((l-k) \frac{\diff}{\diff k}(\psi \wT^+)\right).\nonumber
\end{align}
These bounds are proved in Proposition \ref{prop:Tpm.kl.L2}(iv).
\end{remark}

\subsubsection{Degenerate Stationary Phase}
\label{subsec:KPI.loc.degensp}

We now consider the case $|a| \leq c$. In this case, we have a refined estimate
for $\widehat{h_1}(\xi)$ (see \eqref{KPI.u1+.h1.def} and \eqref{KPI.u1+.h1.ft}):
\textblue{\begin{equation}
	\label{KPI.loc.u1+.h1.ft.dcp}
		\left| \widehat{h_1}(\xi) \right| \lesssim t^{-\frac13}, \quad |a| \le c
\end{equation}}
as follows from Lemma \ref{lemma:partial.airy.dg}. \textblue{The implicit constant in this estimate is independent of $c$, hence remains uniform as $c \downarrow 0$.}
In contrast to \eqref{KPI.g1.est}, this estimate is independent of $\xi$ but has a lower decay rate, corresponding to the fact that at  $a=0$, the phase function $S(l)$ has a single degenerate critical point. We will prove:

\begin{proposition}
	\label{prop:KPI.loc.u1.0}
	Suppose that $|a| \leq c$ for a fixed \textblue{$0<c\le1$}, that  $u_0 \in \mathbf{E}_{1,w}$, and that conditions \eqref{u.C1}, \eqref{u.C2}, and \eqref{u.C2a} hold. Then
	\textblue{\begin{equation}
		\label{KPI.loc.u1.0}
		u_1^+(t,x,y) \lesssim t^{-\frac23}.	
	\end{equation}}
\end{proposition}

\begin{proof}
Repeating the strategy of Subsection \ref{subsec:KPI.loc.nondegensp} we choose a cutoff function $\varphi \in C_0^\infty(\R^2)$ so that $\varphi=1$ for \textblue{$k^2+l^2 < 4$} and $\varphi=0$ for \textblue{$k^2+l^2>16$}. It suffices to estimate
$$ I(t,x,y)= \frac{1}{\pi} \int f(k)g(k) \, dk $$
where
\begin{align*}
	f(k) &= e^{-itS(k)}, \\
	g(k) &= \int_k^\infty e^{itS(l)}i(l-k) \varphi(k,l) \wT^+(k,l) \, dl.
\end{align*}
We then have
$$ I(t,x,y) = \frac{1}{\pi} \int \widehat{f}(-\xi) \widehat{g}(\xi) \, d\xi $$
where $\widehat{f}$ and $\widehat{g}$ are given by \eqref{KPI.u1+.fft.def}--\eqref{KPI.u1+.gft.def} (where $\psi$ is replaced by $\varphi$). Owing to the assumption $|a| \leq c$ we only have the bound
\begin{equation}
	\label{KPI.u1+.dgp.f}
	\left|\widehat{f}(-\xi) \right| \lesssim t^{-1/3}, \quad |a| \leq c. 	
\end{equation}
Since the Airy function estimate is uniform in $\xi$ we conclude that
\textblue{$$ |I(t,x,y)| \lesssim t^{-\frac13} \norm[L^2]{(1+\xi^2)^\frac12 \widehat{g}(\xi)}  \quad |a| \leq c \le1,$$}
so, once again, it suffices to show that
$$ \norm[L^2]{g}+ \norm[L^2]{\frac{\diff g}{\diff k}} \lesssim t^{-\frac13}. $$
Writing $g$ as in \eqref{KPI.u1+.g.id} (with $\psi$ replaced by $\varphi$), we can repeat the argument in Subsection \ref{subsec:KPI.loc.nondegensp} to conclude that 
\begin{equation}
		\label{KPI.u1+.dcp.g.est}
			\norm[L^2]{g}
				\lesssim t^{-\frac13}
					\left( \norm[L^2]{h_2} + \norm[L^2]{\frac{\diff h_2}{\diff l}} \right),
\end{equation}
and
\begin{align}
		\label{KPI.u1+.dcp.gk.est}
			\norm[L^2]{\frac{\diff g}{\diff k}} 
				&\lesssim t^{-\frac13}
						\left(\,
							\norm[L^2_kL^2_l]{h_{2,1}} + 
							\norm[L^2_kL^2_l]{\frac{\diff h_{2,1}}{\diff l}} 
						\right.\\
				\nonumber
				&\qquad \qquad +  	\left. 
							\norm[L^2_kL^2_l]{h_{2,2}} + 
							\norm[L^2_kL^2_l]{\frac{\diff h_{2,2}}{\diff l}}
						\right) \\
				\nonumber
\end{align}

where
\begin{align*}
	h_2(k,l)		&= h_{2,1}(k,l) + h_{2,2}(k,l),\\
	h_{2,1}(k,l) 	&= -i \varphi(k,l) \wT^+(k,l), \\
	h_{2,2}(k,l) 	&= i(l-k) \frac{\diff}{\diff k}\left( \varphi(k,l) \wT^+(k,l) \right).
\end{align*}
The lower order of $l$-derivatives as contrasted to \eqref{KPI.u1+.cp.gk.est} is due to the fact that the estimate \eqref{KPI.loc.u1+.h1.ft.dcp} is uniform in $\xi$. 

\end{proof}

\begin{remark}
	\label{rem:KPI.u1+.dcp.est}
	The estimate \eqref{KPI.u1+.dcp.g.est} and \eqref{KPI.u1+.dcp.gk.est} will be effective under the same assumptions as in Remark \ref{rem:u1+.cp.est}.
\end{remark}

\subsection{The Nonlocal Term}

\label{subsec:nonlocal}

In this subsection, we will determine the large-time asymptotic behavior of the nonlocal term
\begin{align} 
	\label{u.recon.nonloc}
		u_2(t,x,y)
			&= u_{2,1}(t,x,y) + u_{2,2}(t,x,y)
	\intertext{where}
	\label{u.recon.21}
	u_{2,1}(t,x,y)
		&=	\frac{1}{\pi} 
					\int 
						e^{itS_0(k,l;\zeta,\eta)} 
						i(l-k) f(k,l)	
						(\mu^l(l,x;y,t)-1)
					\, dl \, dk,
	\intertext{and}
	\label{u.recon.22}
	u_{2,2}(t,x,y)
		&=	\frac{1}{\pi}
					\int
						e^{itS_0(k,l;\zeta,\eta)}
						f(k,l)
						\frac{\diff \mu^l}{\diff x}(l,x;y,t)
					\, dl \, dk.
\end{align}
The function $\mu^l$ solves the nonlocal Riemann-Hilbert problem
\eqref{mul.RHP}. 

%\textblue{In this subsection, the consequences of $u\in \mathbf E_{1,w}$ enter in two
%ways. When an estimate is reduced directly to weighted $L^1$ or $L^2$ norms of
%$\tu$, these norms are controlled by Lemma~\ref{lem:E1w.bookkeeping}. When an
%estimate is stated in terms of norms of the scattering data $T^\pm$, we use
%Proposition~\ref{prop:Tpm.L2} and the derivative estimates proved in
%Proposition~\ref{prop:Tpm.kl.L2}; the hypotheses of these estimates are verified
%using Lemma~\ref{lem:E1w.bookkeeping}, together with the bounds on
%$\tmu^\pm_\sharp$ from Lemma~\ref{lemma:tmu.smooth} and, where needed, the
%bounds on $\diff_k\tmu^\pm_\sharp$ from Proposition~\ref{prop: mu-k}. The
%mixed-norm estimates below use Minkowski's inequality, Plancherel's theorem,
%Young's inequality in the convolution variable, and weighted Cauchy--Schwarz.}

Recall also that the case $a>0$ corresponds to no stationary points for $S$, and the case $a<0$ corresponds to nondegenerate stationary phase points for $S$.
We also define
\begin{align*}
	S(l)	&=	12al + 4l^3,\\
	S(k)	&=	12ak + 4k^3,
\end{align*}
and we will set
$$ P(l) = S'(l) = 12(a+l^2). $$
Finally, we write
\begin{equation}
	\label{wT.def}
		\wT^\pm(k,l) = T^\pm(k+\eta/12,l+\eta/12). 
\end{equation}

To analyze \eqref{u.recon.21} and \eqref{u.recon.22}, we will first \textblue{state} that
\begin{align}
		\label{norm.mul-1}
		\norm[L^2_l]{\mu^l - 1}
			&\lesssim	
				\begin{cases}
					t^{-1}, & a > \delta > 0, \\
					t^{-\frac13}, & |a| < \delta,\\
					t^{-\frac12}, & a < -\delta < 0,
				\end{cases}
		\intertext{and}
		\label{norm.dmul.dx}
		\norm[L^2_l]{\dfrac{\diff \mu^l}{\diff x}}
			&\lesssim	
				\begin{cases}
					t^{-1}, & a > \delta > 0, \\
					t^{-\frac13}, & |a| < \delta,\\
					t^{-\frac12}, & a < -\delta < 0.
				\end{cases}
\end{align}
%% PAP did a bit of wordsmithing here
\textblue{ whose proof will be given in Proposition \ref{prop:RHP.bounds} of \S \ref{subsec:NLRHP.est},
%which is a consequence of 
based on the estimates in 
Lemma \ref{lemma:CT.1} and Lemma \ref{lemma:CdTdX.1}  in that section}. The implied constants for the cases $|a|>\delta$ diverge as $\delta \downarrow 0$.
Estimates with this condition help quantify the behavior of $u(t,x,y)$ in the transition region between the presence and absence of critical points.
\textblue{
We will then carry out a stationary phase analysis of the integrals with respect to the $k$ variable:
\begin{align}
	\label{S.I1.def}
	I_1^\pm(t)	&=	\iint e^{-itS(k)}i(l-k)\wT^\pm(k,l) e^{itS(l)}(\mu(l;x,y,t)-1) \, dldk,\\
	\label{S.I2.def}
	I_2^\pm(t)	&=	\iint e^{-itS(k)}\wT^\pm(k,l) e^{itS(l)}\frac{\partial}{\partial x}\mu(l;x,y,t)\, dldk
\end{align} 
which yields the bounds
\begin{align}
	\label{norm.I1.k}
		\norm[L^\infty]{I^\pm_1(t;x,y)}
			&\lesssim	
				\begin{cases}
					t^{-2},	&	a> \delta > 0,\\
					t^{-\frac23}, & |a| < \delta,\\
					t^{-1} &  a< -\delta < 0,	
				\end{cases}
		\intertext{and}
	\label{norm.I2.k}
		\norm[L^\infty]{I^\pm_2(t;x,y)}
			&\lesssim	
				\begin{cases}
					t^{-2},	&	a> \delta > 0,\\
					t^{-\frac23}, & |a| < \delta,\\
					t^{-1} &  a< -\delta < 0	
				\end{cases}
\end{align}
where the implied constants for the cases $|a|>\delta$ diverge as $\delta \downarrow 0$.
%%% PAP %%% I removed referece to (5.59) and (5.60) %%%%
Given the bounds \textblue{\eqref{norm.mul-1} and \eqref{norm.dmul.dx} on $\mu^l-1$ and $\diff \mu^l/\diff x$,}
%%, \eqref{norm.I1.k}, and \eqref{norm.I2.k}, 
we can prove:
}
\begin{theorem}
	\label{thm:u2}
	Suppose that $u \in \mathbf{E}_{1,w}$ and that conditions \eqref{u.C1}, \eqref{u.C2}, and \eqref{u.C2a} hold. Fix $\delta>0$. Then
	\begin{equation}
		\label{u2.bd}
			|u_2(t,x,y)| \lesssim 
				\begin{cases}
					t^{-2}, & a > \delta > 0,\\
					t^{-\frac23}, & |a| < \delta,\\
					t^{-1}, 	& a < -\delta < 0.	
				\end{cases}
	\end{equation}
    where the implied constants for the cases $|a|>\delta$ diverge as $\delta \downarrow 0$.
\end{theorem}

\begin{proof}
\textblue{
Notice that $\mu^l$ is not a function of $k$. By assuming \eqref{norm.mul-1}-\eqref{norm.dmul.dx}, the bounds on \eqref{S.I1.def}--\eqref{S.I2.def} follow from stationary phase analysis in $k$ and \eqref{u2.bd} will follow as a direct consequence. } We will also recall
\begin{align}
	\label{Airy.id.again}
    \frac{1}{\sqrt{2\pi}} \int_\R e^{-i\xi k}e^{-it(12ak+4k^3)} dl = \frac{\sqrt{2\pi}}{(12t)^{\frac13}} \Ai\left((12t)^\frac23(a -\xi/12t )\right)
\end{align}
and the fact that 
\begin{equation}
    \left\vert \eqref{Airy.id.again} \right\vert\lesssim \begin{cases}
					t^{-\frac13}, & |a| < \delta,\\
					t^{-1/2}, 	& a < -\delta < 0.	
				\end{cases}
\end{equation}

We divide our proof into three parts.
%\begin{itemize}
    %\item[(i)] 

\medskip

    (i)  $a:=r^2 > \delta > 0$.
For simplicity we only deal with the $+$ part. First
notice that
\begin{align}
\label{diffT-k1}
  & \frac{\partial}{\partial k}\left( \int_{\mathbb{R}} (l-k)\widetilde{u} ( l-k ; y')  e^{i (l^2- k^2) y'}d y'\right)\\
  \nonumber
   &\quad =-\int_{\R}  \widetilde{u}(l-k ; y')e^{i (l^2- k^2) y'} d y'\\
           \nonumber
          &\qquad  + \int_{\R}  (l-k)  \frac{\partial}{\partial k}\widetilde{u}(l-k, y')   e^{i (l^2- k^2) y'}d y'\\
           \nonumber
 &\qquad  - 2i\int_{\R}   (l-k) \widetilde{u}(l-k, y')kye^{i (l^2- k^2) y'} d y'.
          \end{align}
\begin{align}
\label{diffT-k}
  & \frac{\partial}{\partial k}\left( \int_{\mathbb{R}} (l-k)\widetilde{u} *\tmu_\sharp^{ +}(k, l-k ; y')  e^{i (l^2- k^2) y'}d y'\right)\\
  \nonumber
   \nonumber
   &\quad =-\int_{\R} \int_\R  \widetilde{u}(l-l'-k ; y')\tmu_\sharp^+(k,l', y') dl'e^{i (l^2- k^2) y'} dl' d y'\\
           \nonumber
           &\qquad +\int_{\R} \int_\R (l-k) \frac{\partial}{\partial k} \left[\tmu_\sharp^+(k, l', y') \right]\widetilde{u}(l-l'-k ; y')dl'e^{i (l^2- k^2) y'}dl' d y'\\
           \nonumber
           &\qquad  + \int_{\R} \int_\R  (l-k)  \frac{\partial}{\partial k}\widetilde{u}(l-l'-k, y') \tmu_\sharp^+(k, l' ; y') dl' e^{i (l^2- k^2) y'}d l'dy'\\
           \nonumber
 &\qquad  - 2i\int_{\R} \int_\R  (l-k) \widetilde{u}(l-l'-k, y')\tmu_\sharp^+(k, l' ; y')dl'ky'e^{i (l^2- k^2) y'} d l'dy'.
          \end{align}
    
Setting 
$$I_1^+(t;x,y)=I_{1,1}^+(t;x,y)+I_{1,2}^+(t;x,y)+I_{1,3}^+(t;x,y),$$
integrating by parts in $k$ leads to

\begin{align}
  \bigl| I^+_{1,1} &(t;x,y)\bigr|\\
  \nonumber
  &= \left\vert \frac{1}{it}\int_{\mathbb{R}}\int_k^\infty \frac{e^{itS_0(k,l; \zeta, \eta)}}{\partial_k S(k)}  \frac{\partial}{\partial k}\left( \int_{\mathbb{R}} (l-k)\widetilde{u}( l-k ; y')  e^{i (l^2- k^2) y'}d y'\right) (\mu^l(l, x,y) -1) dl dk\right\vert \\
  \nonumber
  &\lesssim \frac{1}{t}     \left(\Norm{\widetilde{u}}{L^{2,-1}_l L^{2}_y}+\Norm{\frac{\partial}{\partial l}\widetilde{u}}{L^{2,1}_l L^{2}_y} + \Norm{\frac{\partial}{\partial l}\tu}{L^{2,1}_l L^{2,1}_y}\right) \times \Norm{\mu^l-1}{L^2_l}\\
  \nonumber
  &=\mathcal{O}(t^{-2}).
\end{align}

\begin{align}
\label{est:I12i}
  \bigl\vert I_{1,2}^+&(t;x,y)\bigr\vert\\
  \nonumber
  &= \Biggl\vert \frac{1}{it}\int_{\mathbb{R}}\int_k^\infty \frac{e^{itS_0(k,l; \zeta, \eta)}}{\diff_k S(k)}  
    \frac{\partial}{\partial k}\left( \int_{\mathbb{R}} (l-k)\widetilde{u}* \tmu_\sharp^+( l-k ; y')  e^{i (l^2- k^2) y'}d y'\right) (\mu^l(l, x,y) -1) dl dk\Biggr\vert \\
  \nonumber
  &\lesssim \frac{1}{t}  \Bigl(\Norm{\tmu_\sharp^+}{L^2_k L^{2,1}_l}\Norm{\widetilde{u}}{L^{2,1}_l L^{2,2}_y}+\Norm{\frac{\partial}{\partial k}\tmu_\sharp^+}{L^{2,1}_l}\Norm{\widetilde{u}}{L^{2,1}_l L^{2,2}_y}\\
  \nonumber &\quad + \Norm{\frac{\partial}{\partial l}\tu}{L^{2,1}_l L^{2,1}_y}\Norm{\tmu_\sharp^+}{L^2_k L^{2,1}_l }\Bigr)\\
  \nonumber
  & \quad \times \Norm{\mu^l-1}{L^2_l}\\
  \nonumber
  &=\mathcal{O}(t^{-2}).
\end{align}
Finally we have that
\begin{align}
    \bigl\vert I_{1,3}^+&(t;x,y)\bigr\vert\\
    \nonumber
    &=\left\vert \frac{1}{it}\int_{\mathbb{R}}\int_k^\infty  \frac{\partial}{\partial k}\left(\frac{1}{\partial_k S(k)}\right)  e^{itS_0(k,l; \zeta, \eta)} \right. \\
    \nonumber
    &\left.
    \qquad \qquad 
    \left( 
        \int_{\mathbb{R}} (l-k)\widetilde{u}* \tmu^+( l-k ; y')  
            e^{i (l^2- k^2) y'}d y'
    \right) (\mu^l(l, x,y) -1) dl dk
    \right\vert\\
    \nonumber
    &\lesssim \mathcal{O}(t^{-2}).
\end{align}
\textblue{ We would like to point out that when integrating by parts, the $k$  derivative with respect to $\int_k^\infty$ vanishes because of the $l-k$ factor. Also
the boundary term has the following form: 
\begin{equation}
\label{int-bdry}
    \int_k^\infty \frac{e^{itS_0(k,l; \zeta, \eta)}}{\partial_k S(k)} \left( \int_{\mathbb{R}} (l-k)\widetilde{u} *\tmu^{ +}(k, l-k ; y')  e^{i (l^2- k^2) y'}d y'\right) (\mu^l(l, x,y) -1) dl
\end{equation}
which vanishes as $k\to \pm \infty$. To see this, we make use of Cauchy-Schwarz inequality and Young's inequality to deduce
\begin{align}
    \left\vert\eqref{int-bdry} \right\vert &\lesssim \frac{1}{|12k^2+12a|}\Norm{\mu^l-1}{L^2_l}\Norm{\int_{\mathbb{R}} (l-k)\widetilde{u} *\tmu^{ +}(k, l-k ; y')  e^{i (l^2- k^2) y'}d y'}{L^2_l}\\
    \nonumber 
   & \lesssim \frac{1}{|12k^2+12a|} \Norm{\mu^l-1}{L^2_l} \left(\Norm{\tu}{L^{2,1}_lL^1_y}+ \Norm{\mu^+_\sharp}{L^{2,1}_l}\Norm{\tu}{L^{1}_lL^1_y}\right)
    \end{align}
which goes to $0$ as $k\to\pm\infty$.} So for $a>\delta>0$, we arrive at 
\begin{equation}
    \left\vert  I_{1}(t; x,y)\right\vert\lesssim t^{-2}.
\end{equation}
For $I_2^+(t;x,y)$, when integrating by parts,  there is another term of the form
\begin{equation}
\label{diffk}
    \frac{1}{it}\int_{\mathbb{R}}\frac{e^{itS_0(k,l; \zeta, \eta)}}{\diff_k S(k)}  
    \left( \int_{\mathbb{R}} \widetilde{u}* \tmu^+( 0 ; y')  d y'\right)\frac{\partial}{\partial x} \mu^l(k, x,y)  dk
\end{equation}
which \textblue{obeys the estimate}
$$\left\vert\eqref{diffk} \right\vert=\mathcal{O}(t^{-2})$$ by directly applying Cauchy-Schwarz inequality.
Thus, we can similarly deduce that 
\begin{align}
    \bigl\vert I^+_{2}&(t; x,y)\bigr\vert \\
    \nonumber
    &\lesssim  \left\vert \frac{1}{it}\int_{\mathbb{R}}\int_k^\infty \frac{e^{itS_0(k,l; \zeta, \eta)}}{\partial_k S_0}  \frac{\partial}{\partial k}\left( \int_{\mathbb{R}}\widetilde{u}( l-k ; y')  e^{i (l^2- k^2) y'}d y'\right) \frac{\partial}{\partial x}(\mu^l(l, x,y) ) dl dk\right\vert\\
    \nonumber
    &+ \left\vert \frac{1}{it}\int_{\mathbb{R}}\int_k^\infty \frac{e^{itS_0(k,l; \zeta, \eta)}}{\partial_k S_0}  \frac{\partial}{\partial k}\left( \int_{\mathbb{R}}\widetilde{u}*\tmu_\sharp^+( l-k ; y')  e^{i (l^2- k^2) y'}d y'\right) \frac{\partial}{\partial x}(\mu^l(l, x,y) ) dl dk\right\vert\\
    \nonumber
    &+\left\vert\frac{1}{it}\int_{\mathbb{R}}\frac{e^{itS_0(k,l; \zeta, \eta)}}{\diff_k S(k)}  
    \left( \int_{\mathbb{R}} \widetilde{u}* \tmu^+( 0 ; y')  d y'\right)\frac{\partial}{\partial x} \mu^l(k, x,y)  dk\right\vert\\
    \nonumber
    &+ \left\vert \frac{1}{it}\int_{\mathbb{R}}\int_k^\infty  \frac{\partial}{\partial k}\left(\frac{1}{\partial_k S(k)}\right)  e^{itS_0(k,l; \zeta, \eta)}\left( \int_{\mathbb{R}}\widetilde{u}* \tmu^+( l-k ; y')  e^{i (l^2- k^2) y'}d y'\right)\frac{\partial}{\partial k} \mu^l(l, x,y) dl dk \right\vert\\
    \nonumber
    &\lesssim t^{-2}.
\end{align}
So for $a>\delta>0$, we arrive at 
\begin{equation}
    \left\vert  I^+_{2}(t; x,y)\right\vert\lesssim t^{-2}.
\end{equation}
%\item[(ii)] 

\medskip

(ii) $a:=-r^2<-\delta <0$. Again, for simplicity, we only deal with the $+$ part.
We first introduce cutoff functions $\psi_\pm$ with disjoint supports that localize near critical points of $S_0$. More specifically, 
recall that for $a = r^2>0$, we define
\begin{align}
\label{cutoff-k}
    \psi_\pm(k)=\begin{cases} 1, & |k\pm r|\leq r/2,\\
                          0, & |k\pm r| \geq 3r/2.
    \end{cases}
\end{align}
We also set 
	$$ \psi_\infty= 1 - \psi_+ - \psi_- .$$
We now have
\begin{align*}
		I^+_{1}(t; x,y)
			&=	I_{1,1}(t; x,y)+I_{1,2}(t; x,y)	
		\intertext{where}
		I^+_{1,1}(t; x,y)	&=	 -\frac{1}{2\pi} \iint e^{itS_0(k,l)} (l-k)T^+(k,l) \psi_\infty(l) \left(\mu^l-1\right)\, dldk,
		\intertext{and}
		I^+_{1,2}(t; x,y)		&=	-\frac{i}{2\pi} \iint e^{itS_0(k,l)} T^+(k,l)(\psi_+(k) + \psi_-(k)) \left(\mu^l-1\right)\, dldk.
\end{align*}
We first point out that for $I_{1,1}^+$ we can again integrate by parts as in (i) with respect to $k$ variable and make use of \eqref{norm.mul-1} to obtain 
\begin{equation}
     |I^+_{1,1}(t; x,y)|\lesssim t^{-3/2}.
\end{equation}
For $I_{1,2}^+$, through stationary phase, we will show that
\textblue{
\begin{align}
    \left\vert I_{1,2}^+ \right\vert &\lesssim \left\vert \int  e^{-itS(k)}\int i(l-k)(\psi_+(k) + \psi_-(k))\wT^+(k,l) e^{itS(l)}(\mu-1) \, dldk \right\vert\\
    \nonumber
    &\lesssim t^{-1/2}\norm[H^1_k]{\int i(l-k)\left(\psi_+(k) + \psi_-(k)\right)\wT^+(k,l) e^{itS(l)}(\mu^l-1) \, dl}\\
    \nonumber 
    &\lesssim t^{-1/2}\int \left( \int \left\vert \partial_k \left[ i(l-k)(\psi_+(k) + \psi_-(k))\wT^\pm(k,l) \right] \right\vert^2 dk\right)^{1/2}\left\vert\mu^l-1\right\vert dl\\
    \nonumber
    &\lesssim t^{-1/2}\norm[L^2_lL^2_k]{\partial_k \left[ i(l-k)(\psi_+(k) + \psi_-(k))\wT^\pm(k,l)\right]}\norm[L^2_l]{\mu^l-1}\\
    \nonumber
    &=\mathcal{O}(t^{-1}).
\end{align}
}
 More precisely, as is given by \eqref{KPI.u1+.rep1} in the proof of Proposition \ref{prop:KPI:u1.-}, we rewrite $I^+_{1,2}(t; x,y)$ as:
\begin{align}
  I^+_{1,2}(t; x,y)=&  \int  \mathcal{F}^{-1}_k \left((\psi_+(k) + \psi_-(k)) \int {T}_{x,y}^+(k,l)i(l-k) \left(\mu^l-1\right) dl\right) \\
  \nonumber
    &\times \left( \frac{1}{\sqrt{2\pi}} \int e^{-i\xi_2 k}e^{-it(12ak+4k^3)} dk \right)d{\xi_2}.
    \end{align}
By standard %\textit
{Fourier} theory, we need to %obtain the following 
show that
\begin{equation}
\mathcal{F}^{-1}_k \left((\psi_+(k) + \psi_-(k))\int {T}_{x,y}^+(k,l)i(l-k) \left(\mu^l-1\right) dl\right) \in L^{2,1}_k,
\end{equation}
which is equivalent to showing \textblue{that}
\begin{equation}
    \frac{\partial}{\partial k} \left[(\psi_+(k) + \psi_-(k))\int{T}_{x,y}^+(k,l)i(l-k)(\mu^l-1)dl \right] \in L^2_k.
\end{equation}
Recall \eqref{diffT-k1} and \eqref{diffT-k}. Without loss of generality, we will only show the following estimate:
\begin{align*}
  \Biggl\Vert & (\psi_+(k)+\psi_-(k)) \int_k^{+\infty}\int_{\R} \int_\R (l-k) \frac{\partial}{\partial k} \left[\tmu_\sharp^+(k, l-l'-k, y') \right] \Biggr.\\
  &\qquad \Biggl. \times \, \widetilde{u}(l' ; y')dl'e^{i (l^2- k^2) y'}dl' d y' (\mu^l-1)\, dl \, \Biggr\Vert_{L^2_k}\\
  &\lesssim \Norm{\psi_\pm(k)}{L^2_k}\Norm{l\frac{\partial}{\partial k}\tmu^+_\sharp}{L^\infty_kL^2_l} \Norm{\mu^l-1}{L^2_l}\Norm{\tu}{L^{2,2}_l L^{2,2}_y}\\
  &\lesssim t^{-1/2}.
\end{align*}
And this follows from an application of Cauchy-Schwarz inequality, the second line of \eqref{norm.mul-1} and the estimate given by \eqref{est: mu-k-l}.
Recall that $a:=-r^2$. For $t\gg 1$, choose $\xi_2$ such that  $-r^2-\xi_2/12t<-r^2/2$, we make use of \eqref{KPI.Ai.est.t.pre} to obtain:
\begin{align}
\label{est:I12i'}
|I^+_{1,2}&(t; x,y)|=\\
\nonumber
   & \Biggl\vert 
        \int  \mathcal{F}^{-1}_k 
            \left(
                (\psi_+(k) + \psi_-(k)) 
                \int 
                    {T}_{x,y}^+(k,l)i(l-k) \left(\mu^l-1\right) dl\right)  \\
    \nonumber   
        & \qquad
            \left( 
                \frac{1}{\sqrt{2\pi}} 
                \int e^{-i\xi_2 k}e^{-it(12ak+4k^3)} \,dk \right) \, d{\xi_2} 
    \Biggr\vert \\
   \nonumber
    &   \lesssim 
        \norm[L^{2,1}_{ (\xi_2 > -6r^2 t)}]{\mathcal{F}^{-1}_k \left((\psi_+(k) + \psi_-(k)) \int {T}_{x,y}^+(k,l)i(l-k) \left(\mu^l-1\right)  dl\right) }\\
    \nonumber
    & \quad \times \left\vert \frac{1}{\sqrt{2\pi}} \int e^{-i\xi_2 k}e^{-it(12ak+4k^3)} dk \right\vert\\
    \nonumber
     & \quad +
            \left( 
                \int_{\lbrace \xi_2\leq -6r^2 t \rbrace} \xi_2^{-2} d\xi_1     
            \right)^{1/2}\\
            \nonumber
    &\qquad \times \norm[L^{2,1}_{ (\xi_2\leq -6r^2 t)}]{\mathcal{F}^{-1}_k \left((\psi_+(k) + \psi_-(k))\int {T}_{x,y}^+(k,l)i(l-k) \left(\mu^l-1\right)  dl\right) }\\[5pt]
     \nonumber
      & \quad  \lesssim t^{-1}.
\end{align}

So for $a<-\delta<0$, we arrive at 
\begin{equation}
    \left\vert  I^+_{1}(t; x,y)\right\vert\lesssim t^{-1}.
\end{equation}
For $I^+_2(t;x,y)$, notice that 
\begin{align}
\label{diffk I22}
    &\frac{\partial}{\partial k} \left[(\psi_+(k) + \psi_-(k))\int{T}_{x,y}^+(k,l)\frac{\partial}{\partial x} \mu^l(l)dl \right]\\
    \nonumber
    &=\frac{\partial}{\partial k}
        \Biggl[ (\psi_+(k) + \psi_-(k))\\
        \nonumber
    &\qquad \qquad
        \int_k^\infty
            \left( 
                \int_{\mathbb{R}}\widetilde{u}* \tmu^+( l-k ; y')  e^{i (l^2- k^2) y'}d y'
            \right)\\
    \nonumber
    &\qquad \qquad 
            \frac{\partial}{\partial x} \mu^l(l, x,y) dl  
        \Biggr]\\[5pt]
    \nonumber
    &=(\psi_+(k) + \psi_-(k))\left( \int_{\mathbb{R}}\widetilde{u}* \tmu^+( 0,k ; y') d y'\right)\frac{\partial}{\partial x} \mu^l(k, x,y)\\
    \nonumber
    &\quad + \int_k^\infty 
        \frac{\partial}{\partial k} 
        \Biggl[(\psi_+(k) + \psi_-(k))  
            \left( \int_{\mathbb{R}}\widetilde{u}* \tmu^+( l-k ; y')  e^{i (l^2- k^2) y'}d y'
            \right)
        \Biggr]\\
        \nonumber
    &\qquad \qquad \qquad
        \frac{\partial}{\partial x} \mu^l(l, x,y) dl.  
\end{align}
The fact that \eqref{diffk I22} belongs to $L^2_k$ can be deduced from the same argument as in the estimation of \eqref{est:I12i}. The same estimation of \eqref{est:I12i'} can be applied to obtain that for $a<-\delta<0$,
\begin{equation}
    \left\vert I^+_{2}(t;x,y)\right\vert =\mathcal{O}(t^{-1}).
\end{equation}
%\textblue{The scattering-data norms used in applying \eqref{est:I12i'} are bounded by the
%estimates of Section~\ref{sec:Tpm}; the underlying weighted norms of $\tu$ are
%recorded in Lemma~\ref{lem:E1w.bookkeeping}.}

%\item[(iii)]

\medskip

\textblue{(iii) $|a|<\delta $. We let $\psi_0$ be a smooth cutoff function such that 
\begin{align}
\label{cutoff-k-0}
    \psi_0(k)=\begin{cases} 1, & |k|\leq 1,\\
                          0, & |k| \geq 2.
    \end{cases}
\end{align}

%%%%%%%%%%%%%%%%%%%%%%%%%%%%%%%%%%%%%%%%%%%%%
% THis doesn't belong here (PAP)
%Appendix \ref{App:Airy}.
%%%%%%%%%%%%%%%%%%%%%%%%%%%%%%%%%%%%%%%%%%%%%
We then replace $\psi_+(k)+\psi_-(k)$ in the proof of the case $a<-\delta<0$ by $\psi_0(k)$ and follow the same proof to obtain the third line in \eqref{norm.I1.k}-\eqref{norm.I2.k}. Since $a=\zeta/12-\eta^2/144$ appearing in \eqref{Airy.id.again} can be arbitrarily small, we are unable to divide the Fourier variable $\xi$ into two separate intervals as we did in \eqref{KPI.Ai.est.t} to deduce time decay faster than $t^{-1/3}$. Instead we directly make use of the fact that $\left\vert \eqref{Airy.id.again}\right\vert = \mathcal{O}(t^{-1/3})$ in our stationary phase analysis in the $k$ variable. For instance, when $|a|<\delta$, \eqref{est:I12i'} becomes:
\begin{align*}
    |I^+_{1,2}&(t; x,y)|=\\
\nonumber
   & \Biggl\vert 
        \int  \mathcal{F}^{-1}_k 
            \left(
                \psi_0(k) 
                \int 
                    {T}_{x,y}^+(k,l)i(l-k) \left(\mu^l-1\right) dl\right)  \\
    \nonumber   
        & \qquad
            \left( 
                \frac{1}{\sqrt{2\pi}} 
                \int e^{-i\xi_2 k}e^{-it(12ak+4k^3)} \,dk \right) \, d{\xi_2} 
    \Biggr\vert \\
   \nonumber
    &   \lesssim 
        \norm[L^{2,1}_{ \xi}]{\mathcal{F}^{-1}_k \left(\psi_0(k) \int {T}_{x,y}^+(k,l)i(l-k) \left(\mu^l-1\right)  dl\right) }\\
    \nonumber
    & \quad \times \left\vert \frac{1}{\sqrt{2\pi}} \int e^{-i\xi_2 k}e^{-it(12ak+4k^3)} dk \right\vert\\
    &\lesssim \Norm{\psi_0(k)}{L^2_k}\Norm{l\frac{\partial}{\partial k}\tmu^+_\sharp}{L^\infty_kL^2_l} \Norm{\mu^l-1}{L^2_l}\Norm{\tu}{L^{2,2}_l L^{2,2}_y}\\
    &\quad \times \left\vert \frac{1}{\sqrt{2\pi}} \int e^{-i\xi_2 k}e^{-it(12ak+4k^3)} dk \right\vert\\
    &  \lesssim t^{-2/3}.
\end{align*}
}
%\end{itemize}
\end{proof}

\textblue{
\begin{remark}
    In the estimation of \eqref{est:I12i'} (see also Proposition \ref{prop:KPI:u1.-}) that the constant depends on our choice of $r^2:=|a|$ hence $\delta$ in the following ways:
    \begin{itemize}
        \item defining the radius of the cutoff function $\psi_\pm(k)$;
        \item obtaining the time decay estimate of \eqref{Airy.id.again} by choosing $ \xi>-6r^2t $.
    \end{itemize}
    So the constants in our estimates are uniform in the sense that as long as the difference $\left\vert x/t-y^2/12t^2\right\vert$ is bounded below for all $x,y,t$, the constants depend continuously on the initial data $u_0\in\mathbf{E}_{1,w} $.
\end{remark}
}

%%%%%%%%%%%%%%%%%%%%%%%%%%%%%%%%%%%%%%%%%%%
%   
%   PAP made this a separate subsection
%   on 6.22.2026
%
%%%%%%%%%%%%%%%%%%%%%%%%%%%%%%%%%%%%%%%%%%%

\subsection{Estimates on the Nonlocal Riemann-Hilbert Problem}
\label{subsec:NLRHP.est}

\textblue{In this subsection,  we establish the estimates \eqref{norm.mul-1}-\eqref{norm.dmul.dx} on the nonlocal Riemann-Hilbert problem \eqref{mul.RHP} for which the scattering data, and hence the solution, evolve in time}. 
%First 
Recall that 
\begin{equation}
	\label{Cauchy.bd}
	\norm[L^2_k \to L^2_k]{C_\pm} = 1.
\end{equation}

%In what follows, 
%we
\textblue{First of all, we}
show that the Riemann-Hilbert problem \eqref{mul.RHP} has a unique solution for $u \in \mathbf{E}_{1,w}$ \textblue{so chosen that the smallness conditions} \eqref{u.C1} and \eqref{u.C2} hold. 
%We also obtain the estimates \eqref{norm.mul-1} and \eqref{norm.dmul.dx} (see Proposition \ref{prop:RHP.bounds}). 

\begin{lemma}
\label{lemma:res.CT.est}
	Suppose that $u \in \mathbf{E}_{1,w}$, $c < 1$, and  $\norm[L^2_y L^{2,-1}_l]{\tu} < \dfrac14(1-c)$.
	Then, the resolvent bound 
	\begin{equation}
		\label{res.CT.est}
			\norm[L^2_k \to L^2_l]
				{(I-\calC_T)^{-1}} < 2	
	\end{equation}
	holds uniformly in $t,x,y$.
	Morever, for each $(t,x,y)$, the Riemann-Hilbert problem \eqref{mul.RHP} has a unique solution with  $\mu^l - 1 \in L^2(\R,dl)$. Finally, the estimate
	\begin{equation}
		\label{mul.norm}
			\norm[L^2_l]{\mu^l-1} \leq 2 \norm[L^2_l]{\calC_T(1)}
	\end{equation}
	holds uniformly in $x,y,t$, where the right-hand side is bounded provided $u \in \mathbf{E}_{1,w}$.
\end{lemma}

\begin{proof}
From the estimate 
\begin{align}
	\norm[L^2(\R^2)]{T^\pm} &\leq \frac{1}{1-c} \norm[L^2_y L^{2,-1}_l]{\tu}	,
\end{align}
it follows that
$$ \norm[L^2_k \to L^2_l]{\calT^\pm} \leq \frac{1}{1-c} \norm[L^2_y L^{2,-1}_l]{\tu}	$$
since the Hilbert-Schmidt norm bounds the operator norm. Using
\eqref{Cauchy.bd},
we conclude that
\begin{equation}
	\label{CT.norm}
		\norm[L^2 \to L^2]{\calC_T} \leq \frac{2}{1-c} \norm[L^2_y L^{2,-1}_l]{\tu}.
\end{equation}
Hence
\begin{equation}
	\label{CT.norm.small}
		\norm[L^2 \to L^2]{\calC_T } < \frac{1}{2} \text{  whenever }
			\norm[L^2_y L^{2,-1}_l]{\tu} < \frac{1-c}{4}
\end{equation}
so that \eqref{res.CT.est} holds.

Next, we have the estimate
\begin{equation}
	\label{CT1.pre}
		\norm[L^2_l]{\calC_T (1)} \lesssim \norm[L^2_k H^1_l]{T^+} + \norm[L^2_k H^1_l]{T^-}
\end{equation}
where the right-hand side of \eqref{CT1.pre} is bounded by
$$ \sum_{\eps \in \{+,-\}}\norm[L^2_k H^1_l]{T^\eps} + \norm[L^2_k H^1_l]{(l-k) T^\eps}$$
owing to the formulas
$$ \calT^\pm(1) = \pm \int_k^{\pm \infty} e^{itS_0(k,l;\zeta,\eta)} T^+(k,l) \, dl $$
and \eqref{Cauchy.bd}.
The right-hand norms are bounded under the hypotheses of the lemma.

We can now use the solution formula
\begin{equation}
	\label{mul.sol}
		\mu^l - 1 = (I-\calC_T)^{-1} \calC_T (1) 
\end{equation}
and the estimates \eqref{CT.norm.small}  and \eqref{CT1.pre} to construct a solution with $\mu^l-1 \in L^2_l(\R)$ with $L^2_l$ norm estimated uniformly in $x,y,t$. Uniqueness follows from the boundedness of the resolvent.
\end{proof}

Having established the existence and uniqueness of $\mu^l$, we can turn attention to proving the bounds \eqref{norm.mul-1} and \eqref{norm.dmul.dx}. \textblue{The proof rests on $L^2$ estimates on the quantities $\calC_T(1)$ and $\calC_{\diff T/\diff x}(1)$ which are the respective dominant terms in the solution formula \eqref{mul.sol} for $\mu^l-1$ and the formula \eqref{dmul.dx.eqn} for $\diff \mu^l/\diff x$.
These $L^2$ estimates are proved in Lemmas \ref{lemma:CT.1} and \ref{lemma:CdTdX.1}. }We will prove:

\begin{proposition}
	\label{prop:RHP.bounds}
	Suppose that $u \in \mathbf{E}_{1,w}$ and that the conditions 
	\eqref{u.C1}, \eqref{u.C2}, and \eqref{u.C2a} hold. Then the estimates \eqref{norm.mul-1} and \eqref{norm.dmul.dx} hold.
\end{proposition}

\begin{proof}
We give the proof modulo Lemmas \ref{lemma:CT.1} and \ref{lemma:CdTdX.1}	in what follows.

First, from the solution formula \eqref{mul.sol}, we have 
\begin{equation}
	\label{mul-1.est.pre}
		\norm[L^2_l]{\mu^l - 1}
			\leq 2 \norm[L^2_l]{\calC_T(1)}.	
\end{equation}

Second, we claim that
\begin{equation}
	\label{dmul.dx.est.pre}
		\norm[L^2_l]{\frac{\diff \mu^l}{\diff x}}
		\leq 	
		2 \norm[L^2_l]{\calC_{\diff T/\diff x}(1)} +
				4 \norm[L^2 \to L^2]{\calC_{\diff T/\diff x}} \norm[L^2_l]{\calC_T(1)}.
\end{equation}
To see this, note that the function $\dfrac{\diff \mu^l}{\diff x}$ satisfies the integral equation
\begin{align}
	\frac{\diff \mu^l}{\diff x}
			&=	\calC_{\diff T/\diff x}(\mu^l) +
				\calC_T\left( \frac{\diff \mu^l}{\diff x} \right)
	\intertext{where}
	\label{C.dTdX.op}
		\calC_{\diff T/\diff x}(f) 
			&=	C_+ \frac{\diff \calT^-}{\diff x} f +
				C_- \frac{\diff \calT^+}{\diff x} f,
	\intertext{and}
	\label{dTdX.op}
		\left( \frac{\diff \calT^\pm}{\diff x} \right)(f)(k)
			&= \pm \int_k^{\pm \infty} e^{itS} i(l-k) T^\pm(k,l) f(l) \, dl 
\end{align}
Note that 
\begin{equation}
	\label{dTdX.est}	
		\norm[L^2_l \to L^2_k]{\frac{\diff \calT^\pm}{\diff x}}
	\leq \norm[L^2_k L^2_l]{(l-k)T^\pm}.
\end{equation}
Since the right-hand side is the Hilbert-Schmidt norm of the integral operator \eqref{dTdX.op}. Hence
\begin{equation}
	\label{C.dTdX.est}
		\norm[L^2_l \to L^2_l]{\calC_{\diff T/\diff x}} \leq 2 \norm[L^2_k L^2_l]{(l-k)T^\pm}.
\end{equation}
It follows that
\begin{align}
	\label{dmul.dx.eqn}
		\frac{\diff \mu^l}{\diff x}
			&=	(I-\calC_T)^{-1} \calC_{\diff T/\diff x}(\mu^l)\\
			&=	(I-\calC_T)^{-1}\calC_{\diff T/\diff x}(1) +(I-\calC_T)^{-1}\calC_{\diff T/\diff x}(I-\calC_T)^{-1}\calC_T(1)
			\nonumber	
\end{align}
where in the second step we used \eqref{mul.sol}. From the resolvent bound \eqref{res.CT.est} we conclude that \eqref{dmul.dx.est.pre} holds.

It follows from \eqref{mul-1.est.pre} and \eqref{dmul.dx.est.pre} that, to obtain the bounds \eqref{norm.mul-1} and \eqref{norm.dmul.dx}, it suffices to estimate $\norm[L^2_l]{\calC_T(1)}$ and $\norm[L^2_k]{\calC_{\diff T/\diff x}(1)}$ in each of the cases $a>\delta>0$ (no stationary points), $|a| < \delta$ (transition region), and  $a<-\delta<0$ (nondegenerate stationary points). These estimates are proved in Lemmas \ref{lemma:CT.1} and \ref{lemma:CdTdX.1}.
\end{proof}
 
Now we turn to the estimates on $\calC_T(1)$ and $\calC_{\diff T/\diff x}(1)$.

\begin{lemma}
	\label{lemma:CT.1}
	Suppose that $u \in \mathbf{E}_{1,w}$ and the conditions
	\eqref{u.C1}, \eqref{u.C2}, and \eqref{u.C2a} hold. Fix $\delta>0$. Then, the estimates
	\begin{equation}
	\label{C.T.1.est}
		\norm[L^2_k]{\calC_T(1)}
			\lesssim
			\begin{cases}
				t^{-1}, 			& a > \delta > 0,\\
				t^{-\frac13},	& |a| < \delta,\\
				t^{-\frac12}, 	& a < -\delta < 0.
			\end{cases}
	\end{equation}
	%hold.
    \textblue{hold, where the implied constants for the cases $|a|>\delta$ diverge as $\delta \downarrow 0$.}
\end{lemma}

\begin{proof}
	By \eqref{Cauchy.bd}, it suffices to show that
	\begin{equation}
		\label{Tpm.1.bd}
			\norm[L^2_k]{\calT^\pm(1)}
			\lesssim
			\begin{cases}
				t^{-1}, 			& a > \delta > 0,\\
				t^{-\frac13},	& |a| < \delta,\\
				t^{-\frac12}, 	& a < -\delta < 0.
			\end{cases}
	\end{equation}
	We will only consider $\calT^+(1)$ since the analysis of $\calT^-(1)$ is similar.
	
	Corresponding to the decomposition \eqref{Tpm.split}, we write 
	$$\calT^+(1) = \calT^+_1(1) + \calT^+_2(1)$$
  	where
  	\begin{align*}
  		\calT^+_1(1)
  			&=	\int e^{itS_0} T^+_1(k,l) \, dl,
  		&
  		\calT^+_2(1)
  			&=	\int	 e^{itS_0} T^+_2(k,l) \, dl.
  	\end{align*}
    (for simplicity we write $T_1^+$ for $T_1$).
  	
  	\medskip
	
	(i) The case $a:=r^2>\delta>0$.
	
	\medskip
	
	In this case, the phase function $S_0$ has no critical points and
	\begin{equation}
		\label{T.1.P}
			P(l) = 12l^2 -2\eta l + \zeta
	\end{equation}
	satisfies
	$$ P(l) \geq 12a. $$
	We have
	\begin{align*}
		\calT^+(1) 
			&= -\frac{1}{2\pi t} \int_k^\infty 
					T^+(k,l) 
						\frac{1}{P(l)} 
						\frac{\diff}{\diff l}
						\left(e^{itS_0}\right) 
				\, dl	\\
			&= -\frac{1}{2\pi t}\left(I_1 + I_2\right)
		\intertext{where, for $j=1,2$,}
		I_j	&=	\int_k^\infty 
					T_j^+(k,l)
					\frac{1}{P(l)}
					\frac{\diff}{\diff l} \left( e^{itS_0} \right) 
				\, dl.
	\end{align*}
	Integrating by parts we have, for $j=1,2$,
	\begin{align*}
		I_j	&=	-\frac{T_j^+(k,k)}{P(k)} - 
				\int_k^\infty 
					e^{itS_0}
						\left( 
							\frac{P'(l)}{P(l)^2}T^+_j(k,l) +
							\frac{1}{P(l)} \frac{\diff T^+_j}{\diff l}(k,l)
						\right) 
				\, dl 
	\end{align*}
	For $j=1$, we \textblue{use inclusions (i)-(iii) in Lemma~\ref{rem:Tpm.kl.L2} to control the
norms of $\tu$ appearing on the right-hand side,}
	\begin{align}
		\label{nocp.I1.est}
		\norm[L^2_k]{I_1}
			&\lesssim \norm[L^\infty]{T_1^+} + \norm[L^2]{T_1^+} + 
					\norm[L^2]{\frac{\diff T_1^+}{\diff l}}\\
			&\lesssim \norm[L^1(\R^2)]{\tu} + \norm[L^2_y L^{2,-1}_l]{\tu} +  
				\norm[L^{2,1}_y L^{2,-1}_l]{\tu} + \norm[L^{2,1}_y L^{2}_l]{\frac{\diff \tu}{\diff l}}
			\nonumber
	\end{align}
	where we used the change of variables $(k,l) \to (k,l+k)$,  and the $L^2$ bounds on integrals of the form
	$$ \int e^{il(l+2k)y} g(l) \, dl$$
	to conclude that
	$$
	\norm[L^2]{T^+_1} \lesssim \norm[L^2_y L^{2,-1}_l]{\tu},$$
	and
	$$
		\norm[L^2]{\frac{\diff T^+_1}{\diff l}}
			\lesssim \norm[L^{2,1}_y L^{2,-1}_l]{\tu} + \norm[L^{2,1}_y L^{2}_l]{\frac{\diff \tu}{\diff l}}.
	$$
	Next, for $j=2$, we estimate
	\begin{align}
		\label{nocp.I2.est}		
			\norm[L^2_k]{I_2}
			\lesssim \norm[L^\infty]{T_2^+} + \norm[L^2]{T^+_2} + 
				\norm[L^2_k]{ \int_k^\infty e^{itS_0} \frac{1}{P(l)} \frac{\diff T_2^+}{\diff l} \, dl}.
	\end{align}
	To estimate the first two right-hand terms in \eqref{nocp.I2.est}, we have
	\begin{align}
		\norm[L^\infty]{T_2^+} 
		\label{nocp.I2.1}
			&\lesssim \norm[L^{2,1}_y L^{2}_l]{\tu} \norm[L^\infty_{k,y} L^2_l]{\tmu^+_\sharp},
		\intertext{and}
		\label{nocp.I2.2}
		\norm[L^2]{T_2^+}
			&\lesssim \norm[L^1(\R^2)]{\tu} \norm[L^\infty_y L^2_k L^2_l]{\tmu^+_\sharp}.
	\end{align} 
	To estimate the third right-hand term in \eqref{nocp.I2.est}, we first note that
	\begin{align*}
		\frac{1}{P(l)}&\frac{\diff T_2^+}{\diff l}(k,l)
			\\ 
			&=-\frac{i}{2\pi}
					\int
						e^{i(l^2-k^2)y}
								\frac{2ily}{P(l)} (\tu*\tmu^+_\sharp)(k,l-k;y) 
					\, dy\\
			&\quad -
				\frac{i}{2\pi}
					\int
						e^{i(l^2-k^2)y}
							\frac{1}{P(l)}
							\left(	
								\frac{\diff \tu}{\diff l}*\tmu^+_\sharp
							\right)
							(k,l-k;y)
					\, dy						
	\end{align*}
	which implies that
	\begin{align}
		\label{nocp.I2.3}
		\norm[L^1_l L^2_k]{\frac{1}{P(l)}\frac{\diff T^+_2}{\diff l}}
			&\lesssim \norm[L^{1,1}_y L^1_l]{\tu} \norm[L^\infty_y L^2_k L^2_l]{\tmu^+_\sharp}\\
			&\quad
				+ \norm[L^1_y L^1_l]{\frac{\diff \tu}{\diff l}} \norm[L^\infty_y L^2_k L^2_l]{\tmu^+_\sharp}.
		\nonumber
	\end{align}
    \textblue{In \eqref{nocp.I2.1}--\eqref{nocp.I2.3} we use Young's inequality in
the convolution variable. We refer the reader to Lemma \ref{rem:Tpm.kl.L2} and Remark \ref{rem:Tpm.kl.L2'} for the estimates involving $\tu$ and $\tmu^+_\sharp$ .}
	Estimates \eqref{nocp.I2.1}--\eqref{nocp.I2.3} leads to the time decay of $I_2$ and, together with \eqref{nocp.I1.est},
	imply the estimate \eqref{Tpm.1.bd}, hence \eqref{C.T.1.est}.
	\medskip
	
	(ii) The case $a:=-r^2<-\delta<0$.
	
	\medskip
	
	Now we study $\calT^+(1)$ in the case $a<=-r^2 <0$ where the phase function \eqref{S0.def} has nondegenerate critical points at $l_\pm = \frac{\eta}{12} \pm r$. Let $\psi_\pm \in C_0^\infty(\R)$ be nonnegative bump functions with $\psi_\pm(l)=0$ for $|l - l_\pm| \geq r/4$ and $\psi(l)=1$ for $|l-l_\pm| \leq r/8$. Finally, let $\psi_\infty=1-\psi_+- \psi_-$. We now have
	\begin{align}
		(\calT^+(1))(k)
			&=	I_1+I_2,	
		\intertext{where}
        \label{CT1-infty}
		I_1	&=	 -\frac{1}{2\pi} \int_k^\infty e^{itS_0(k,l)} T^+(k,l) \psi_\infty(l) \, dl
		\intertext{and}
        \label{CT1-loc}
		I_2	&=	-\frac{i}{2\pi} \int_k^\infty e^{itS_0(k,l)} T^+(k,l)(\psi_+(l) + \psi_-(l)) \, dl.
	\end{align}
	Using integration by parts in the same way as in (i), we can bound
	\begin{equation}
		\label{cp.I1.bd}
			\norm[L^2_k]{I_1} \lesssim t^{-1}
	\end{equation}
	with constants depending on $\tu$, uniform in $(t,x,y)$ but diverging as $a \uparrow 0$. 
	To bound $I_2$ we will use \eqref{Tpm.split} to write
	$$ I_2 = J_1 + J_2$$
	where
	\begin{align}
		J_1	&=	-\frac{i}{2\pi} 
					\int_k^\infty 
						e^{itS_0(k,l)} T_1^+(k,l)(\psi_+(l) + \psi_-(l)) \, dl ,
	\intertext{and}
    \label{CT1.x.cp.form'}
		J_2	&=	-\frac{i}{2\pi} 
					\int_k^\infty 
						e^{itS_0(k,l)} T_2^+(k,l)(\psi_+(l) + \psi_-(l)) \, dl .
	\end{align} 
	We will show how to estimate the terms involving $\psi_+$ since the estimates for $\psi_-$ are similar. First, we make a change of variables
	\begin{equation}
	\label{kl.shift.post}
		(k,l) \to (k+\eta/12,l+\eta/12)	,
	\end{equation}
  so that
	\begin{align}
		\label{J1+.def}
		J_1^+	&=	-\frac{i}{2\pi}\int_k^\infty e^{itS(k,l)}\wT^+_1(k,l) \psi_+(l) \, dl,
		\intertext{and}
		\label{J2+.def}
		J_2^+	&=	-\frac{i}{2\pi}\int_k^\infty e^{itS(k,l)}\wT^+_2(k,l) \psi_+(l) \, dl
	\end{align}
	where the superscript $+$ indicates the term of $J_i$ involving $\psi_+$, and $\psi_+(l)$ is shorthand for $\psi_+$ in the new variables, 
	\begin{equation}
		\label{wT+.def}
			\wT^+(k,l) = T^+(k + \eta/12,l+\eta/12),
	\end{equation}
	and
	\begin{equation}
		\label{S.def.r}
			S(k,l) = -12r^2(l-k) +4(l^3-k^3). 
	\end{equation}
	Note that $S$ has critical points at $l=\pm r$ and $\psi_+$ (in the new variable) has support near $l=r$. 
	
	To analyze the asymptotics of \eqref{J1+.def} and \eqref{J2+.def}, we will use the Fourier multiplication formula
	\begin{equation}
		\label{Fourier.mult}
			\int_\R f(l) g(l) \, dx = \int_\R \widehat{f}(-\xi) \widehat{g}(\xi) \, d\xi
	\end{equation}
	where we take
	$$ f(l) = H(l-k) e^{itS(k,l)}, $$
	and
	$$ g(l) = \wT^+(k,l) \psi_+(l). $$
	Here we view $\wT^+(k,l)$ as a smooth function on $\R^2$ defined by 
	\eqref{Tpm.split}, \eqref{T1.def}--\eqref{T2.def}, and \eqref{wT+.def}. 
	 We have the `Half-Airy' estimate (see Lemma \ref{lemma:partial.airy.ndg})
	\begin{equation}
		\label{f.ft.est}
			|\widehat{f}(\xi)| \lesssim t^{-\frac12}(1+|\xi|). 
	\end{equation}
	Let 
	$$ 
		g_j(l) = \wT^+_j(k,l) \psi_+(l).
	$$
	Then
	\begin{align*}
		\widehat{g}_j(\xi)
			&=	(2\pi)^{-\frac12}
					\int e^{-i\xi l} \psi_+(l) \wT^+_j(k,l) \, dl.
	\end{align*}
	It follows from the Plancherel Theorem that
	\begin{align}
		\label{g.Sobolev}
			\norm[L^2]{(1+\xi^2)\widehat{g}_j}
				&\lesssim 
					\norm[L^2_l]{\psi_+(\dotarg) \wT^+(k,\dotarg)} 	+
					\norm[L^2_l]{\frac{\diff^2}{\diff l^2}\left(\psi_+(\dotarg)\wT^+_j\right)},
	\end{align}
	so that, using \eqref{Fourier.mult}, \eqref{f.ft.est}, and \eqref{g.Sobolev}, we conclude that
	\begin{align}
		\label{J1.est.pre}
			\norm[L^2_k]{J_1}
				&\lesssim t^{-\frac12} 
						\left(
							\norm[L^2_k L^2_l]{\psi_+(\cdot)\wT_1^+(\cdot,\cdot) }	 +
							\norm[L^2_k L^2_l]{\frac{\diff^2}{\diff l^2}
													\left(
														\psi_+(\cdot) \wT_1^+(\cdot,\cdot)
													\right)}
						\right),
		\intertext{and}
		\label{J2.est.pre}
			\norm[L^2_k]{J_2}
				&\lesssim t^{-\frac12} 
						\left(
							\norm[L^2_k L^2_l]{\psi_+(\cdot)\wT_2^+(\cdot,\cdot) }	 +
							\norm[L^2_k L^2_l]{\frac{\diff^2}{\diff l^2}
													\left(
														\psi_+(\cdot) \wT_2^+(\cdot,\cdot)
													\right)}
						\right).
	\end{align}
	
	The proof of \eqref{Tpm.1.bd}, and hence \eqref{C.T.1.est}, in the case $a < c < 0$, will be finished provided we can show that the right-hand norms in \eqref{J1.est.pre} and \eqref{J2.est.pre} are finite.

	Consider first \eqref{J1.est.pre}. From \eqref{Iv.bd}, we have
	\begin{equation}
		\label{J1.term1.est}	
			\norm[L^2_k L^2_l]{\wT^+_1} \lesssim \norm[L^2_y L^{2,-1}_l]{\tu}
	\end{equation}
	which shows that the first right-hand term of \eqref{J1.est.pre} is bounded. To show that the second right-hand term of \eqref{J1.est.pre} is bounded, 
	we compute
	\begin{align*}
		\frac{\diff^2}{\diff l^2}\left( \psi_+(l) \wT^+_1(k,l) \right)
			&=\psi_+''(l) \wT^+_1(k,l) + 
				2\psi_+'(l) \frac{\diff \wT^+_1}{\diff l}(k,l) +
				\psi_+(l) \frac{\diff^2 \wT^+_1}{\diff l^2}(k,l) ,
	\end{align*} 
	so that
	\begin{align}
		\label{J1.term2.est}
			\norm[L^2_k L^2_l]
				{
					\frac{\diff^2}{\diff l^2}
					\left(\psi_+ \wT^+_1 \right)
				}
		&\lesssim 
			\norm[L^{2,1}_y L^{2}_l]{y^2\tu} +
			\norm[L^{2,1}_y L^{2}_l]{y\frac{\diff \tu}{\diff l}} +
			\norm[L^{2,1}_y L^{2}_l]{ \frac{\diff^2 \tu}{\diff l^2} }
	\end{align}
	where we used
	\begin{align*}
		\frac{\diff T^+_1}{\diff l}
			&= -\frac{i}{2\pi} \int e^{i(l^2-k^2)y}
					\left(2iyl \tu(y,l-k) + \frac{\diff \tu}{\diff l}(y,l-k) \right) \, dy,
	\intertext{and}
		\frac{\diff^2 T^+_1}{\diff l^2}
			&= -\frac{i}{2\pi} \int e^{i(l^2-k^2)y}
					\left( -4l^2y^2 \tu(y,l-k) + 2ily \frac{\diff \tu}{\diff l}(y,l-k) + \frac{\diff^2 \tu}{\diff l^2}(y,l-k) \right) \, dy.
	\end{align*}
	Combining \eqref{J1.term1.est} and \eqref{J1.term2.est}, we recover 
	\begin{equation}
		\label{J1.est}
		\norm[L^2_k]{J_1} \lesssim t^{-\frac12}. 
	\end{equation}

	Now we consider \eqref{J2.est.pre}. First, we may bound
	\begin{align}
		\label{J2.term1.est}
			\norm[L^2_k L^2_l]{\wT^+_2}
				&\lesssim \norm[L^1(\R^2)]{\tu} \norm[L^\infty_y(L^2_kL^2_l)]{\tmu^+_\sharp}.
	\end{align}
	Next, we compute
	\begin{align}
		\label{J2.term2.form}
		\frac{\diff^2}{\diff l^2} \left( \psi_+\wT^+_2 \right) (k,l)
			&=	\psi_+''(l) \wT^+_2(k,l) +	
				2 \psi_+'(l) \frac{\diff \wT^+_2}{\diff l}(k,l) +\psi_+(l) \frac{\diff^2 \wT^+}{\diff l^2}(k,l).
	\end{align}
	We bound each of the three right-hand terms of \eqref{J2.term2.form} in $L^2_k L^2_l$ norm.
	
	First,
	\begin{align*}
		 \wT^+_2(k,l)
			&=	- \frac{i}{2\pi} \int e^{i(l^2-k^2)y} (\tu*\mu^+_\sharp)(k,l-k;y) \, dy,
	\end{align*} 
	so 
	\begin{align}
		\label{J2.term2.1}
		\norm[L^2_k L^2_l]{\wT^+_2}
			&\lesssim \norm[L^1(\R^2)]{\tu} \norm[L^\infty_y(L^2_k L^2_l)]{\mu^+_\sharp},
	\end{align}
	bounding the first right-hand term of \eqref{J2.term2.form}.
	
	Next,
	\begin{align*}
		\psi_+'(l) \frac{\diff \wT^+_2(k,l)}{\diff l}(k,l)
			&=	-\frac{i}{2\pi} \int e^{i(l^2-k^2)y} \, 2ily \, \psi_+'(l) (\tu*\tmu^+_\sharp)(k,l-k;y) \, dy\\
			&\quad +	
				-\frac{i}{2\pi} \int e^{i(l^2-k^2)y} \psi_+'(l) \left( \frac{\diff \tu}{\diff l}*\mu^+_\sharp \right)(k,l-k;y) \, dy
	\end{align*}
	so
	\begin{align}
		\label{J2.term2.2}
		\norm[L^2_k L^2_l]{\psi_+' 	\frac{\diff \wT^+_2}{\diff l}}
			&\lesssim 
				\left(
					\norm[L^1(\R^2)]{y\tu} +
					\norm[L^1(\R^2)]{\frac{\diff \tu}{\diff l} } 
				\right)
				\norm[L^\infty_y(L^2_k L^2_l)]{\tmu^+_\sharp}
	\end{align}
	where the implied constant depends on $r$ through $\sup_l |l\psi_+'(l)|$. This bounds the second right-hand term of \eqref{J2.term2.form}.
	
	Finally,
	\begin{align*}
		\psi_+(l) \frac{\diff^2 \wT^+_2}{\diff l^2}(k,l)
			&= -\frac{i}{2\pi} 
					\int e^{i(l^2-k^2)y}
						\psi_+(l) 
							(-4l^2 y^2) (\tu*\tmu^+_\sharp)(k,l-k;y)
					\, dy \\
			&\quad -\frac{i}{2\pi}	
					\int e^{i(l^2-k^2)y}
						\psi_+(l)
							(2ily) \left( \frac{\diff \tu}{\diff l}*\tmu^+_\sharp \right)(k,l-k;y)
					\, dy \\
			&\quad -\frac{i}{2\pi}
					\int e^{i(l^2-k^2)y}
						\psi_+(l)
							\left( 
								\frac{\diff^2 \tu}{\diff l^2} * \tmu^+_\sharp
							\right)(k,l-k;y)
					\, dy,
	\end{align*}
	so that
	\begin{multline}
		\label{J2.term2.3}
			\norm[L^2_kL^2_l]{\psi_+(l) \frac{\diff^2 \wT^+_2}{\diff l^2} }
			\\\lesssim
				\left(	
					\norm[L^1(\R^2)]{y^2\tu}  +
					\norm[L^1(\R^2)]{y \frac{\diff \tu}{\diff l}}  +
					\norm[L^1(\R^2)]{\frac{\diff^2 \tu}{\diff l^2}}
				\right)
				\norm[L^\infty_y(L^2_k L^2_l)]{\tmu^+_\sharp}.
	\end{multline}
    \textblue{Given $\psi_+$ is compactly supported, $l^2\psi_+(l)$ are bounded by constants
depending on $r$, and hence on $\delta$. For  norms involving
$\tu$ and $\tmu^+_\sharp$ in
\eqref{J2.term2.3} we refer the reader to Lemma \ref{rem:Tpm.kl.L2} and Remark \ref{rem:Tpm.kl.L2'}.}
	This bounds the third right-hand term of \eqref{J2.term2.form}.
	
	Combining \eqref{J2.term2.1}--\eqref{J2.term2.3}, we recover
	\begin{equation}
		\label{J2.est}
			\norm[L^2_k]{J_2} \lesssim t^{-\frac12}.	
	\end{equation}
	
	Finally, combining \eqref{J1.est} and \eqref{J2.est}, we recover the bound \eqref{Tpm.1.bd} (and hence, \eqref{C.T.1.est}) in the case $a:=-r^2<-\delta<0$.
	
	\medskip
	
	(iii) \textblue{The case $|a|\leq\delta$. We first recall from Lemma \ref{lemma:partial.airy.dg} that  
\begin{align}
\label{Airy.id.again.bis'}
    \frac{1}{\sqrt{2\pi}} \int_k^{+\infty} e^{-i\xi l}e^{it(12al+4l^3)} dl =\mathcal{O}(t^{-1/3}).
\end{align}
We let $\psi_0$ be a smooth cutoff function such that 
\begin{align}
\label{cutoff-l-0}
    \psi_0(l)=\begin{cases} 1, & |l|\leq 1,\\
                          0, & |l| \geq 2.
    \end{cases}
\end{align}
and split
$$\calT^+(1) = I_1 + I_2$$
where
\begin{align*}
	I_1	&=	\int_k^\infty e^{itS_0} \psi_\infty(l) T^+(k,l) \, dl,\\
	I_2 &=	\int_k^\infty e^{itS_0} \psi_0(l) T^+(k,l) \, dl 
\end{align*}
where $\psi_\infty = 1- \psi_0$. The rest of the proof will follow from the the proof of the case $a:=r^2<-\delta<0$. To illustrate this, 
notice that for $I_1$, the stationary point is outside of the support of the amplitude function, so we directly carry out integration by parts as is done for \eqref{CT1-infty}. For $I_2$, in this case \eqref{CT1-loc} becomes
\begin{equation}
\label{ct1.x.0'}
    \int_k^\infty 
					e^{itS_0} i(l-k) 
						\psi_0(l) T^+(k,l) \, dl
\end{equation}
and we will follow the same arguments as those of \eqref{J1+.def}-\eqref{J2+.def} and make use of \eqref{Airy.id.again.bis'} to obtain 
\begin{align}
    \Norm{\eqref{ct1.x.0'}}{L^2_k}\lesssim t^{-\frac13}.
\end{align}
Notice that in this case we do not divide the $\xi$ variable into two separate intervals.
}
\medskip

This completes the proof.
\end{proof}

%\todo[size=\tiny]{Can we get $t^{-\frac13}$ assuming only $a=o(1)$ or $|a|<\delta$?}
\begin{lemma}
	\label{lemma:CdTdX.1}
	Suppose that $u \in \mathbf{E}_{1,w}$ and the conditions \eqref{u.C1}, \eqref{u.C2}, and \eqref{u.C2a} hold. Fix $\delta>0$. Then, the estimates
	\begin{equation}
		\norm[L^2_k]{\calC_{\diff T/\diff x}(1)}
			\lesssim
			\begin{cases}
				t^{-1}, 			& a > \delta > 0,\\
				t^{-\frac13},	&	|a| < \delta,\\
				t^{-\frac12}, 	& a < -\delta < 0
			\end{cases}
	\end{equation}
	hold. And  the implied constants for the cases $|a|>\delta$ diverge as $\delta \downarrow 0$.
\end{lemma}

\begin{proof}
As before, it suffices to prove that
	\begin{equation}
	\label{dTdX.1.est}
		\norm[L^2_k]{\frac{\diff}{\diff x} \calT^\pm (1)}
			\lesssim
			\begin{cases}
				t^{-1}, & a> \delta > 0,\\
				t^{-\frac13}, & |a|<\delta,\\
				t^{-\frac12}, & a < -\delta < 0.	
			\end{cases}
	\end{equation}
	We will focus on the estimate for $\dfrac{\diff}{\diff x} \calT^+(1)$ since the other estimate is similar. 
		
	\medskip
	(i) The case $a:=r^2 > \delta > 0$.
	\medskip
	As before, recalling \eqref{T.1.P} and the associated integration by parts argument, we have
	\begin{align*}
		\frac{\diff}{\diff x} \calT^+(1) 
			&=	\frac{1}{it} 
					\int_k^\infty 
						i(l-k) T^+(k,l) 
							\frac{1}{P(l)} 
						\frac{\diff}{\diff l} 
							\left( e^{itS_0} \right) 
					\, dl\\
			&= -\frac{1}{it}
					\int_k^\infty e^{itS_0}
						\frac{\diff}{\diff l}
							\left(
								\frac{i(l-k) T^+(k,l)}{P(l)} 
							\right)
					\, dl	\\
			&=	-\frac{1}{it}
					\int_k^\infty 
						e^{itS_0} (l-k) T^+(k,l) \frac{P'(l)}{P(l)^2	} 
					\, dl\\
			&\quad		
				-\frac{1}{it}
					\int_k^\infty
						e^{itS_0} \frac{1}{P(l)} 
							\left( 
								T^+(k,l) + (l-k) \frac{\diff T^+}{\diff l}(k,l)
							\right)
					\, dl.
	\end{align*}
	As $P(l)^{-1}$ and $P'(l)/P(l)^2$ are $L^2$ we may use the Cauchy--Schwarz inequality and Minkowski's integral inequality to estimate
	\begin{multline}
		\label{dTdX.1.nocp}	
			\norm[L^2_k]{\frac{\diff}{\diff x}\calT^+(1)}\\
			\lesssim t^{-1} 
					\left(
						\norm[L^2_k L^2_l]{(l-k)T^+} +
						\norm[L^2_k L^2_l]{T^+} +
						\norm[L^2_k L^2_l]{(l-k)\frac{\diff T^+}{\diff l}}
					\right).
	\end{multline}
    \textblue{The first two scattering-data norms on the right-hand side of
\eqref{dTdX.1.nocp} are controlled by Proposition~\ref{prop:Tpm.L2}, while the
last one is controlled by the derivative estimate \eqref{Tpm.est03.bound1}
proved in Proposition~\ref{prop:Tpm.kl.L2}. This proves
\eqref{dTdX.1.est} in the case $a>\delta>0$.}

	\medskip
	(ii) The case $a:=-r^2 < -\delta < 0$.
	\medskip
	
	As in the proof of Lemma \ref{lemma:CT.1}, we introduce cutoff functions $\psi_\pm$ with disjoint supports that localize near critical points of $S_0$. We also set 
	$$ \psi_\infty= 1 - \psi_+ - \psi_- $$
	so that
	\begin{align}
		\label{CT1.x.cp.form}
			\frac{\diff}{\diff x} \calT^+(1)
			&= I_1 + I_2
		\intertext{where}
		\label{CT1.x.I1}
		I_1	
			&=	\int_k^\infty 
					e^{itS_0} i(l-k) 
						\psi_\infty(l) T^+(k,l) \, dl,
		\intertext{and}
		\label{CT1.x.I2}
		I_2		
			&= \int_k^\infty 
					e^{itS_0} i(l-k) 
						\left(\psi_+(l) + \psi_-(l) \right) 
					T^+(k,l) 
				\, dl.
	\end{align}
	Since $\psi_\infty$ is supported away from critical points of $S_0$, we can integrate by parts and obtain
	\begin{align*}
		I_1
			&=	- \frac{1}{t} 
					\int_k^\infty 
						e^{itS_0}
						\frac{1}{P(l)}\left(\psi_\infty(l)T^+(k,l) + (l-k)\frac{\diff T^+}{\diff l}(k,l) \right)
					\, dl \\
			&\quad 
				+  \frac{1}{t}
					\int_k^\infty 
						e^{itS_0}
						\frac{P'(l)}{P(l)^2}
						(l-k)\psi_\infty(l) T^+(k,l)
					\, dl.
	\end{align*}
	Proceeding as in the proof of Lemma \ref{lemma:CT.1} we obtain
	\begin{equation}
		\label{CT1.x.infty.bd}
			\norm[L^2_k]{I_1} 
				\lesssim t^{-1}
					\left(
						\norm[L^2_l]{T^+} +
						\norm[L^2_l]{(l-k)T^+} +
						\norm[L^2_l]{(l-k)\frac{\diff T^+}{\diff l}}
					\right)
	\end{equation}
	with implied constants diverging as $r \downarrow 0$. 
	
	%%%%%%%
	
	Next, we write
	$$ I_2 = I_2^+ + I_2^- $$
	where
	$$ I_2^\pm = \int_k^\infty e^{itS_0} i(l-k) \psi_\pm(l) T^+(k,l) \, dl. $$
	We will estimate $\norm[L^2_k]{I_2^+}$ since the analysis of $\norm[L^2_k]{I_2^-}$ is similar.
	
	First, changing variables from $(k,l)$ to $(k+\eta/12,l+\eta/12)$, we have
	$$ I_2^+ = \int_k^\infty e^{itS} i(l-k) \psi_+(l)\wT^+(k,l) \, dl $$
	where we write $\psi_+(l)$ as shorthand for $\psi_+(l+\eta/12)$ and $S(k,l)$ is given by \eqref{S.def.r}.
	We will use the Fourier multiplication formula \eqref{Fourier.mult} with
	\begin{align*}
		f(l) &= e^{itS}, \\
		g(l) &= (l-k) H(l-k) \psi_+(l) T^+(k,l),
	\end{align*} 
	noting that
	$$ 
		(2\pi)^{-\frac12} 
			\int e^{-i\xi l} e^{itS} \, dl = 
				\frac{e^{-it(12ak + 4k^3)}}{(12t)^\frac13 } (2\pi)^\frac12
					\Ai
						\left(
							(12t)^\frac23 
								\left(a- \frac{\xi}{12t} \right) 
						\right).
	$$
	From the asymptotics (see, for example, \cite[\S 9.7, eq.\ 9.7.9]{NIST:DLMF}
	$$
	\Ai(-x) \sim \frac{1}{\sqrt{\pi} x^{1/4}}\cos\left(\frac23 x^\frac32 - \frac{\pi}{4}  \right) 
		+ \bigO{\frac{1}{x^\frac74} }, \quad x \to \infty
	$$
	we have, for $\xi > -12r^2 t+(12t)^{-\frac13}$, that
	\begin{equation}
		\label{I2+.Airy.asy}
			|\widehat{f}(\xi)| 
		\sim  t^{-\frac12}
				\left(\cos\left(  12t(r^2+\xi/12t)^\frac32 - \frac{\pi}{4}\right)+ \bigO{t^{-\frac76} }\right). 
	\end{equation}
	On the other hand, 
	\begin{align*}
		\widehat{g}(\xi)
			&=	\int_k^\infty  e^{-i\xi l}(l-k) \psi_+(l) T^+(k,l) \, dl
		\intertext{and}
		\xi \widehat{g}&(\xi)\\
			&= i \int_k^\infty \frac{\diff}{\diff l}(e^{-i\xi l}) (l-k) \psi_+(l) T^+(k,l) \, dl\\
			&=-i\int e^{-i\xi l} \left( (\psi_+(l)  + (l-k)\psi_+'(l)) T^+(k,l) + (l-k)\psi_+(l) \frac{\diff T^+}{\diff l}(k,l) \right) \, dl 
	\end{align*}
	so that, by Schwartz inequality and the Plancherel theorem,
	\begin{align}
		\label{I2+.Airy.est1}
		\norm[L^2_\xi]{\widehat{g}}
			&\lesssim \norm[L^2_l]{(\dotarg-k)T^+(k,\dotarg)},
		\intertext{and}
		\label{I2+.Airy.est2}
		\norm[L^2_\xi]{\xi \widehat{g}(\xi)}
			&\lesssim \norm[L^2_l]{T^+(k,\dotarg)} + \norm[L^2_l]{(\dotarg-k)T^+(k,\dotarg)} \\
			&\quad +
				\norm[L^2_l]{(\dotarg-k) \frac{\diff T^+}{\diff l}(k,\dotarg)}.
			\nonumber
	\end{align}
	We now estimate
	\begin{align*}
		\left| \int \widehat{f}(-\xi) \widehat{g}(\xi) \, d\xi \, \right|
			&\lesssim  J_1 + J_2
		\intertext{where}
			J_1 &= \int_{\xi < -6r^2 t}
						(1+\xi^2)^{-\frac12}
						\left[ 
							(1+\xi^2)^\frac12 |\widehat{g}(\xi)|
						\right] 
					\, d\xi ,
		\intertext{and}
			J_2	&= t^{-\frac12} 
					\int_{\xi > -6r^2 t} 
						(1+\xi^2)^{-\frac12} 
						\left[ 
							(1+\xi^2)^\frac12|\widehat{g}(\xi)|
						\right] 
					\, d\xi.
	\end{align*}
	Using \eqref{I2+.Airy.est1} and \eqref{I2+.Airy.est2} together with the Schwartz inequality,  we have
	\begin{align*}
		|J_1|	&\lesssim t^{-1} 
					\left( 
							\norm[L^2_l]{(\dotarg-k)T^+(\dotarg,k)}+ 
							\norm[L^2_l]{(\dotarg-k) \frac{\diff T^+}{\diff l}(k,\dotarg)}
					\right)
	\intertext{and}
		|J_2|	&\lesssim t^{-\frac12}
					\left( 
							\norm[L^2_l]{(\dotarg-k)T^+(\dotarg,k)}+ 
							\norm[L^2_l]{(\dotarg-k) \frac{\diff T^+}{\diff l}(k,\dotarg)}
					\right).
	\end{align*}
	We can now conclude that
	\begin{align}
		\label{CT1.x.pm.bd}
		\norm[L^2_k]{I_2}
			&\lesssim t^{-\frac12}	\
					\left( 
							\norm[L^2_kL^2_l]{(l-k)T^+}+ 
							\norm[L^2_kL^2_l]{(l-k) \frac{\diff T^+}{\diff l}}
					\right).
	\end{align}
 	Combining \eqref{CT1.x.infty.bd} and \eqref{CT1.x.pm.bd}, we conclude that \eqref{dTdX.1.est} holds for $a < -\delta < 0$.
 	
 	\medskip
 	(iii) \textblue{The case $|a|\leq \delta$. 
 	\medskip
 	
 	As before we use the facts that in the region $|a|<\delta$
\begin{align}
	\label{Airy.id.again.bis}
    \frac{1}{\sqrt{2\pi}} \int e^{-i\xi l}e^{it(12al+4l^3)} dl &= \frac{\sqrt{2\pi}}{(12t)^{\frac13}} \Ai\left((12t)^\frac23(a -\xi/12t )\right)\\
    \nonumber
    &=\mathcal{O}(t^{-1/3}).
\end{align}

We let $\psi_0$ be a smooth cutoff function such that 
\begin{align*}
    \psi_0(l)=
    	\begin{cases} 
    		1, & |l|\leq 1,\\
 			0, & |l| \geq 2,
    \end{cases}
\end{align*}
and split
$$\frac{\diff}{\diff x} \calT^+(1) = I_1 + I_2$$
where
\begin{align*}
	I_1	&=	\int_k^\infty e^{itS_0} (l-k) \psi_\infty(l) T^+(k,l) \, dl,\\
	I_2 &=	\int_k^\infty e^{itS_0} (l-k) \psi_0(l) T^+(k,l) \, dl ,
\end{align*}
where now $\psi_\infty = 1- \psi_0$. The rest of the proof will follow from the the proof of the case $a:=r^2<-\delta<0$.
To illustrate this, we notice for $I_1$ since the stationary point is outside of the support of the amplitude function, we directly carry out integration by parts as is done for \eqref{CT1.x.I1}. For $I_2$, \eqref{CT1.x.I2} becomes
\begin{equation}
\label{ct1.x.0}
    \int_k^\infty 
					e^{itS_0} i(l-k) 
						\psi_0(l) T^+(k,l) \, dl
\end{equation}
and we will follow the same argument as that of \eqref{CT1.x.I2} and make use of \eqref{Airy.id.again.bis} to obtain 
\begin{align}
    \Norm{\eqref{ct1.x.0}}{L^2_k}\lesssim t^{-\frac13}	\
					\left( 
							\norm[L^2_kL^2_l]{(l-k)T^+}+ 
							\norm[L^2_kL^2_l]{(l-k) \frac{\diff T^+}{\diff l}}
					\right).
\end{align}
Notice that in this case we do not divide the $\xi$ variable into two separate intervals.}
\medskip
 	
This completes the proof.
\end{proof}

\appendix

\section{Airy-Type Integrals}
\label{App:Airy}
In this section, we prove the estimates \eqref{KPI.g1.est} and \eqref{KPI.loc.u1+.h1.ft.dcp}. Denote
$$H(t;a) = \int_k^\infty e^{-i\xi l}e^{itS(l)} \, dl.$$

\begin{lemma}
	\label{lemma:partial.airy.ndg}
		Suppose that $a \leq -r^2$ for a fixed $r>0$. Then
		\begin{equation}
			\label{partial.airy.ndg.est}
				|H(t;a)| \lesssim_{r} t^{-\frac12}(1+|\xi|).
		\end{equation}
\end{lemma}

\begin{proof}
	First, compute
	\begin{align*}
		H(t;a)
			&=	\int_0^\infty e^{-i\xi (l+k)} e^{itS(l+k)} \, dl
	\end{align*}
	and note that the phase function $S(l+k)$ has critical points at 
	$l_\pm = -k \pm r$. Let $\eta \in C_0^\infty(\R)$ with $\eta(s)=1$ for $|s| < r/4$ and $\eta(s)=0$ for $|s|\geq r/2$. Let
	 $$ \eta_\pm(l) = \eta(l-l_\pm), \quad \eta_\infty = 1 - \eta_+(l)-\eta_-(l).$$
	 It suffices to estimate 
	 $$ I_*(t) = \int_0^\infty \eta_*(l) e^{-i\xi (l+k)} e^{itS(l+k)} \, dl$$
	 for $*=+,-,\infty$, noting that one or both of $I_\pm$  may identically zero depending on the value of $k$.
	 
	 For $I_\infty$, we have 
	 \begin{align*}
	 	I_\infty &= \frac{1}{12it} 
	 		\int_0^\infty 
	 			\frac{\eta_\infty(l) e^{-i\xi(l+k)}} {(l+k)^2-r^2 } 
	 			\diff_l \left(e^{itS(l+k)} \right)  
	 		\, dl\\
	 		&= \frac{1}{12it} \left( -\frac{\eta_\infty(0) e^{itS(k)}e^{-i\xi k}}{k^2-r^2} \right. \\
	 		&\qquad - 
	 			\left. 
	 				\int_0^\infty e^{itS(l+k)} e^{-i\xi k}
	 				\frac{\diff}{\diff l}
	 					\left( 
	 						\frac{\eta_\infty(l)e^{-i\xi l}}{(l+k)^2-r^2}
	 					\right)
	 				\, dl
	 			\right).
	 \end{align*}
	 The first right-hand term in parentheses is bounded because 
	 $$\eta_\infty(0) = 1 - \eta(k-r) - \eta(k+r)$$ 
	 vanishes for $k$ near $\pm r$, and the second term is bounded by constants times $(1+|\xi|)$ for the same reason. Hence,
	 \begin{equation}
	 	\label{partial.airy.ndg.infty}
	 	|I_\infty(t)| \lesssim_{\,r} t^{-1}(1+|\xi|).
	 \end{equation}
	 
	 Next we show how to estimate $I_+$ since the estimates for $I_-$ are similar. We compute
	 \begin{align*}
	 	I_+(t)	
	 		&=	\int_0^\infty \eta_+(l) e^{-i\xi(l+k)}e^{itS(l+k)} \, dl\\
	 		&=	\int_{k-r}^\infty \eta(s) e^{-i\xi(s+r)}e^{itS(s+r)} \, ds
	 \end{align*}
	 where, setting $\phi = S(s+r)$, we have
	 $$ \phi'(s) = 12((s+r)^2-r^2) $$
	 and 
	 $$ \phi''(s)= 24(s+r).$$
	 As $\eta$ has support in $|s|<r/2$, we have the lower bound
	 $$ \phi''(s) \geq 12r > 0$$
	 and we can conclude from Van der Corput's Lemma that
	 \begin{equation}
	 	\label{partial.airy.ndg.+}	
	 		|I_+(t)| \lesssim_{\, r} t^{-\frac12} (1+|\xi|)
	 \end{equation}
	 where the factor of $(1+|\xi|)$ comes from the derivative of $\eta(s)e^{-i\xi(s+r)}$. Similarly
	 \begin{equation}
	 	\label{partial.airy.ndg.-}
	 		|I_+(t)| \lesssim_{\, r} t^{-\frac12} (1+|\xi|).
	 \end{equation}
	 Estimates \eqref{partial.airy.ndg.infty}, \eqref{partial.airy.ndg.+}, and \eqref{partial.airy.ndg.-} imply \eqref{partial.airy.ndg.est}.
\end{proof}

\begin{lemma}
	\label{lemma:partial.airy.dg}
		Suppose that $|a| \leq c$ for a fixed $c>0$. Then
		$$ |H(t;a)| \lesssim t^{-\frac13}, $$
\textblue{with an implicit constant independent of $a$, $c$, $k$, and $\xi$. In particular, this estimate remains uniform as $c \downarrow 0$.}
\end{lemma}

\begin{proof}
	By scaling we have $H(t;a) = (12t)^{-1/3} I(t;a)$. Here
	\begin{equation}
		\label{partial.airy.ndg.mb}
			I(t;a) = \int_m^\infty e^{i(s^3/3 + bs)} \, ds
	\end{equation}
	where
	$$ b = (12t)^\frac23 (a-\xi/12t), \quad m = (12t)^\frac13 k. $$
	It suffices to show that the right-hand side integral in \eqref{partial.airy.ndg.mb} is bounded uniformly in $m,b \in \R$. We split into the three cases $b > 1$, $b < -1$, and $|b| \le 1$.
	
	If $b>1$, a straightforward integration by parts shows that
	$$ \left| \int_m^\infty e^{i(s^3/3+bs)} \, ds \right| \lesssim 1 $$
	uniformly in $m \in \R$ and $b>1$. 
	
	If $b<-1$, let $b=-r^2$ so that $r>1$, fix $q \in (0,1)$, and let $\eta_\pm(s) = \psi(s \mp r)$ where $\psi \in C_0^\infty(\R)$ with $\psi(s)=1$ for $|s|<q/4$ and $\psi(s)=0$ for $|s|\geq q/2$. Let $\eta_\infty = 1- \eta_+ - \eta_-$ and write
	$$ \int_m^\infty e^{i(s^3/3 +bs)} \, ds = I_+ + I_- + I_\infty$$
	where
	$$ I_\pm = \int_m^\infty \eta_\pm(s) e^{i(s^3/3+bs)}\,ds, \qquad I_\infty=\int_m^\infty \eta_\infty(s)e^{i(s^3/3+bs)}\, ds. 
	$$
	Trivially, $|I_\pm| \leq 4q$. On the other hand, using integration by parts, we have
	$$ 
		I_\infty = \left. \frac{\eta_\infty(s)}{i(s^2+b)} \right|_m^\infty -
			\int_m^\infty e^{i(s^3/3+bs)}\frac{d}{ds} \left( \frac{\eta_\infty(s)}{i(s^2+b)}\right) \, ds.
	$$ 
	The first right-hand term is bounded above in modulus by a constant multiple of $(2r-q/2)^{-1}(q/2)^{-1}$ owing to the support of $\eta_\infty$. Since $q$ is fixed and $r>1$, this bound is uniform for $b<-1$. To bound the second right-hand term, note that 
	$$ \frac{d}{ds}\left( \frac{\eta_\infty(s)}{s^2+b}\right)
		= \frac{\eta_\infty'(s)}{s^2+b} - \eta_\infty(s) \frac{2s}{(s^2+b)^2}.
	$$
	Since $q$ is fixed and $r>1$, both terms are absolutely integrable with bounds independent of $r$, hence uniform for $b<-1$. 
	
	Finally, suppose that $|b| \leq 1$ and let $\mu \in C_0^\infty(\R)$ with $0 \leq \mu(s) \leq 1$, $\mu(s)=1$ for $|s|\leq 1$, and $\mu(s)=0$ for $|s|\geq 2$. We may write
	\begin{align*}
		\int_m^\infty e^{i(s^3/3 + bs)} \, ds &= I_1 + I_2
		\intertext{where}
		I_1	&=	\int_m^\infty \mu(s) e^{i(s^3/3+bs)} \, ds,
		\intertext{and}
		I_2	&=	\int_m^\infty (1-\mu(s)) e^{i(s^3/3+bs)} \, ds.
	\end{align*}
	It is easy to see that 
	\begin{equation}
		\label{partial.airy.I1.bd}
			|I_1| \lesssim 1.
	\end{equation}
	We may integrate by parts in  $I_2$, recalling that $1-\mu(s)=0$ for $|s|\leq 1$ and $\mu'(s)$ is supported in $1 \leq |s| \leq 2$:
	\begin{align*}
		I_2	&=	\int_m^\infty e^{ibs} (1-\mu(s)) \frac{1}{is^2} \frac{d}{ds} \left(e^{is^3/3}\right) \, ds\\
			&=	\left. e^{ibs}(1-\mu(s))\frac{1}{is^2} e^{is^3/3} \right|_{m}^{\infty}\\
			&\quad - \int_m^\infty 	e^{is^3/3} \frac{d}{ds}\left( (1-\mu(s))\frac{1}{is^2} e^{ibs}\right) \, ds.
	\end{align*}
	The first term is bounded since $|s|\geq 1$ on the support of $1-\mu$. The second term is bounded uniformly for $|b| \leq 1$, owing to the supports of $1-\mu$ and $\mu'$ and the integrability of $|s|^{-2}$ and $|s|^{-3}$ on $|s| \ge 1$.
\end{proof}

\section{Norm translation}

\label{App:Norm}

The purpose of this lemma is to connect weighted norms appearing in the proofs of different lemmas and propositions to the norm of space $\mathbf{E}_{1,w}$.
\begin{lemma}
    \label{rem:Tpm.kl.L2}
    %The hypotheses of Proposition \ref{prop:Tpm.kl.L2} imply that the following inclusions, used in the proof below, hold.
    Suppose that $u \in \mathbf{E}_{1,w}$.
		\begin{enumerate}[(i)]
			\setlength{\itemsep}{5pt}
			\item \label{inclusion 1} $\tu,\dfrac{\diff \tu}{\diff y}, y \dfrac{\diff \tu}{\diff y} \in L^2_y L^{2,-1}_l(\R^2)$.
            \item \label{inclusion 2} $\tu \in L^{2,5}_yL^{2}_l(\R^2)$ and  $\dfrac{\diff \tu}{\diff y} \in L^{2,2}_y L^{2,2}_l(\R^2)$.
            \item \label{inclusion 3} $\dfrac{\diff \tu}{\diff l} \in L^{2,4}_y L^{2,3}_l(\R^2)$, $\dfrac{\diff^2 \tu}{\diff l^2} \in L^{2,3}_y L^{2}_l(\R^2)$, and 
            $\dfrac{\diff^3 \tu}{\diff l^3} \in L^{2,1}_y L^2_l(\R^2)$.
		\end{enumerate}
        \end{lemma}
        \begin{proof}
       \textblue{It is easy to check that the norms above are controlled by the $\mathbf{E}_{1,w}$ norm. Indeed, by standard Fourier theory, multiplication by the $l$ variable corresponds
to differentiation in  the $x$ variable, while differentiation in the $l$ corresponds to
multiplication by the  $x$ variable. Then it is straightforward to show that 
(ii)-(iii) above follow from the
$\mathbf H^{4,5}$ part in  the definition of $\mathbf E_{1,w}$ in \eqref{IEw.norm}. For instance, by Cauchy-Schwarz inequality, integration by parts and an approximation argument, 
\begin{align*}
    \Norm{\frac{\diff\tu}{\diff l}}{L^{2,4}_y L^{2,3}_l}^2&=\iint_{\R^2} y^8 l^6 \frac{\diff \tu(l,y)}{ \diff l}\frac{\diff \overline{\tu(l,y)}}{ \diff l}dldy\\
    &\lesssim
        \Norm{y^4 l^4 \tu(l,y)}{L^2(\R^2)}
        \Norm{l^2\frac{\partial^2 \tu(l,y)}{\partial l^2}}{L^2(\R^2)}\\
    &\quad +
        \Norm{y^4 l^3 \tu(l,y)}{L^2(\R^2)}
        \Norm{l^2\frac{\partial \tu(l,y)}{\partial l}}{L^2(\R^2)}\\
    &\lesssim 
        \Norm{y^4 l^4 \tu(l,y)}{L^2(\R^2)} \times\\
    &\quad
        \left(
            \Norm{\frac{\partial^4 \tu(l,y)}{\partial l^4}}{L^2(\R^2)}^{1/2}\Norm{l^4 \tu(l,y)}{L^2(\R^2)}^{1/2} 
        \right.\\
    &   \qquad \left. + 
        \Norm{\frac{\partial^3 \tu(l,y)}{\partial l^3}}{L^2(\R^2)}^{1/2}
        \Norm{l^3\tu(l,y)}{L^2(\R^2)}^{1/2} \right.\\
    &   \qquad \left. + 
        \Norm{\frac{\partial^2\tu(l,y)}{\partial l^2}}{L^2(\R^2)}^{1/2}
        \Norm{l^2\tu(l,y)}{L^2(\R^2)}^{1/2}\right)\\
    &\quad +
        \Norm{y^4 l^3 \tu(l,y)}{L^2(\R^2)} \times\\
    &\qquad 
        \left(
            \Norm{\frac{\partial^2 \tu(l,y)}{\partial l^2}}{L^2(\R^2)}^{1/2}
            \Norm{l^4 \tu(l,y)}{L^2(\R^2)}^{1/2}+
            \Norm{\frac{\partial\tu(l,y)}{\partial l}}{L^2(\R^2)}^{1/2}
            \Norm{l^3 \tu(l,y)}{L^2(\R^2)}^{1/2}
        \right)\\
    &\lesssim 
        \Norm{(1+l^4)(1+y^4) \tu(l,y)}{L^2(\R^2)}^{3/2}
        \left(\sum_{i=1}^4
            \Norm{\frac{\partial^i \tu(l,y)}{\partial l^i}}{L^2(\R^2)}^{1/2}
        \right)\\
    &\lesssim 
        \Norm{(1+y^4)(1+\partial_x^4) u(x,y)}{L^2(\R^2)}^{3/2}
        \left(
            \sum_{i=1}^4 \Norm{{x^i u(x,y)}}{L^2(\R^2)}^{1/2}
        \right)
\end{align*}
and
\begin{align*}
    \Norm{\frac{\diff\tu}{\diff y}}{L^{2,2}_y L^{2,2}_l}^2&=\iint_{\R^2} y^4 l^4\frac{\diff \tu(l,y)}{ \diff y}\frac{\diff \overline{\tu(l,y)}}{ \diff y}dydl\\
    &\lesssim\Norm{y^4 \frac{\partial^4 u(x,y)}{\partial x^4}}{L^2(\R^2)}\Norm{\frac{\partial^2 u(x,y)}{\partial y^2}}{L^2(\R^2)}
    \\
    &\quad +\Norm{y^3 \frac{\partial^4 u(x,y)}{\partial x^4}}{L^2(\R^2)}\Norm{\frac{\partial u(x,y)}{\partial y}}{L^2(\R^2)}.
\end{align*}
The singular low-frequency bounds in (i) can be obtained by using the
$\partial_x^{-1}$ terms in the definition of $\mathbf E_{1,w}$. Indeed,
\[
    \partial_x^{-1}u\in L^2
    \quad\Longleftrightarrow\quad
    |l|^{-1}\tu\in L^2_{l,y},
\]
and
\[
    \partial_x^{-1}\partial_y u\in L^2_xL^{2,1}_y
    \quad\Longleftrightarrow\quad
    |l|^{-1}\frac{\diff\tu}{\diff y}\in L^{2,1}_yL^2_l.
\]
These estimates imply
$\tu,\diff_y\tu,y\diff_y\tu\in L^2_yL^{2,-1}_l$, using the weighted $L^2$ bounds away from $l=0$.

The bound $\tu\in L^{2,1}_yL^{2,-1}_l$ follows from Cauchy--Schwarz inequality:
\[
    \norm[L^{2,1}_yL^{2,-1}_l]{\tu}^2
    \leq
    \norm[L^2]{|l|^{-1}\tu}\,
    \norm[L^{2,2}_yL^2_l]{\tu}.
\]
The first factor is controlled by $\Norm{\partial_x^{-1}u}{ L^2}$, and the second by $\mathbf{H}^{4,5}(\mathbb{R}^2)$.
}
\end{proof}
\begin{remark}
\label{rem:Tpm.kl.L2'}
   Given inclusion relations \eqref{inclusion 1}-\eqref{inclusion 3} hold, together with smallness assumptions \eqref{u.C1}, \eqref{u.C2} and \eqref{u.C2a}, we arrive at the following conclusions:
   \begin{itemize}
   \item Lemma \ref{lemma:tmu.exist}, Special cases of Lemma \ref{lemma:tmu.smooth} given by Remark \ref{rem:lemma tmu smooth} and Lemma \ref{lemma:convolve}-Lemma \ref{lemma:tmu y} and Proposition \ref{prop: mu-k} all hold.
   \item All the estimates involving $\tmu_\sharp^\pm$ in Section \ref{sec:mupm} and Section \ref{sec:Tpm} hold.
       \item The estimates involving $\tu$  needed for  \eqref{Tpm.est03.bd1.1.norms}, \eqref{Tpm.est03.bd1.needs}, \eqref{Tpm.est03.bd2.1.norms}, \eqref{Tpm.est03.bd2.2.norms}, \eqref{TP.est04.T1.bd} and  \eqref{TP.est04.T2.bd} hold. Thus Proposition \ref{prop:Tpm.kl.L2} holds.
       \item It can be verified that the inclusion relations \eqref{inclusion 1}-\eqref{inclusion 3} are sufficient to obtain all estimates of $\tu$  needed in Section \ref{sec:asy}. 
   \end{itemize}
\end{remark}

\bibliographystyle{amsplain}
\bibliography{KPI}

\end{document}